\documentclass[12pt]{amsart}
\usepackage{amscd,amssymb,amsthm,amsmath,amssymb,mathrsfs,enumerate}
\usepackage[matrix,arrow,curve]{xy}
\usepackage[margin=2cm]{geometry}
\usepackage{tikz,tkz-euclide}
\usepackage{booktabs,longtable}
\usepackage{multirow}
\usepackage{hyperref}
\usepackage{mathtools}

\usepackage[utf8]{inputenc}
\usepackage[T1]{fontenc}

\usepackage{xcolor}
\usepackage
{inconsolata}

\usepackage{listings}
\usepackage{appendix}

\lstdefinelanguage{Magma}{
  morekeywords={
    function,return,for,in,to,do,end,if,then,else,elif,while,
    repeat,until,break,continue,and,or,not,true,false,local,load
  },
  sensitive=true,
  morecomment=[l]{//},
  morecomment=[s]{/*}{*/},
  morestring=[b]",
}

\newcommand{\sslash}{\mathbin{/\mkern-6mu/}}

\newtheorem{theorem}{Theorem}[subsection]

\newtheorem{proposition}[theorem]{Proposition}
\newtheorem{lemma}[theorem]{Lemma}
\newtheorem*{theorem*}{Theorem}
\newtheorem*{corollary*}{Corollary}
\newtheorem*{example*}{Example}
\newtheorem*{maintheorem*}{Main Theorem}
\newtheorem*{corollaryA*}{Corollary A}
\newtheorem*{corollaryB*}{Corollary B}
\newtheorem*{corollaryC*}{Corollary C}
\newtheorem*{corollaryD*}{Corollary D}
\newtheorem*{corollaryE*}{Corollary E}
\newtheorem*{corollaryF*}{Corollary F}
\newtheorem{corollary}[theorem]{Corollary}
\newtheorem{conjecture}[theorem]{Conjecture}

\newtheorem*{proposition*}{Proposition}

\theoremstyle{definition}
\newtheorem{example}[theorem]{Example}
\newtheorem{definition}[theorem]{Definition}

\theoremstyle{remark}
\newtheorem{remark}[theorem]{Remark}

\newtheorem*{remark*}{Remark}

\makeatletter\@addtoreset{equation}{section} \makeatother

\newcommand{\vol}{\operatorname{vol}}
\newcommand{\ord}{\operatorname{ord}}
\newcommand{\vvol}[2]{\ensuremath{\mathrm{vol}_{#1}(#2)}}
\newcommand{\deq}{\ensuremath{\stackrel{\mathrm{def}}{=}}}
\newcommand{\equ}{\ensuremath{\,=\,}}

\newcommand{\N}{\ensuremath{\mathbb{N}}}

\newcommand{\Q}{\ensuremath{\mathbb{Q}}}
\newcommand{\R}{\ensuremath{\mathbb{R}}}
\newcommand{\PP}{\ensuremath{\mathbb{P}}}

\newcommand{\bH}{\ensuremath{\mathbb{H}}}
\newcommand{\bQ}{\ensuremath{\mathbb{Q}}}
\newcommand{\bP}{\ensuremath{\mathbb{P}}}

\newcommand{\bG}{\ensuremath{\mathbb{G}}}
\newcommand{\mtc}[1]{\mathcal{#1}}

\newcommand{\mts}[1]{\mathscr{#1}}

\DeclareMathOperator{\PGL}{PGL}
\DeclareMathOperator{\Amp}{Amp}

\DeclareMathOperator{\Proj}{Proj}

\DeclareMathOperator{\Supp}{Supp}

\DeclareMathOperator{\Pic}{Pic}

\DeclareMathOperator{\Div}{Div}

\DeclareMathOperator{\GIT}{GIT}

\DeclareMathOperator{\Bs}{Bs}

\DeclareMathOperator{\sst}{ss}

\DeclareMathOperator{\CM}{CM}
\DeclareMathOperator{\Bl}{Bl}

\newcommand{\cO}{\ensuremath{\mathscr{O}}}
\newcommand{\cI}{\ensuremath{\mathscr{I}}}

\DeclareMathOperator{\Bbig}{Big}

\author{Ivan Cheltsov, DongSeon Hwang, Alex K\"uronya, Antonio Laface, \\ Fr\'ed\'eric Mangolte, Alex Massarenti, Jihun Park, Junyan Zhao}

\thanks{Throughout this paper, all varieties are assumed to be projective and defined over~$\mathbb{C}$.}

\title{On K-stability of Fano's last Fanos}

\let\origmaketitle\maketitle
\def\maketitle{
  \begingroup
  \def\uppercasenonmath##1{} 
  \let\MakeUppercase\relax 
  \origmaketitle
  \endgroup
}

\begin{document}

\begin{abstract}
We study K-stability of smooth Fano threefolds of Picard rank $2$ and degree $22$ which can be obtained by blowing up
a smooth complete intersection of two quadrics in $\mathbb{P}^5$ along a conic. We also describe the automorphism groups of these threefolds.
\end{abstract}



\address{ \emph{Ivan Cheltsov}\newline
\textnormal{University of Edinburgh, Edinburgh, Scotland
\newline
\texttt{i.cheltsov@ed.ac.uk}}}

\address{ \emph{DongSeon Hwang}\newline
\textnormal{Center for Complex Geometry, Institute for Basic Science, Daejeon, Korea
\newline
\texttt{dshwang@ibs.re.kr}}}

\address{ \emph{Alex K\"uronya}\newline
\textnormal{Goethe University Frankfurt, Frankfurt, Germany
\newline
\texttt{kuronya@math.uni-frankfurt.de}}}

\address{\emph{Antonio Laface}\newline
\textnormal{Universidad de Concepci\'on, Concepci\'on, Chile
\newline
\texttt{alaface@udec.cl}}}

\address{\emph{Fr\'ed\'eric Mangolte}\newline
\textnormal{Aix Marseille University, CNRS, I2M, Marseille, France
\newline
\texttt{frederic.mangolte@univ-amu.fr}}}

\address{\emph{Alex Massarenti}\newline
\textnormal{University of Ferrara, Ferrara, Italy
\newline
\texttt{msslxa@unife.it}}}

\address{ \emph{Jihun Park}\newline
\textnormal{Institute for Basic Science, Pohang, Korea \newline POSTECH, Pohang, Korea \newline
\texttt{wlog@postech.ac.kr}}}

\address{ \emph{Junyan Zhao}\newline
\textnormal{University of Maryland, College Park, Maryland
\newline
\texttt{jzhao81@umd.edu}}}

\maketitle

\tableofcontents

\newpage

\section{Introduction}
\label{section:intro}

In his last published paper \cite{Fano}, Fano constructed smooth 3-folds of degree $22$ in $\mathbb{P}^{13}$ whose twice hyperplane sections are canonically embedded curves of genus $12$. In modern language, these 3-folds are smooth Fano 3-folds of degree $22$. For very long time, Fano's last paper has been ignored by mathematicians until Andreatta and Pignatelli revised Fano's construction in modern language \cite{AndPig}. In particular, they showed that Fano's last Fanos can be obtained as divisors in the linear system $|H|$ on the smooth Fano 4-fold $Y$ with $-K_Y\sim 2H$, where $Y$ is a blow up of the smooth quadric 4-fold in a conic \cite{Mukai1989,Wisniewski1990}. Alternatively, these Fano 3-folds can be constructed as the blow-up of a smooth complete intersection of two quadrics in $\mathbb{P}^5$ along a conic. In the notations of \cite{fanography}, these are smooth Fano 3-folds in Family \textnumero 2.16.

The Fano 4-fold $Y$ is known to be K-polystable \cite{Delcroix2022}. So, a natural question is whether all Fano's last Fanos are K-polystable or not. This question has been already addressed in the book \cite{Book}, where it has been conjectured that all Fano's last Fano's are K-stable, and the following result has been proven.

\begin{theorem*}[{\cite[Section~5.6]{Book}}]
A general smooth Fano 3-fold in Family \textnumero 2.16 is K-stable.
\end{theorem*}

To be precise, \cite[Section~5.6]{Book} presents the following example of a K-stable Fano's last Fano.

\begin{example*}[{\cite[Section~5.6]{Book}}]
Let $Q_1$ be the~smooth quadric $\{u_1u_4+u_2u_5+u_3u_6=0\}\subset\mathbb{P}^5$, and let $Q_2$ be the~smooth quadric
$$
\{u_1^2+\omega u_2^2+\omega^2u_3^2+u_4^2+\omega u_5^2+\omega^2u_6^2+u_1u_4+\omega u_2u_5+\omega^2u_3u_6=0\}\subset\mathbb{P}^5,
$$
where $\omega$ is a~primitive cube root of unity, and $u_1,\ldots,u_6$ are coordinates on $\mathbb{P}^5$. Let $V=Q_1\cap Q_2$. Then $V$ is smooth.
Let $C_2$ be the smooth conic $V\cap\{u_1=u_2=u_3=0\}$,
let $f\colon X\to V$ be the~blow up of this conic.
Then $X$ is a~smooth Fano 3-fold in Family \textnumero 2.16,
the group $\mathrm{Aut}(X)$ contains a subgroup isomorphic to $\mathfrak{A}_4$, and $X$ is K-stable.
\end{example*}

Since K-stability is an open property (see Section~\ref{section:openness} for more details), it follows from this example that a general Fano's last Fano is K-stable. However, the notion of ``generality'' here is non-effective: the argument only guarantees the existence of a Zariski open subset in the moduli space parametrizing K-stable members, without providing an explicit criterion to determine whether a given Fano's last Fano lies in this locus. In particular, the result cannot be directly applied to verify K-stability for a specific example. The goal of this paper is to remove this gap.

Namely, in this paper, we prove an explicit (a bit technical) generalization of the theorem above, which can be applied to the vast majority of Fano's last Fanos.
To state our main result, let $X$ be a smooth Fano 3-fold in Family \textnumero 2.16 that is defined over a subfield $\Bbbk\subset\mathbb{C}$,
let $X_\mathbb{C}$ be its geometric model, and let $G$ be a (finite) subgroup in $\mathrm{Aut}(X)$.
Then we have the~following $G$-equivariant Sarkisov link:
$$
\xymatrix{
&X\ar[ld]_{f}\ar[rd]^{g}&\\
V&&\mathbb{P}^2}
$$
where $V$ is a smooth complete intersection of two quadrics $Q_1$ and $Q_2$ in $\mathbb{P}^5$,
$f$ is a blow up of a smooth conic $C_2\subset V$, and $g$ is a conic bundle.
Let $\Delta$ be the~discriminant curve of the~conic bundle $g$.
Then $\Delta$ is a reduced (possibly reducible) quartic curve in $\mathbb{P}^2$, which is defined over $\Bbbk$.

\begin{maintheorem*}
\label{theorem:main}
Suppose that $G$ does not fix $\Bbbk$-points in $\mathrm{Sing}(\Delta)$. Then $X_{\mathbb{C}}$ is K-stable.
\end{maintheorem*}

\begin{corollaryA*}
Suppose that $\mathrm{Sing}(\Delta)$ does not have $\Bbbk$-points. Then $X_{\mathbb{C}}$ is K-stable.
\end{corollaryA*}

\begin{corollaryB*}
Suppose that $\Delta$ is smooth. Then $X_{\mathbb{C}}$ is K-stable.
\end{corollaryB*}

\begin{corollaryC*}[{\cite[Lemma~5.7]{AbbanCheltsovKishimotoMangolte1}}]
Suppose that $X(\Bbbk)=\varnothing$. Then $X_{\mathbb{C}}$ is K-stable.
\end{corollaryC*}

\begin{proof}
Suppose that $X_{\mathbb{C}}$ is not K-stable.
Then, by Main Theorem, $\Delta$ is singular, and $\mathrm{Sing}(\Delta)$ contains a $\Bbbk$-point $p$.
Let $Z$ be the~fiber of the~conic bundle $g$ over $p$ with reduced structure. Then $Z$ is defined over $\Bbbk$,
and $f(Z)$ is a line in $V$ that intersects  the conic $C_2$ at one point, which must be defined over $\Bbbk$.
In particular, we see that $V(\Bbbk)\ne\varnothing$, so $X(\Bbbk)\ne\varnothing$ by the~Lang--Nishimura theorem,
which contradicts our assumption.
\end{proof}

\begin{corollaryD*}
Suppose that $G$ does not fix $\Bbbk$-points in $\Delta$. Then $X_{\mathbb{C}}$ is K-stable.
\end{corollaryD*}

\begin{corollaryE*}
Suppose that $G$ does not fix $\Bbbk$-points in $X$. Then $X_{\mathbb{C}}$ is K-stable.
\end{corollaryE*}

\begin{proof}
Suppose that $X_{\mathbb{C}}$ is not K-stable.
Then, by Main Theorem, $\Delta$ is singular, and $\mathrm{Sing}(\Delta)$ contains a $\Bbbk$-point $p$ that is fixed by $G$.
Let $Z$ be the~fiber of the~conic bundle $g$ over $p$ with reduced structure, and let $E$ be the $f$-exceptional surface.
Then $Z$ and $E$ are defined over $\Bbbk$, both $Z$ and $E$ are $G$-invariant, and $f(Z)$ is a line in $V$ that intersects  the conic $C_2$ at one point.
In particular, we see that $Z\not\subset E$. On the other hand, we have $Z\cdot E=1$, so the intersection $Z\cap E$ consists of a single point,
which must be $G$-invariant and defined over $\Bbbk$. This contradicts our assumption.
\end{proof}

\begin{corollaryF*}[{\cite[Lemma~4.5]{AbbanCheltsovKishimotoMangolte2}}]
\label{corollary:main-6}
If $G$ is abelian, $\Bbbk=\mathbb{C}$, $G$ fixes no points in $X$, then $X$ is K-stable.
\end{corollaryF*}

Let us show how to compute the equation of the curve $\Delta$. Let $\Pi$ be the plane in $\mathbb{P}^5$ that contains the conic $C_2$.
Replacing $Q_1$ by another quadric in the pencil generated by $Q_1$~and~$Q_2$ if necessary, we can additionally assume that $Q_1$ contains $\Pi$, but $Q_2$ does not contain $\Pi$.
Then $Q_1$ is given by
$$
\alpha_0(u_1,u_2,u_3)+\alpha_1(u_1,u_2,u_3)u_4+\alpha_2(u_1,u_2,u_3)u_5+\alpha_3(u_1,u_2,u_3)u_6=0,
$$
where $\alpha_0$ is a quadratic form in $\Bbbk[u_1,u_2,u_3]$,
and $\alpha_1$, $\alpha_2$, $\alpha_3$ are some linear forms in $\Bbbk[u_1,u_2,u_3]$.
Moreover, up to a change of coordinates on $\mathbb{P}^5$, we may assume that
$$
C_2=\big\{u_1=u_2=u_3=u_4^2-u_5u_6=0\big\}.
$$
Then $Q_2$ is given by
$$
\beta_0(u_1,u_2,u_3)+\beta_1(u_1,u_2,u_3)u_4+\beta_2(u_1,u_2,u_3)u_5+\beta_3(u_1,u_2,u_3)u_6+u_4^2-u_5u_6=0,
$$
where $\beta_0$ is a quadratic form in $\Bbbk[u_1,u_2,u_3]$,
and $\beta_1$, $\beta_2$, $\beta_3$ are some linear forms in $\Bbbk[u_1,u_2,u_3]$.
Now, we let $\pi\colon \widetilde{\mathbb{P}}^5\to \mathbb{P}^5$ be the blow up of the plane $\Pi$,
let $\widetilde{Q}_1$ and $\widetilde{Q}_2$ be the strict transforms of the quadrics $Q_1$ and $Q_2$ on $\widetilde{\mathbb {P}}^5$, respectively.
Then we have the following commutative diagram:
$$
\xymatrix{
&\widetilde{\mathbb{P}}^5\ar[ld]_{\pi}\ar[rd]^{\eta}&\\
\mathbb{P}^5\ar@{-->}[rr]_{\chi}&&\mathbb{P}^2}
$$
where $\eta$ is a $\mathbb{P}^3$-bundle, and $\chi$ is a projection from $\Pi$.
Identify $X=\widetilde{Q}_1\cap\widetilde{Q}_2$, $\pi\vert_{X}=f$, $\eta\vert_{X}=g$.
The Cox ring of $\widetilde{\mathbb{P}}^5$ together with  its grading matrix and irrelevant ideal are
$$
 \mathcal{R}\big(\widetilde{\mathbb{P}}^5\big)=\mathbb{C}[x_1,\dots,x_7],
 \qquad
 \begin{bmatrix*}[r]
 1&1&1&1&1&1&0\\
 -1&-1&-1&0&0&0&1
 \end{bmatrix*},
 \qquad
 \mathscr I_{\rm irr}(\widetilde{\mathbb P}^5)
 =(x_1,x_2,x_3)\cap (x_4,x_5,x_6,x_7).
$$
See Section~\ref{section:ZD} for details.
In these coordinates, the blow up $\pi$ is given by
$$
\pi(x_1,\dots,x_7)=[x_1x_7:x_2x_7:x_3x_7:x_4:x_5:x_6]
$$
and $\eta$ is given by $\eta(x_1,\dots,x_7)=[x_1:x_2:x_3]$.
This implies that $\widetilde{Q}_1$ is given by
\begin{equation}
\label{eq1}
\alpha_0(x_1,x_2,x_3)x_7+\alpha_1(x_1,x_2,x_3)x_4+\alpha_2(x_1,x_2,x_3)x_5+\alpha_3(x_1,x_2,x_3)x_6=0;
\end{equation}
and $\widetilde{Q}_2$ is given by
\begin{equation}\label{eq2}
\beta_0(x_1,x_2,x_3)x_7^2+(\beta_1(x_1,x_2,x_3)x_4+\beta_2(x_1,x_2,x_3)x_5+\beta_3(x_1,x_2,x_3)x_6)x_7+x_4^2-x_5x_6=0.
\end{equation}
Moreover, in the open subset $\{\alpha_1(x_1,x_2,x_3)\neq 0\}$, we have
$$
x_4=-\frac{\alpha_0(x_1,x_2,x_3)x_7 + \alpha_2(x_1,x_2,x_3)x_5+\alpha_3(x_1,x_2,x_3)x_6}{\alpha_1(x_1,x_2,x_3)},
$$
so plugging this equality in the defining equation of $\widetilde{Q}_2$, we get a quadratic form in the variables $x_5,x_6,x_7$ whose matrix is
$$
\left(\begin{array}{ccc}
2\alpha_2^2 & 2\alpha_2\alpha_3-1  & \beta_2-\alpha_2\beta_1+2\alpha_0\alpha_2 \\
2\alpha_2\alpha_3-1 & 2\alpha_3^2 & \beta_3 - \alpha_3\beta_1 + 2\alpha_0\alpha_3 \\
\beta_2-\alpha_2\beta_1+2\alpha_0\alpha_2 & \beta_3 -\alpha_3\beta_1 + 2\alpha_0\alpha_3 & 2(\beta_0 + \alpha_0^2-\alpha_0\beta_1)
\end{array}\right).
$$
Computing the determinant of this matrix, we see that
$$
\Delta=\Big\{\Big(\alpha_0 - \frac{1}{2}\alpha_1\beta_1 + \alpha_2\beta_3 + \alpha_3\beta_2\Big)^2+\Big(\alpha_1^2 - 4\alpha_2\alpha_3\Big)\Big(\beta_0-\frac{1}{4}\beta_1^2+\beta_2\beta_3\Big)=0\Big\}\subset\mathbb{P}^2_{x_1,x_2,x_3}.
$$
This quartic curve is singular if and only if $g$ admits non-reduced fibers.

\begin{remark*}
Let $E$ be the~$f$-exceptional surface. Then $E$ is cut out on $X$ by $x_7=0$, so that
$$
E=\big\{\alpha_1(x_1,x_2,x_3)x_4+\alpha_2(x_1,x_2,x_3)x_5+\alpha_3(x_1,x_2,x_3)x_6=x_4^2-x_5x_6=x_7=0\big\}\subset\widetilde{\mathbb{P}}^5.
$$
When the linear forms $\alpha_1,\alpha_2, \alpha_3$ are linearly independent, the divisor $E$ is isomorphic to $\mathbb P^1\times\mathbb P^1$.
On the other hand, if $\alpha_1,\alpha_2,\alpha_3$ are linearly dependent, then $E$ is the second Hirzebruch surface~$\mathbb{F}_2$,
whose negative curve is the smooth fiber of the conic bundle $g$, and it has equations
$$
\alpha_1(x_1,x_2,x_3)=\alpha_2(x_1,x_2,x_3)=\alpha_3(x_1,x_2,x_3)=x_4^2-x_5x_6=x_7=0.
$$
\end{remark*}

Now, we can easily apply Corollaries~A and B to construct K-stable Fano's last Fanos.
However, to apply Corollaries~D and E, we have to know more about symmetry groups of Fano's last~Fanos.
In Section~\ref{section:Aut}, we will describe all these groups.
Namely, we will show that all possible non-trivial automorphism groups of complex Fano's last Fanos are the following $17$ finite groups:
\begin{center}
$C_2$, $C_3$, $C_2^2$, $C_4$, $C_5$, $C_6$, $C_8$, $C_2\times C_4$, $C_2^3$, $\mathrm{D}_4$, $C_{10}$, \\$\mathfrak{A}_4$, $C_2^2\times C_{4}$, $C_2^2 \rtimes C_4$, $C_2\times C_8$, $C_2\times\mathfrak{A}_4$, $C_2^2\rtimes C_8$,
\end{center}
where $C_m$ is the cyclic group of order $m$, $\mathrm{D}_4$ is the dihedral group of order $8$, and $\mathfrak{A}_4$ is the alternating group of degree $4$.
See Section~\ref{section:Aut} for more details.

\medskip
\noindent
\textbf{Postscript remarks.} After the first version of this article was posted, Conjecture~\ref{conj:K-ss degeneration} was verified in \cite[Appendix A]{KLPZ26}. This provides strong evidence toward a complete description of the K-moduli stack.

\medskip
\noindent
\textbf{Acknowledgements.}
Ivan Cheltsov was supported by Simons Collaboration grant \emph{Moduli of varieties}. DongSeon Hwang was supported by the National Research Foundation of Korea(NRF) grant funded by the Korea government(MSIT) (2021R1A2C1093787) and the Institute for Basic Science (IBS-R032-D1).  Alex K\"uronya acknowledges support by the Deutsche  Forschungsgemeinschaft  (DFG) via  the Collaborative Research Centre TRR 326 "Geometry and Arithmetic of Uniformized Structures", project number 444845124, and  by the project ANR-23-CE40-0026 ``Positivity on K-trivial varieties''.
Antonio Laface was partially supported by Proyecto FONDECYT Regular n. 1230287. Jihun Park was supported by IBS-R003-D1 from the Institute for Basic Science in Korea.
This work started during the winter school \emph{K-stability of Fano varieties
and birational invariants} held in Trento, has been developed in CIRM during of a semester-long \emph{Morlet Chair}, was continued in G\"okova during \emph{Thirtieth G\"okova Geometry / Topology Conference} and finished in Kracow during the conference \emph{Structures and Symmetries in Algebraic Geometry}.

\section{K-stability}
\label{section:K-stability}

K-polystability is a stability condition which characterizes the existence of a certain type of extremal K\"ahler metric whose K\"ahler form belongs to the anticanonical class:
a K\"ahler-Einstein metric. Moreover, it allows the construction of moduli spaces for Fano varieties that are projective in some cases where GIT fails. It is worth mentioning that K-stability is an algebro-geometric condition while the K\"ahler-Einstein equation is a highly nonlinear PDE.

To check that a variety is K-stable consists of considering all $\mathbb{G}_m$-equivariant degenerations over $\mathbb{A}^1$ and checking that the Donaldson-Futaki invariant (see Definition~\ref{definition:DF}) is positive, unless the degeneration is trivial (the precise notion of trivial leads to the many different flavours of K-stability). There is a vast body of complex geometry on this subject. We avoid detailing this aspect. One can consult \cite{Gabor18} for a survey. Likewise, many intertwined works have contributed to the results we discuss.  We make no attempt to cite all of them properly and refer to \cite{Xu21,Xubook} for appropriate references.

\subsection{K\"ahler-Einstein manifolds}
Let $X$ be an $n$-dimensional  compact complex manifold, and let $J$ be the endomorphism of the tangent bundle given by multiplication by $i$ on each tangent space. A real $(1,1)$-form $\omega$ on $X$ defines a Hermitian metric on $X$ if the symmetric bilinear form
$$
(v,u)\ \mapsto \ \omega\big(v, J(u)\big)
$$
defined for real tangent vectors $v,u$ is positive definite. A \emph{K\"ahler} metric on $X$ is given by a real $(1,1)$-form $\omega$ which can be given in a local coordinate system $(z_\alpha)_{1\leqslant \alpha\leqslant n}$ by
$$
\omega\ =\ \sqrt{-1}\sum_{1\leqslant \alpha,\beta\leqslant n}g_{\alpha\beta}dz_{\alpha}\wedge d{\bar z}_{\beta}
$$
where $(g_{\alpha\beta})_{1\leqslant \alpha,\beta\leqslant n}$ is a positive definite Hermitian matrix with $C^\infty$ entries and $\omega$ is a closed form.

\begin{example}
On $\mathbb{P}^n$ with homogeneous coordinates $[u_0:\ldots:u_n]$, the \emph{Fubini-Study} metric $\omega_{\textup{FS}}$ is the $U_{n+1}(\mathbb{C})$-invariant Hermitian metric defined locally near the point $[1:0\ldots:0]$ by
$$
\omega_{\textup{FS}}\ =\ \sqrt{-1}\partial\bar\partial\log\left(1+\sum_{1\leqslant \alpha\leqslant n}\vert z_{\alpha}\vert^2\right),
$$
where $z_\alpha=\frac{u_\alpha}{u_0}$. In particular, its value at $[1:0\dots:0]$ is $\sqrt{-1}\sum_{1\leqslant \alpha \leqslant n}dz_{\alpha}\wedge d{\bar z}_{\alpha}$ so it is positive definite and clearly the form $\omega_{FS}$ is closed, and thus the Fubini-Study metric is K\"ahler.
\end{example}

\noindent The \emph{Ricci curvature} of  K\"ahler metric $\omega$ is the  $(1,1)$-form given in local coordinates by
$$
\operatorname{Ric}(\omega)\ =\ -\sqrt{-1}\partial\bar\partial\log\det(g_{\alpha \beta})\;.
$$
We say that $\omega$ is \emph{K\"ahler-Einstein} if
\begin{equation}
\label{eq:ke}
 \operatorname{Ric}(\omega)\ =\ \lambda\omega
\end{equation}
for some constant $\lambda\in \mathbb{R}$.

\begin{example}
\label{example:FS}
The Fubini-Study metric on the complex projective space $\mathbb{P}^n$ is an example of K\"ahler-Einstein metric with $\lambda=2n+1$ in equality~\eqref{eq:ke}
\end{example}

\begin{remark}
\label{remark:cscK}\textup{
The \emph{scalar curvature function} of $\omega$ is }
$$
S(\omega)\ =\ \operatorname{Tr}_{\omega}\big(\operatorname{Ric}(\omega)\big)\ = \ n\cdot \frac{\operatorname{Ric}(\omega)\wedge\omega^{n-1}}{\omega^n}.
$$
The K\"ahler metric $\omega$ is said to be a \emph{constant scalar curvature metric} (or \emph{cscK metric}) if $S(\omega)$ is constant.
If $\omega$ is a K\"ahler-Einstein metric, then $S(\omega)=n\lambda$, so $\omega$ is a cscK metric.
\end{remark}

Note that $\omega^n$ is a volume form and that $\operatorname{Ric}(\omega)$ only depends on this volume form.
In fact, the form $\operatorname{Ric}(\omega)$  is, up to the multiplication by a constant, the curvature form of the Hermitian metric given by $\omega^n$ on the canonical line bundle $K_X=\bigwedge^n T^*_X$. Therefore, the form $\operatorname{Ric}(\omega)$ defines a cohomology class which satisfies the relation
$$
c_1(X)\ =\ \frac1{2\pi}\cdot [\operatorname{Ric}(\omega)]\ \in\  H^{1,1}\big(X,\mathbb{R}\big)
$$
where $c_1(X)$ is the first Chern class of the tangent bundle.
A necessary condition then arises to solve equation~\eqref{eq:ke}: that the first Chern class of $X$ is zero or contains a positive (resp. negative) definite form according to $\lambda=0$, $\lambda>0$ (resp. $\lambda<0$).
By Kodaira embedding theorem, we get that if $X$ is K\"ahler-Einstein, then one of the following three cases holds:
\begin{itemize}
\item $\lambda=0$ and $K_X\sim_{\mathbb{Q}} 0$;
\item $\lambda<0$ and $K_X$ is ample, so $X$ is projective;
\item $\lambda>0$ and $-K_X$ is ample, so $X$ is projective, and $X$ is a Fano manifold.
\end{itemize}
In the $\lambda\leqslant 0$ case, equivalent conditions of existence of K\"ahler-Einstein metrics on a compact variety where elucidated by Aubin \cite{Au78} and Yau \cite{Ya78} (the $\lambda=0$ case being a consequence of the proof of Calabi Conjecture see \cite{Bo97} for details). These two results are generalized for singular complex varieties in \cite{EGZ09}. In summary, we have: if $K_X\sim_{\mathbb{Q}} 0$ or $K_X$ is ample, then $X$ is always K\"ahler-Einstein. However, if $-K_X$ is ample, then $X$ is not always K\"ahler-Einstein.
This follows from the obstruction for a compact K\"ahler manifold to have a constant scalar curvature metric, for example a K\"ahler-Einstein one, which is due to Matsushima:

\begin{theorem}[\cite{Ma57}]
If $X$ is a K\"ahler-Einstein Fano manifold, then $\mathrm{Aut}(X)$ is reductive.
\end{theorem}

\begin{example}
\label{example:F1-dP7}
If $X$ is obtained by blowing up $\mathbb{P}^2$ in one or two distinct points, then $-K_X$ is ample, so $X$ is a two-dimensional Fano manifold, but $\mathrm{Aut}(X)$ is not reductive, so $X$ is not K\"ahler-Einstein.
\end{example}

The Yau-Tian-Donaldson Conjecture, which is now proved, asserts that the existence of a K\"ahler-Einstein metric on a Fano manifold $X$ is equivalent to $X$ being K-polystable, see Definition~\ref{definition:K-stability} and Theorem~\ref{theorem:CDS} below.

\subsection{K-stability via Donaldson--Futaki invariant}
\label{section:DF-invariant}
Let $X$ be a $\bQ$-Fano variety, i.e. a normal projective variety whose anticanonical divisor $-K_X$ is $\mathbb{Q}$-Cartier and ample. Here, we do not need to assume any restriction on the singularities of $X$ except that the divisor $-K_X$ is assumed to be $\mathbb{Q}$-Cartier in order to define its ampleness. However, for simplicity, let us assume that $X$ has at most Kawamata log terminal (klt) singularities.
For standard terminology of higher dimensional geometry, see e.g. \cite{KM98,Ko13,Ko-modbook}.
Our assumption that $X$ has klt singularities is not restrictive for our purposes by Theorem~\ref{theorem:oda} below.

\begin{definition}[Donaldson--Futaki invariant]
\label{definition:DF}
Notation as above and set $L=-K_X$.
\begin{itemize}
\item A \emph{normal test configuration} $(\mathcal{X},\mathcal{L})$ of the polarized pair $(X,L)$ consists of
\begin{enumerate}
\item a normal projective variety $\mathcal{X}$ with a $\mathbb{G}_m$ action,
\item a flat $\mathbb{G}_m$-equivariant morphism $p: \mathcal{X}\to \mathbb{P}^1$, where $\mathbb{G}_m$ acts on $\mathbb{P}^1$ by
$$
(t,[x:y])\mapsto [tx:y]\;,
$$
\item a $\mathbb{G}_m$-invariant $p$-ample $\mathbb{Q}$-line bundle $\mathcal{L}\to \mathcal{X}$ and a $\mathbb{G}_m$-equivariant isomorphism
$$
\bigl(\mathcal{X}\setminus p^{-1}(0),\mathcal{L}\vert_{\mathcal{X}\setminus p^{-1}(0)}\bigr)\ \cong\ \bigl(X\times(\mathbb{P}^1\setminus \{0\}),\operatorname{pr}_1^*(L)\bigr)
$$
where $\operatorname{pr}_1$ is the first projection and $0=[0:1]$.
\end{enumerate}

\item 
The \emph{Donaldson--Futaki invariant} of the test configuration $(\mathcal{X},\mathcal{L})$ is
$$
\operatorname{DF}(\mathcal{X},\mathcal{L})\ :=\
\frac{1}{(L^n)}\bigl((\mathcal{L}^{n})\cdot K_{\mathcal{X}/\mathbb{P}^1} +\frac{n}{n+1} (\mathcal{L}^{n+1})\bigr)\;.
$$
\end{itemize}
\end{definition}

\begin{remark}
Definition~\ref{definition:DF} differs from the original one and uses many works by several authors. Using semistable reduction and birational geometry,
\cite{LX14} reduces the necessary test conditions needed to prove K-(semi, poly)stability to the ones with klt central fibers; see Definition~\ref{definition:K-stability} below. Moreover, in the literature, the definition of the Donaldson--Futaki invariant may differ by a positive constant.
\end{remark}

Let $\mathcal{X}_\infty=p^{-1}(\infty)$ in Definition~\ref{definition:DF}, where $\infty=[1:0]$.
Then the test configuration $(\mathcal{X},\mathcal{L})$ is said to be
\emph{trivial} if there is a $\mathbb{G}_m$-equivariant isomorphism
$$
\bigl(\mathcal{X},\mathcal{L}\bigr) \ \cong\ \bigl(X\times\mathbb{P}^1 ,\operatorname{pr}_1^*(-K_X)\bigr).
$$
Similarly, the test configuration $(\mathcal{X},\mathcal{L})$ is said to be of \emph{product-type} if $\mathcal{X}\setminus \mathcal{X}_\infty\cong X\times (\mathbb{P}^1\setminus \{\infty\})$,

\begin{remark}
\label{remark:product-type}
 Note that the Donaldson--Futaki invariant of a trivial test configuration is trivial.
Moreover, if a product-type test configuration of a pair $(X,-K_X)$ has positive Donaldson--Futaki invariant then by  composing the $\mathbb{G}_m$-action by $t\mapsto \frac1t$ we get a product-type test configuration of  $(X,-K_X)$ with negative Donaldson--Futaki invariant.
\end{remark}

\begin{definition}[K-stability]
\label{definition:K-stability}
The Fano variety $X$ is
\begin{enumerate}
\item \emph{K-semistable} if for all normal (or klt) test configuration $(\mathcal{X},\mathcal{L})$ of $(X,-K_X)$, one has
$$
\operatorname{DF}(\mathcal{X},\mathcal{L})\geqslant 0\;.
$$
\item \emph{K-polystable} if it is K-semistable and $\operatorname{DF}(\mathcal{X},\mathcal{L})=0$ if and only if $(\mathcal{X},\mathcal{L})$ is  of product type.
\item \emph{K-stable} if it is K-semistable and $\operatorname{DF}(\mathcal{X},\mathcal{L})=0$ if and only if $(\mathcal{X},\mathcal{L})$ is trivial.
\end{enumerate}
\end{definition}

By Remark~\ref{remark:product-type}, if $X$ is K-stable, then it is K-polystable. If $X$ is K-polystable, then by definition it is K-semistable.
Moreover, by \cite[Corollary 1.3]{BX19}, if $X$ is K-stable, then $\mathrm{Aut}(X)$ is finite, which gives the following.

\begin{corollary}
\label{corollary:finitepoly}
The Fano variety $X$ is K-stable if and only if it is K-polystable and $\operatorname{Aut}(X)$ is finite.
\end{corollary}

\begin{proof}
Indeed, if $\operatorname{Aut}(X)$ is finite then any product-type configuration is trivial.
\end{proof}

More generally by \cite{Ma57} in the smooth case and by \cite{ABHLX20} in the singular case, we have

\begin{theorem}
\label{theorem:reduct}
If the Fano variety $X$ is K-polystable, then $\operatorname{Aut}(X)$ is reductive.
\end{theorem}

Originally the name K-stability was introduced in \cite{Ti97} for Fano varieties. It was defined in reference to the K-energy functional, previously introduced by Mabuchi, on the space of K\"ahler metrics on a compact complex manifold with fixed K\"ahler class. It was generalized to polarized manifolds in \cite{Do02}. In K-energy functional ``K'' stands for the first letter of the German word ``Kanonisch'' but also for ``Kinetic'' and ``K\"ahler''.

\subsection{Yau-Tian-Donaldson Conjecture}

The Yau--Tian--Donaldson Conjecture predicts that the existence of a K\"ahler-Einstein metric on a complex Fano manifold is equivalent to being K-polystable.
The solution was published in \cite{CDS15} and \cite{Ti15,Ti15corr}. See \cite{Es16} for an overview of the proof.

\begin{theorem}[Chen--Donaldson--Sun, Tian]
\label{theorem:CDS}
A complex Fano manifold is K\"ahler-Einstein if and only if it is K-polystable.
\end{theorem}

In Definition~\ref{definition:K-stability}, we assumed that the Fano variety $X$ has klt singularities.
However, this definition makes sense without this assumption. This was a rather surprising fact that klt singularities, which are natural to consider from the perspective of MMP, are the ones which also naturally appear in K-stability by the following result of Odaka.

\begin{theorem}[{\cite[Theorem 1.3]{Od13klt}}]
\label{theorem:oda}
Let $X$ be a Fano variety. If $X$ is K-semistable, then $X$ has klt singularities.
\end{theorem}

By Theorems~\ref{theorem:CDS}, the product of K-polystable smooth Fano varieties is K-polystable.
This result remains valid for singular Fano varieties:

\begin{theorem}[\cite{Zh20}]
\label{theorem:Zhuang20}
Let $X$ and $Y$ be Fano varieties. If both $X$ and $Y$ are K-polystable, then $X\times Y$ is K-polystable.
\end{theorem}

The proof of this theorem is purely algebraic. It can also be derived from a generalization of Theorem~\ref{theorem:CDS} to Fano varieties with klt singularities, which is the endpoint of many works by many authors, see~\cite{LXZ22} and references therein.
We can go further on understanding connection between existence of K\"ahler-Einstein metrics and K-stability  via Chow-Mumford (CM)-stability, see \cite[\S~6]{Ti99}.

\subsection{Fujita--Li's valuative criterion}
\label{section:Fujita-Li}

Determining whether a given Fano manifold is K-(semi)stable is a quite challenging problem. The Fujita--Li valuative criterion (see Theorem~\ref{theorem:Fujita-Li}) gives an alternative definition of K-(semi)stability. From now on, we let $X$ be a Fano variety.

\begin{definition}
Let $f \colon Y\to X$ be a birational morphism such that $Y$ is projective and normal, and let $E$ be a prime divisor on $Y$.
We say that $E$ is a \emph{divisor over} $X$.
Thus, a divisor $E$ over $X$ comes together with a morphism $f \colon Y\to X$.
We say that $f(E)$ is the \emph{center} of $E$ in $X$, which is denoted by $c_X(E)$.
\end{definition}

The \emph{discrepancy} $a(E,X)=\ord_E(K_Y-f^*(K_X))$ of a prime divisor $E$ over $X$ is defined in \cite[Definition 2.22]{KM98}.
The \emph{log discrepancy} of a prime divisor $E$ over $X$ is
$$
A_X(E)\ =\ 1+\ord_E(K_Y-f^*(K_X))\ =\ 1+a(E,X).
$$
So $X$ being klt is equivalent to saying that $A_X(E)$ is positive for every prime divisor $E$ over $X$.
Similarly, we define the \textit{expected vanishing order} of the anticanonical divisor $-K_X$ at a prime divisor $E$ over $X$ as
$$
S_X(E)\ :=\ \frac1{\vol_X (-K_X)}\int_0^\infty \vol_Y \bigl(f^*(-K_X)-t E\bigr)dt\;.
$$
This definition uses volumes of $\mathbb{R}$-divisors, and we refer the reader to Section~\ref{section:ZD} for more details.
Finally, we define the $\beta$-invariant of a prime divisor $E$ over $X$ to be
$$
\beta(E)\ =\ A_X(E)-S_X(E)\;.
$$

\begin{theorem}[Fujita, Li, Blum--Xu]
\label{theorem:Fujita-Li} A Fano variety $X$ is K-semistable (resp. K-stable) if and only if $\beta(E)\geqslant 0$ (resp. $\beta(E)>0$) for every prime divisor $E$ over $X$.
\end{theorem}

\begin{proof}
See \cite{Fu19,Li17,BX19}.
\end{proof}

Note that this theorem implies Theorem~\ref{theorem:oda}.
Indeed, if the singularities of $X$ are worse than klt, then there exists a prime divisor $E$ over $X$ with $A_X(E)\leqslant 0$,
which gives $\beta(E)<0$, so that $X$ is not K-semistable by Theorem~\ref{theorem:Fujita-Li}.

Recall that from Section~\ref{section:ZD} that the volume of a divisor is positive if and only if the divisor is big.
Keeping this in mind, let $\tau(E)$ be the \emph{pseudo-effective threshold} of a divisor $E$ over $X$ with respect to $-K_X$ as
$$
\tau(E)\ :=\ \sup\bigl\{t\in \mathbb{Q}_{>0} \ \vert\ f^*(-K_X)-tE \ \textrm{is big} \bigr\},
$$
so we have
$$
S_X(E)\ =\ \frac1{(-K_X)^n}\int_0^{\tau(E)}  \vol \bigl(f^*(-K_X)-tE\bigr)dt.
$$

\subsection{Global log canonical threshold and stability threshold}
\label{section:Tian}

Let $X$ be a Fano variety with klt singularities. We shall consider two natural ``thresholds'' of $X$:
\begin{enumerate}
\item the global log canonical threshold of $X$, also known as $\alpha$-invariant, measuring the worst singularities of the divisors in the linear system $|-mK_X|$ as $m \to \infty$,
\item the stability threshold, also known as $\delta$-invariant, which measures the ``average'' singularities of these divisors.
\end{enumerate}
Recall from \cite[\S~8]{Ko97} that the log canonical threshold (lct) of an effective $\mathbb{Q}$-divisor $D$ on $X$ is the following positive rational number:
$$
\operatorname{lct}(X, D)\ :=\  \sup\bigl\{ \lambda\in\mathbb{Q}^+ \,\vert\, (X,\lambda D) \,\text{is log canonical} \bigr\}\;.
$$

\begin{definition}[$\alpha$-invariant]
The \emph{global log canonical threshold} $\alpha(X)$ of $X$ is given by
$$
\alpha(X)\ =\ \inf\bigl\{\operatorname{lct}(X, D)\,\vert\, D \,\text{is an effective $\mathbb{Q}$-divisor such that}\, D \sim_\mathbb{Q} -K_X\bigr\}
$$
\end{definition}

Observe that the invariant $\alpha(X)$ has a global nature. It measures the singularities of effective $\mathbb{Q}$-divisors on $X$ that are $\mathbb{Q}$-linearly equivalent to the anticanonical divisor $-K_X$.  This definition appeared in \cite{Cheltsov,CheltsovGAFA,CP}, and Demailly proved in the appendix to \cite{CS} that $\alpha(X)$ coincides with the $\alpha$-invariant introduced much earlier by Tian in \cite{Ti87}.

To describe the stability threshold of $X$, by \cite{FO18}, an effective $\mathbb{Q}$-divisor $D \sim_\mathbb{Q} -K_X$ on $X$ is called of \emph{$m$-basis type}, where $m \geqslant 1$, if there exists a basis $s_1, . . . , s_{N_m}$
of $H^0(X, -mK_X)$ such that
$$
D=\frac{\{s_1 = 0\} + \{s_2 = 0\} +\dots+ \{s_{N_m}= 0\}}{mN_m}
$$
where $N_m=h^0(X, -mK_X)$. Define
$$
\delta_m(X):= \inf\bigl\{\operatorname{lct}(X,D) \,\vert\, D \sim_\mathbb{Q} -K_X \ \text{of $m$-basis type}\bigr\}\;.
$$
By \cite[Theorem A]{BJ20}, $\lim_{m\to \infty} \delta_m(X)$ exists.

\begin{definition}[$\delta$-invariant]
The \emph{stability threshold} $\delta(X)$ of $X$ is given by
$$
\delta(X)= \lim_{m\to \infty} \delta_m(X)
$$
\end{definition}

Global log canonical threshold of $X$ and stability threshold of $X$ are related as follows:

\begin{theorem}[\relax{\cite[Theorem A]{BJ20}}]
\label{theorem:ineq}
Let $n$ be the dimension of the Fano variety $X$. Then
$$
0\ <\ \frac{n+1}n\alpha(X)\ \leqslant\  \delta(X)\ \leqslant\ (n+1)\alpha(X)\;.
$$
\end{theorem}

Furthermore, there are valuative characterizations of these thresholds. The theorem below follows from \cite{FO18} and \cite[Theorem C]{BJ20}.

\begin{theorem}
We have
$$
\alpha(X)\ =\ \inf_{E/X} \frac{A_X(E)}{\tau(E)}\,,
\qquad
\delta(X)\ =\ \inf_{E/X} \frac{A_X(E)}{S_X(E)} \,,
$$
where both infima are taken through all prime divisors $E$ over $X$.
\end{theorem}

From Theorem~\ref{theorem:Fujita-Li}, the condition $\delta(X)\geqslant 1$ is equivalent to $X$ being K-semistable. Gathering \cite{FO18},\cite{BJ20} and \cite{LXZ22} we get:

\begin{theorem}
\label{theorem:kstab-delta}
A Fano variety $X$ is K-semistable (resp. K-stable) if and only if $\delta(X)\geqslant 1$ (resp. $\delta(X)>1$).
\end{theorem}

\begin{corollary}
A Fano variety is K-stable if
$$
\alpha(X)\ >\ \frac{\dim (X)}{\dim (X)+1}.
$$
\end{corollary}

\begin{proof}
Immediate from Theorem~\ref{theorem:ineq} and Theorem~\ref{theorem:kstab-delta} above.
\end{proof}

Note that there are K-stable Fano varieties $X$ such that $\alpha(X)<\frac{\dim (X)}{\dim (X)+1}$.

\subsection{Equivariant K-stability}
\label{section:equivariant}

In Definition~\ref{definition:K-stability}, we assumed that $X$ is a possibly singular Fano variety defined over $\mathbb{C}$.
Suppose now that $X$ is a Fano variety with klt singularities that is defined over a subfield $\Bbbk\subset\mathbb{C}$.
Then we can repeat Definition~\ref{definition:K-stability} assuming that everything there is defined over $\Bbbk$.
This gives us definition of $X$ being K-semistable or K-polystable as a variety defined over $\Bbbk$.
Similarly, we may define $G$-equivariant K-semistability and equivariant K-polystability of $X$ for an algebraic subgroup $G\subset\mathrm{Aut}(X)$.

Let $X_{\mathbb{C}}$ be the geometric model of $X$.
There was a folklore conjecture that $X_{\mathbb{C}}$ is K-semistable (K-polystable, respectively) if and only if $X$ is $G$-equivariantly K-semistable ($G$-equivariantly K-polystable, respectively). This conjecture had been first confirmed for smooth complex Fano manifolds with a reductive group action,
and then for singular complex Fano varieties with a finite group action, and then proved in full generality in \cite{Zhuang21}.
From \cite[Theorem~1.1]{Zhuang21}, we extract the following.

\begin{theorem}
\label{theorem:equiv}
If $X$ is $G$-equivariantly K-semistable, then $X_{\mathbb{C}}$ is K-semistable.
If $G$ is reductive and $X$ is $G$-equivariantly K-polystable, then $X_{\mathbb{C}}$ is K-polystable.
\end{theorem}

\begin{corollary}
If $X$ is K-semistable (resp. K-polystable), then $X_{\mathbb{C}}$ is K-semistable (resp. K-polystable).
\end{corollary}

A very handy sufficient condition for K-polystability of $X_{\mathbb{C}}$, which we will use to prove the Main Theorem, is the following result.

\begin{theorem}[{\cite{Zhuang21}}]
\label{theorem:Zhuang}
Suppose that $\beta(E)>0$ for every $G$-invariant geometrically irreducible divisors $E$ over $X$ defined over $\Bbbk$.
Then $X_{\mathbb{C}}$ is K-polystable.
\end{theorem}

The main ingredient in the proof of Theorem~\ref{theorem:equiv} are the following results.

\begin{theorem}[{\cite[Theorem~1.2]{Zhuang21}}]
\label{theorem:deltamin}
Suppose that $X_{\mathbb{C}}$ is not K-semistable. Then
$$
\delta(X_{\mathbb{C}})\ =\ \inf_{E/X}\frac{A_{X}(E)}{S_X(E)}
$$
where the infimum runs over all $G$-invariant geometrically irreducible divisors $E$ over $X$ defined over $\Bbbk$.
\end{theorem}

By \cite[Postscript note]{Zhuang21}, the infimum in Theorem~\ref{theorem:deltamin} is in fact a minimum.

\subsection{Openness and K-moduli}
\label{section:openness}

We briefly explain how K-stability behaves in families. This was first of all understood for Fano manifolds with finite automorphism groups.
Namely, in  \cite[Theorem~1.3]{Od13open}, Odaka proved the Zariski openness of the K\"ahler-Einstein condition for complex Fano manifolds whose automorphism group is finite.  The proof uses Chen--Donaldson--Sun Theorem~\ref{theorem:CDS}.
Then Donaldson, \cite[Theorem~1]{Do15}, gave a more direct proof without using Theorem~\ref{theorem:CDS}.

\begin{theorem}[Odaka, Donaldson]
Let $\mathcal{X}\to T$ be a family of complex Fano manifolds, where the base $T$ is a quasi-projective variety.
Define $T^*\subset T$ to be the set of parameters $t$ such that $\mathcal{X}_t$ admits a K\"ahler-Einstein metric, and assume for each $t \in T$ the automorphism group of $\mathcal{X}_t$ is finite, then $T^*$ is a Zariski-open subset of~$T$.
\end{theorem}

Later \cite{BL22} proves that uniform K-stability is a Zariski open condition.
The authors use characterization from \cite{BJ20} of uniform K-stability by the value of the $\delta$-invariant $\delta(X)>1$.

\begin{theorem}[Blum--Liu] 
Let $\mathcal{X}\to T$ be a projective $\mathbb{Q}$-Gorenstein family of Fano varieties with klt singularities over a
normal base, then $$\{t\in T\, \vert\,  \mathcal{X}_t \ \text{is uniformly K-stable}\}$$ is an open subset of $T$.
\end{theorem}

The following theorem is then the outcome of many works.

\begin{theorem}
\label{theorem:openness}
Let $\mathcal{X}\to T$ be a projective $\mathbb{Q}$-Gorenstein family of Fano varieties with klt singularities over a
normal base, then
\begin{enumerate}
\item
$\{t\in T\, \vert\,  \mathcal{X}_t \ \text{is K-stable}\}$ is a Zariski open subset of $T$.
\item
$\{t\in T\, \vert\,  \mathcal{X}_t \ \text{is K-semistable}\}$  is a Zariski open subset of $T$.
\item
$\{t\in T\, \vert\,  \mathcal{X}_t \ \text{is K-polystable}\}$  is a constructible subset of $T$.
\end{enumerate}
\end{theorem}

We conclude this section with the following theorem, which is usually called the \emph{K-moduli Theorem} and is attributed to many people (cf. \cite{ABHLX20,BHLLX21,BLX19,BX19,CP21,Jia20,LWX21,LXZ22,Xu21,XZ20,XZ21}).

\begin{theorem}[K-moduli Theorem]
\label{kmoduli}
Fix two numerical invariants $n\in \mathbb{N}$ and $V\in \bQ_{>0}$. Consider the moduli functor $\mts{M}^K_{n,V}$ sending a base scheme $S$ to the groupoid
\[
\left\{\mts{X}/S\left| \begin{array}{l} \mts{X}\to S\textrm{ is a proper flat morphism, each geometric fiber}\\ \textrm{$\mts{X}_{\bar{s}}$ is an $n$-dimensional K-semistable $\bQ$-Fano variety of}\\ \textrm{volume $V$, and $\mts{X}\to S$ satisfies Koll\'ar's condition}\end{array}\right.\right\}.
\]
Then there is an Artin stack, still denoted by $\mts{M}^K_{n,V}$, of finite type over $\mathbb{C}$ with affine diagonal which represents the moduli functor. The $\mathbb{C}$-points of $\mts{M}^K_{n,V}$ parameterize K-semistable $\mathbb{Q}$-Fano varieties $X$ of dimension $n$ and volume $V$. Moreover, the Artin stack $\mts{M}^K_{n,V}$ admits a good moduli space $\overline{M}^K_{n,V}$, which is a projective scheme, whose $\mathbb{C}$-points parameterize K-polystable $\mathbb{Q}$-Fano varieties. The CM $\mathbb{Q}$-line bundle $\lambda_{\CM}$ on $\mts{M}^K_{n,V}$ descends to an ample $\mathbb{Q}$-line bundle $\Lambda_{\CM}$ on $\overline{M}^K_{n,V}$.
\end{theorem}

\section{Volumes and Zariski decompositions}
\label{section:ZD}


\subsection{Zariski decomposition of divisors on surfaces}

Let $X$ be a smooth projective surface, and let $D$ be a pseudo-effective $\R$-divisor on $X$.
As observed by Zariski \cite{Zar} and later generalized in \cite{KMM}, the divisor $D$ can be decomposed in a very helpful way as follows (besides the original sources, further details can be found in \cite[Section 2.1]{GeomAsp} and \cite[Section 2.3.E]{PAGI}).

\begin{theorem}[{\cite[Theorem 7.3.1]{KMM}}]
\label{thm:ZD}
There exists a unique effective $\R$-divisor
$$
N_D=\sum_{i=1}^m a_iN_i
$$
such that
\begin{enumerate}
\item[(i)] $P_D=D-N_D$ is nef,
\item[(ii)] $N_D$ is either zero or its intersection matrix $(N_i\cdot N_j)$ is negative definite,
\item[(iii)] $P_D\cdot N_i=0$ for $i=1,\dots, m$.
\end{enumerate}
Furthermore, the divisor $N_D$ is uniquely determined as a cycle by the numerical equivalence class of $D$, and if $D$ is a $\Q$-divisor, then so are $P_D$ and $N_D$. The decomposition $D=P_D+N_D$ is called the {\em Zariski decomposition} of $D$.
\end{theorem}

\begin{remark}
The following holds with the notation of Theorem~\ref{thm:ZD}.
\begin{enumerate}
\item If $D$ is nef then $P_D=D$ an $N_D=0$.
\item If $D=\sum_{i=1}^m a_iN_i$ such that each $a_i>0$ and the intersection matrix $(N_i\cdot N_j)$ is negative definite, then $P_D=0$ and $N_D=D$.
\end{enumerate}
In the context of Zariski decomposition, a negative curve is any reduced and irreducible curve with negative self-intersection.
\end{remark}

Even though positive and negative parts of a $\Q$-divisor are themselves $\Q$-divisors, there are examples of integral divisors with non-integral positive and negative parts.

\begin{example} Suppose that $X$ can be obtained by blowing up $\mathbb{P}^2$ at a point, let $E$ be the exceptional divisor, and let $H$ be the strict transform on $X$ of a general line in $\mathbb{P}^2$. Then $E$ is the only negative curve on the surface, $\mathrm{Pic}(X)$ is generated by $H$ and $E$,	
the cone $\mathrm{Eff}(X)$ is generated by $H$ and $H-E$, and the nef cone of $X$ is generated by $H$ and $H-E$. In particular, for $a,b>0$ the Zariski decomposition of the divisor $D=aH+bE$ is $P_D=aH$ and $N_D=bE$.
\end{example}

One of the most important properties of Zariski decomposition is that the positive part carries all the global sections.

\begin{proposition}\label{prop:ZD positive part}
Let $X$ be a smooth projective surface, $L$ a pseudo-effective integral Cartier divisor with Zariski decomposition $L=P_L+N_L$. Then for every $k\geqslant 1$, we have
\[
H^0(X,kL) \simeq H^0(X,\lfloor kP_L \rfloor)\ .
\]
\end{proposition}
\begin{proof}
This is \cite[Proposition 2.3.21]{PAGI}	
\end{proof}

It is a natural question to ask how Zariski decomposition changes as we move around the big or pseudo-effective cone. The following result gives a complete answer.

\begin{theorem}[Variation of Zariski decomposition, \cite{BKS04}, Theorem 1.1]\label{thm:Variation of ZD}
Let $X$ be a smooth projective surface over the complex numbers. Then there exists a locally finite decomposition of $\Bbig(X)$ into locally rational polyhedral chambers such that the support of the negative part of the Zariski decomposition is constant in each chamber. 	
\end{theorem}

\begin{corollary}[Continuity of Zariski decomposition]\label{cor:cont of ZD}
With notation as above, the Zariski decomposition of pseudo-effective $\R$-divisors is continuous in $\Bbig(X)$.
\end{corollary}
\begin{proof}
This is \cite[Proposition 1.14]{PAGI}.
\end{proof}

\begin{remark}[Non-continuity of Zariski decomposition]
The restriction in Corollary~\ref{cor:cont of ZD} is non-negotiable. There exist examples where Zariski decomposition is not continuous when converging towards the pseudo-effective boundary.
\end{remark}

\begin{example}[Two-point blowing-up of the plane] This item is taken from  \cite[Example 3.5]{BKS04}
\label{eg:two_points}
Let $X$ be the blow-up of the projective plane at two points; the corresponding exceptional divisors will be denoted by $E_1$ and $E_2$.
We denote the pullback of the hyperplane class on ${\PP}^2$ by $L$.
These divisor classes generate the Picard group of $X$
and their intersection numbers are: $L^2=1$, $(L.E_i)=0$ and $(E_i.E_j)=-{\delta}_{ij}$ for $1\leqslant i,j\leqslant 2$.
There are  three irreducible negative curves: the two exceptional divisors,
$E_1$,  $E_2$ and  the strict transform  of the line through the two blown-up points, $L-E_1-E_2$, and  the corresponding hyperplanes determine the
chamber structure on the big cone.  They divide the big cone into five regions
on each of which the support of the negative part of the Zariski decomposition remains constant.
In this particular case the chambers are simply described as the set of divisors
that intersect negatively the same set of negative curves (it happens for instance on K3 surfaces that the negative part of a divisor $D$ contains irreducible curves intersecting $D$ non-negatively). We will parametrize the chambers with big and nef divisors (in fact with faces of the nef cone containing big divisors). In our case, we pick big and nef divisors $A,Q_1,Q_2,L,P$
based on the following criteria:
\begin{eqnarray*}
	A\ : & \text{ample} \\
	Q_1\ : & (Q_1\cdot E_1)=0\ ,\ (Q_1\cdot C) > 0 \ \text{for all other curves} \\
	Q_2\ : & (Q_2\cdot E_2)=0\ ,\ (Q_2\cdot C) > 0 \ \text{for all other curves}  \\
	L\ : & (L\cdot E_1)=0\ ,\  (L\cdot E_2)=0\ ,\ (L\cdot C) > 0 \ \text{for all other curves} \\
	P\ : & (P\cdot L-E_1-E_2)=0\ ,\ (P\cdot C) > 0 \ \text{for all other curves} \; .
\end{eqnarray*}
The divisors  $L, P,Q_1,Q_2$ are big and nef divisors in the nef boundary (hence necessarily non-ample) which are in the relative interiors of the indicated faces.
For one of these divisors, say $P$,  the corresponding chamber consists of all $\R$-divisor classes whose negative part consists of curves orthogonal to $P$. This is listed in the following
table. Observe that apart from the nef cone,  the chambers do not contain the nef divisors they are associated to.
Note that  not all possible combinations of negative divisors  occur.
This in part is accounted for by the fact that certain faces of the
nef cone do not contain big divisors.
\begin{center}
	\begin{tabular}{|l|r|r|} \hline
		Chamber & curves in the negative part of $D$ & $(D\cdot C)<0$  \\ \hline
		$\Sigma_A$ & $\emptyset$ & none \\
		$\Sigma_{Q_1}$ & $E_1$ & $E_1$ \\
		$\Sigma_{Q_2}$ & $E_2$ & $E_2$ \\
		$\Sigma_{L}$ & $E_1,E_2$ & $E_1,E_2$ \\
		$\Sigma_{P}$ & $L-E_1-E_2$ & $L-E_1-E_2$ \\
		\hline
	\end{tabular}
\end{center}
The following picture describes a cross section of the effective cone of $X$ with the chamber structure indicated.
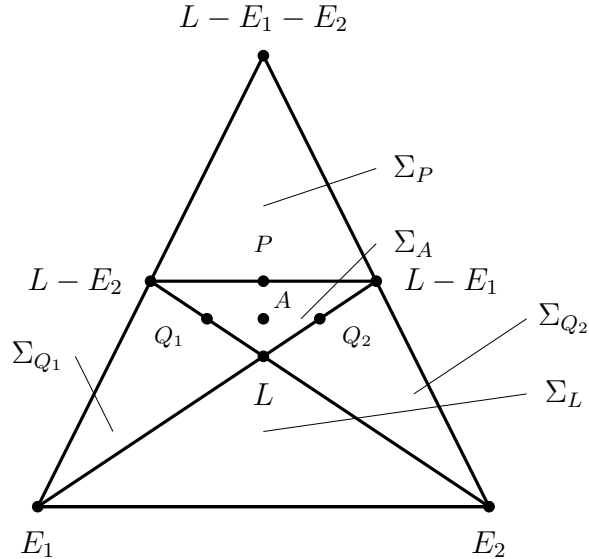
\begin{center}
	\begin{tikzpicture}[scale=0.5]
		\draw[very thick] (0,0) -- (12,0) -- (6,12) -- cycle;
		\filldraw[black] (0,0) circle (4pt) (12,0) circle (4pt) (6,12) circle (4pt);
		\draw (0,-1) node{$E_1$} (12,-1) node{$E_2$} (6,13) node{$L-E_1-E_2$};
		\draw[very thick] (0,0) -- (9,6) -- (3,6) --  (12,0) -- cycle;
		\filldraw[black] (9,6) circle (4pt) (3,6) circle (4pt) (6,4) circle (4pt);
		\draw (6,3) node{$L$} (11,6) node{$L-E_1$} (1,6) node{$L-E_2$};
		\filldraw[black] (6,5) circle (4pt) (6,6) circle (4pt) (7.5,5) circle (4pt) (4.5,5) circle (4pt);
		\draw (6.5,5.5) node{\scriptsize $A$} (8.5,4.5) node{\scriptsize $Q_2$} (3.5,4.5) node{\scriptsize $Q_1$} (6,7) node{\scriptsize $P$};
		\draw[very thin] (6,2) -- (13,3)  (14,3) node{$\Sigma_L$};
		\draw[very thin] (7,5) -- (9,7)  (10,7) node{$\Sigma_A$};
		\draw[very thin] (6,8) -- (9,9) (10,9) node{$\Sigma_P$};
		\draw[very thin] (10,3) -- (13,5) (14,5) node{$\Sigma_{Q_2}$};
		\draw[very thin] (2,2) -- (1,4) (0,4) node{$\Sigma_{Q_1}$};
	\end{tikzpicture}
\end{center}
Let $D=aL-b_1E_i-b_2E_2$ be a big $\R$-divisor. Then one can express
the volume of $D$ in terms of the coordinates $a,b_1,b_2$ as follows:
\[
\vvol{X}{D}= \left\{ \begin{array}{ll} D^2=a^2-b_1^2-b_2^2 & \textrm{ if $D$ is nef, i.e. $D\in\Sigma_A$} \\
	a^2-b_2^2 & \textrm{ if\ } D\cdot E_1<0 \mbox{ and } D\cdot E_2\geqslant 0 \textrm{ i.e. $D\in \Sigma_{Q_1}$}\\
	a^2-b_1^2 & \textrm{ if\ } D\cdot E_2<0 \mbox{ and } D\cdot E_1\geqslant 0 \textrm{ i.e. $D\in\Sigma_{Q_2}$}\\
	a^2 & \textrm{ if } D\cdot E_1<0 \mbox{ and } D\cdot E_2<0 \textrm{ i.e. $D\in\Sigma_L$}\\
	2a^2-2ab_1-2ab_2+2b_1b_2 & \textrm{ if\ } D\cdot(L-E_1-E_2)<0\
	\textrm{ i.e. $D\in\Sigma_P$}.
\end{array} \right.
\]	
\end{example}

\begin{remark}[Algorithm to compute Zariski decomposition]
Computing the Zariski decomposition of a divisors is algorithmic \emph{assuming} one knows the intersection form and enough of the negative curves on the surface. Of this information the knowledge of the negative curves on the underlying surface tends to be prohibitively difficult.
In any case, let us assume that we have determined the Mori cone of the surface $X$ we work with, and let $D$ be a pseudo-effective $\R$-divisor on $X$. In  this case the following (pseudo)algorithm delivers the Zariski decomposition of $D$.
In fact the proof of Theorem~\ref{thm:ZD} is precisely the verification that the algorithm above terminates and gives the right result.
\begin{enumerate}
	\item If $D$ is nef then set $P_D=D$, $N_D=0$ and stop.
	\item If $D$ is not nef then there are finitely many negative curves $C_1,\dots,C_r$ such that $(D\cdot C_i)<0$ for all $1\leqslant i\leqslant r$.
	Solve the system of linear equations
	\[
	(D-\sum_{i=1}^{r}a_iC_i\cdot C_j) \equ 0\ \ \text{for all $1\leqslant j\leqslant r$.}
	\]
	\item Set
	\[
	D \deq D-\sum_{i=1}^{r}a_iC_i\ ,
	\]
	and return to (1).
\end{enumerate}
The algorithm delivers $P_D$, from which we calculate $N_D=D-P_D$. Unsurprisingly, the latter happens to be the sum of the  divisors  $\sum_{i=1}^{r}a_iC_i$ picked up along the way. An implementation of the algorithm in the \texttt{Magma} computer algebra system is available in
Appendix~\ref{app:zdec}.
\end{remark}

\subsection{Zariski decomposition in higher dimensions}

Once we leave the realm of surfaces, Zariski decomposition will cease to exist in its original intersection-theoretic sense. An attempt to generalize it to higher dimensions via Proposition~\ref{prop:ZD positive part} leads to the following notion. A lot more information on the topic can be found in Nakayama's book \cite{Nakayama_ZD} and \cite{Prokhorov}.

\begin{definition}[Zariski decomposition in the sense of Cutkosky--Kawamata--Moriwaki]\label{defn:ZD CKM}
Let $X$ be a projective variety, $D$ an integral Cartier divisor on $X$.

We say that $D$ \emph{has a Zariski decomposition in the sense of CKM} if there exists a proper birational morphism $\pi\colon Y\to X$ with $Y$ smooth and an effective $\Q$-divisor $N$ on $Y$ such that   the following hold.
\begin{enumerate}
    \item The divisor $P\deq \pi^*D-N$ is a nef $\Q$-divisor on $Y$;
    \item for every $m\geqslant 1$, the natural maps
    \[
    H^0(Y,\pi^*(mD)-\lceil mN \rceil) \longrightarrow H^0(Y,\pi^*(mD))
    \]
    are isomorphisms.
\end{enumerate}
\end{definition}

\begin{remark}
    As opposed to the case of surfaces it is not automatic for (pseudo)effective divisors on higher-dimensional varieties to possess a Zariski decomposition in the sense of CKM. In fact, there exist many cases where the decomposition does not exist. One such obstacle to the existence of Zariski decomposition is the irrationality of the volume of the divisor, which we will discuss in Subsection 3.3.
\end{remark}

\begin{remark}
For our purposes we will often need a Zariski decomposition on the original variety $X$. In this case we say that \emph{$D$ has a Zariski decomposition on $X$}.	
\end{remark}

\begin{remark}[Zariski decomposition for finitely generated divisors]\label{rmk:ZD for fg}
    As explained in \cite[Example 2.1.31]{PAGI}, given divisor $D$ with a finitely generated section ring on a normal projective variety $X$, there exist a non-negative integer $p$, and proper birational map $\pi\colon Y\to X$ with $Y$ normal (after blowing up further we can safely assume it is smooth) and an effective integral divisor $N$ on $Y$ such that
    \[
    P \deq \pi^*(pD) - N
    \]
    is a base-point free divisor on $Y$ (in particular $P$ is nef), and
    \[
    R(X,D)^{(p)} \equ R(Y,P)\ .
    \]
    This implies in particular that $D$ has a Zariski decomposition in the sense of CKM on $Y$, where $\frac{1}{p}P$ plays the positive, and $\frac{1}{p}N$ the negative part.
\end{remark}

\begin{theorem}\label{thm:MDS ZD}
    Let $X$ be a Mori dream space, $D$ an effective divisor on $X$. Then $D$ has a Zariski decomposition in the sense of CKM over $X$.
\end{theorem}
\begin{proof}
    On a Mori dream space every effective divisor has a finitely generated section ring. The rest follows from Remark~\ref{rmk:ZD for fg}.
\end{proof}

\begin{remark}
    It is important to point out that the CKM Zariski decomposition of a divisor $D$ on a Mori dream space might not (and in general does not) exist on $X$, only on some proper birational model over $X$.
\end{remark}

\begin{corollary}
    Every effective divisor on a log Fano variety has a Zariski decomposition in the sense of CKM.
\end{corollary}
\begin{proof}
    By \cite{BCHM} or \cite{CasciniLazic}, log Fano varieties are Mori dream spaces; apply Theorem~\ref{thm:MDS ZD}.
\end{proof}

\begin{remark}
    As opposed to the case of Zariski decompositon on surfaces, the existence of a Zariski decomposition in the sense of CKM is not determined by the numerical equivalence class of a divisor. For a closely related discussion about this subtle question see \cite[Section 4]{KKL12}.
\end{remark}


The following result is stated in~\cite[Theorem~3.3.4.8]{adhl}.

\begin{theorem}[Zariski decomposition via Cox rings]
\label{thm:zariski-cox}
Let $X$ be a Mori dream space, and let $D_1, \dots, D_r$
be the prime divisors associated with a minimal set of
generators of the Cox ring of $X$.
Let $D$ be an effective divisor on $X$.
Consider the cone
\[
\tau_i := \operatorname{cone}(w_D, -w_i) \cap \operatorname{cone}(w_1, \dots, w_{i-1}, w_{i+1}, \dots, w_r) \subseteq \operatorname{Cl}_{\mathbb{Q}}(X),
\]
where $w_D$ and $w_i$ denote the classes of $D$ and $D_i$ in
${\rm Cl}_{\mathbb{Q}}(X)$.
Let $\mu_i$ be the smallest non-negative rational number such that
$w_D - \mu_i w_i$ lies on an extremal ray of $\tau_i$.
Then $D$ admits a Zariski decomposition
\[
D = P + N, \quad \text{where} \quad N := \mu_1 D_1 + \cdots + \mu_r D_r, \quad P := D - N,
\]
with unique rational Weil divisors $P$ and $N$ on $X$.

Moreover, for any $n \in \mathbb{Z}_{\ge 0}$ such that $nN$ is an integral Weil divisor, the space of global sections
$H^0(X, \mathcal{O}_X(nN))$ has dimension at most one.
\end{theorem}

An implementation of the Zariski decomposition based on Theorem~\ref{thm:zariski-cox} in the \texttt{Magma} computer algebra system can be found in Appendix~\ref{app:zdec}.

\begin{theorem}\label{thm:ZD on Fano threefolds}
Let $X$ be a smooth Fano threefold. Then every effective divisor has a CKM Zariski decomposition on $X$.
\end{theorem}
\begin{proof}
    This is a direct consequence of Theorem~\ref{thm:zariski-cox} and the fact that Fano threefolds are Mori dream spaces.
\end{proof}

\begin{remark}
 Here is a sketch of an  alternative argument for Theorem~\ref{thm:ZD on Fano threefolds}, which has a more geometric flavour, and explains the special nature of Fano threefolds that is relevant for us. Instead of the cohomological generalization of Zariski decomposition to higher dimensions, Nakayama in \cite{Nakayama_ZD} introduced a numerical variant (called $\sigma$-decomposition) decomposing pseudo-effective divisors into effective and movable ones.

 On  Fano varieties movable divisors are strongly movable (that is, they have stable base loci of codimension at least two). Moreover, Fano threefolds do not possess small modifications, therefore strongly movable divisors are automatically nef. Hence, Nakayama's $\sigma$-decomposition yields a CKM-style Zariski decomposition on the Fano threefold in question (not just over it).
\end{remark}

\begin{remark}
Analogously to the surface case, regions where the support of the negative part of the CKM Zariski decomposition is constant give rise to a chamber structure, although the situation is more complicated in higher dimensions (partially for the non-numerical nature of CKM-style Zariski decomposition). Whenever there is sufficient information about it, the geography of proper birational models determines the variation of Zariski decomposition. This is certainly the case for Mori dream spaces, or whenever certain local Cox rings are finitely generated. For a model statement cf. \cite[Theorem 5.4]{KKL12}.
\end{remark}

\subsection{Volumes of divisors}

The volume of a Cartier divisor or a line bundle is the asymptotic rate of growth of its global sections; at the same time, this notion yields something that works as a replacement for both the self-intersection and the dimension of the space of global sections at the same time. In this section, we work over an algebraically closed field of characteristic zero. For further information,  the reader is invited to consult \cite{PAGI,GeomAsp}.

\begin{definition}[Volume of an integral divisor]
Let $X$ be a projective variety of dimension $d$, $L$ an integral Cartier divisor on $X$. We define the \emph{volume of $L$} as
\[
\vvol{X}{L} \deq \limsup_{m} \frac{h^0(X,mL)}{m^d/d!}\ .
\]
\end{definition}

\begin{remark}
By its very definition $\vvol{X}{L}\geqslant 0$. At the same time, since $X$ is projective, we will also have $\vvol{X}{L} < +\infty$. Although it is far from clear from the definition (and nontrivial to prove, see \cite{LM} or \cite[Example 11.4.7]{PAGI}), the upper limit in the definition of the volume is in fact a limit. 
\end{remark}

\begin{example}[Volume of divisors on curves]
Let $C$ be a smooth projective curve, $L$ a $\Q$-effective divisor on $C$. By the Riemann--Roch theorem
\[
h^0(C,mL) - h^1(C,mL) = \deg_C(mL) + 1 - g(C)\ ,
\]
in particular
\[
\frac{h^0(C,mL)}{m} \rightarrow  \deg_C(L) =  \vvol{C}{L}\ .
\]
If $L$ is not $\Q$-effective then $h^0(C,mL)=0$ for all $m\geqslant 1$ hence $\vvol{C}{L}=0$.
\end{example}

\begin{example}[Volume on projective spaces]
Let $X=\PP^d$, and $L$ a Cartier divisor from the hyperplane class. A quick computation shows that $\vvol{X}{L}=1$, this is in fact the reason for the normalization factor in the definition. 	
\end{example}

\begin{example}[Volume of a nef divisor]\label{ex:volume nef divisor}
Let $X$ be a projective variety of dimension $d$, $L$ a nef Cartier divisor on $X$. Then there exists a positive constant $\gamma$ such that
\[
h^i(X,mL) \leqslant \gamma \cdot m^{d-1}
\]
for all $1\leqslant i\leqslant d$ and all $m\geqslant 0$. Consequently,
\[
\vvol{X}{L} = \limsup_{m} \frac{h^0(X,mL)}{m^d/d!} = \limsup_{m} \frac{\chi(X,mL)}{m^d/d!} = (L^d)
\]
by the Asymptotic Riemann--Roch theorem \cite{PAGI}. Non-nef divisors can easily have negative self-intersection, which then is difficult to interpret as a measure of positivity.
\end{example}

\begin{example}[Volumes of divisors on surfaces]
Let $X$ be a smooth projective surface, $L$ a pseudo-effective line bundle on $L$ with Zariski decomposition $L=P_L+N_L$. Then
\[
\vvol{X}{L} \equ \vvol{X}{P_L} \equ (P_L^2)
\]
via Proposition~\ref{prop:ZD positive part}.  In particular, since $P_L$ is a $\Q$-divisor,  the volume of any line bundle on a surface is rational.
\end{example}

\begin{remark}[Finite generation and the rationality of the volume]
The Hilbert--Samuel theorem for graded rings implies that a Cartier divisor with a finitely generated section ring will have rational volume.
\end{remark}

\begin{remark}[Volume and Zariski decomposition]
Let $D$ be an effective integral (or $\Q$-) divisor with a CKM-style Zariski decomposition $D=P+N$; recall that $P$ and $N$ are $\Q$-divisors. It follows from the definition of $P$ that
\[
\vvol{X}{D} \equ \vvol{X}{P} \equ (P^n)\ ,
\]
in particular $\vvol{X}{D}\in \Q$.
\end{remark}

\begin{remark}[Arithmetic properties of volumes]
Even though the volume of a nef Cartier divisor is a non-negative integer, this observation fails for big line bundles in general. More concretely, irrationality of the volume serves as a strong obstruction to the existence of a Zariski decomposition in the sense of CKM.
\begin{enumerate}
	\item Every positive rational number occurs as the volume of a big line bundle on a smooth projective surface.
	\item It has been pointed out by Cutkosky  \cite{Cutkosky,PAGI}  that there exist threefolds with line bundles that have irrational algebraic volumes.
	\item \cite{KLM_volume} showed that there are (plenty of) big line bundles with transcendental volumes.
\end{enumerate}
\end{remark}

Next, we collect the fundamental properties of volumes of line bundles.

\begin{theorem}\label{thm:volume fundamental}
Let $X$ be a projective variety of dimension $d$, $L,L'$  integral Cartier divisors on $X$. Then the following hold.
\begin{enumerate}
	\item $L$ is big iff $\vvol{X}{L}>0$.
	\item If $a\in\N$ then $\vvol{X}{aL} = a^d\cdot \vvol{X}{L}$.
	\item If  $L\equiv L'$ then
	$\vvol{X}{L'} = \vvol{X}{L}$.
	\item (Log-concavity) $\vvol{X}{L+L'}^{1/d} \leqslant \vvol{X}{L}^{1/d} + \vvol{X}{L'}^{1/d}$.
\end{enumerate}
\end{theorem}
\begin{proof}
Proofs to these statements can be found in  \cite{PAGI}, namely Lemma 2.2.3, Proposition 2.2.35, Proposition 2.2.41, and Theorem 11.4.9.
\end{proof}

\begin{remark}
By Theorem~\ref{thm:volume fundamental} (2), the volume extends uniquely to $\Q$-divisors in the following manner: let $D$ be a $\Q$-Cartier divisor $a$ an arbitrary positive integer such that $aD$ is integral. Then
\[
\vvol{X}{D} \deq \frac{\vvol{X}{aD}}{a^d}
\]
is well-defined.

Also, Theorem~\ref{thm:volume fundamental} (3) yields a definition of the volume on $N^1(X)$; the two observations combined give rise to a function
\[
\vvol{X}{} \colon N^1(X)_\Q \longrightarrow \R_{\geqslant 0}\ .
\]
\end{remark}

\begin{theorem}[Continuity of the volume]\label{thm:cont of volume}
Let $X$ be a projective variety. Then the volume of an integral divisor extends continuously to $N^1(X)_\R$. 	
\end{theorem}

\begin{remark}
Theorem~\ref{thm:cont of volume} implies in particular that it makes sense to talk about the volume of an $\R$-divisor.
\end{remark}

\begin{remark}
One can pursue the definition analogous to the volume for the higher cohomology of a divisor. As it turns out,  Theorem~\ref{thm:volume fundamental} (1)-(3) and Theorem~\ref{thm:cont of volume} all continue to hold \cite{Kuronya_ACF, dFKL}.
\end{remark}

\begin{remark}
One can ask the question how the volume behaves as a global function on $N^1(X)_\R$. Using Example~\ref{ex:volume nef divisor} we see that
the volume equals the self-intersection form when restricted to the nef cone of $X$, outside the big cone however, the volume is identically zero. Since $\vvol{X}{}$ is continuous over the whole real N\'eron--Severi space, the best that one can hope for is that it exhibits a piecewise polynomial behaviour with respect to some chamber decomposition.
\end{remark}

\begin{example}[Birational behaviour of volumes]
Let $X$ be a projective variety of dimension $d$, let $L$ be an ample line bundle on $X$, and let $x\in X$ be a (closed) non-singular point. Write $\pi\colon Y\to X$ for the blowing-up of $X$ at $x$ with exceptional divisor $E$.

To begin with,  the volume is a birational invariant, more specifically,
\[
\vvol{Y}{\pi^*L} = \vvol{X}{L}\ .
\]
Next,  for positive integers $m,k\geqslant 1$  we have
\[
H^0(X,\cO_X(mL)\otimes \cI_{x/X}^k) \simeq H^0(Y,\cO_Y(m\pi^*L-kE))\ .
\]
As a consequence, the number $\vvol{Y}{\pi^*L-tE}$ can be thought of as the asymptotic measure of the global sections of $mL$ vanishing at $x$ to order at least $mt$, and the ratio
\[
\frac{\vvol{Y}{\pi^*L-tE}}{\vvol{X}{L}}
\]
represents the proportion of such global sections (again, in an asymptotic sense).
\end{example}

The following result yields a clean way to argue for the piecewise polynomial nature of the volume function under suitable finiteness conditions. For further details on notation, definitions  and proofs we refer the reader to \cite{KKL12}.

\begin{definition}
Let $X$ be a $\Q$-factorial projective variety. A $\Q$-divisor $D$ is called \emph{gen} if for all $\Q$-Cartier $\Q$-divisors $D'$ on $X$ numerically equivalent to $D$ the section ring $R(X,D')$ is finitely generated.
\end{definition}

\begin{remark}
The significance of gen divisors comes from the following situation. Let $D$ and $D'$ be numerically equivalent divisors, if both of them have finitely generated section rings then $\Proj R(X,D) \simeq \Proj R(X,D')$, but otherwise it is not guaranteed (see \cite[Lemma 3.11]{KKL12}).  Being gen guarantees the $\Q$-factoriality of the ample model of $(X,D)$ (cf. \cite[Lemma 4.6]{KKL12}). The main sources of gen divisors are ample divisors, adjoint divisors, or in fact any divisor on a variety with $\Pic(X)_\Q\simeq N^1(X)_\Q$, like Mori dream spaces.
\end{remark}

\begin{theorem}[\cite{KKL12}, Theorem 5.4]\label{thm:KKL optimal models}
Let $X$ be a projective $\Q$-factorial variety, $\mathcal C\subseteq \Div_\R(X)$ a rational polyhedral cone. Write $\pi\colon \Div_\R(X)\to N^1(X)_\R$ for the natural projection. Assume that
\begin{enumerate}
    \item the local Cox ring  $\mathfrak R = R(X,\mathcal C)$ is finitely generated,
    \item $\Supp \mathfrak R$ contains an ample divisor,
    \item $\pi(\Supp \mathfrak R)$ spans $N^1(X)_\R$,
    \item every divisor in the interior of $\Supp \mathfrak R$ is gen.
\end{enumerate}
Under these conditions we can run a $D$-MMP for any $\Q$-divisor $D\in \mathcal C$ which terminates.
Furthermore, there exists  a finite decomposition
\[
\Supp \mathfrak R \equ \coprod \mathcal{N}_i
\]
into convex cones $\mathcal{N}_i$ satisfying the following properties.
\begin{enumerate}
    \item Each $\overline{\mathcal{N}_i}$ is a rational polyhedral cone,
    \item for every index $i$, there exists a $\Q$-factorial projective variety $X_i$ and a birational contraction $\phi_i\colon X\dashrightarrow X_i$ such that $\phi_i$ is an optimal model \textup{(cf. \cite[Definition 2.3(2)]{KKL12})} for every divisor in $\mathcal{N}_i$,
    \item every element of the cone $(\phi_i)_*\mathcal{N}_i$ is semiample.
\end{enumerate}
\end{theorem}

\begin{theorem}\label{thm:volume pw poly}
Let $X$ be a projective $\Q$-factorial variety. Under the conditions of Theorem~\ref{thm:KKL optimal models}, the function $\vvol{X}{\,}$ restricted to the cone $\mathcal{N}_i$ is given by a homogeneous polynomial of degree $\dim\,X$.

Consequently, the volume function restricted to $\Supp \mathfrak R$ is pieceweise polynomial (homogeneous of degree equal $\dim\,X$ with respect to a finite rational polyhedral decomposition).
\end{theorem}
\begin{proof}
    Since the volume of a divisor is a birational invariant, the first statement follows from Theorem~\ref{thm:KKL optimal models} and the fact that pulling back induces a homogeneous linear map between the appropriate real N\'eron--Severi spaces. We obtain the second claim by considering the maximal-dimensional chambers under the $\mathcal{N}_i$'s.
\end{proof}

\begin{corollary}\label{thm:MDS volume pw poly}
Let $X$ be a Mori dream space. Then the volume function on $N^1(X)_\R$ is piecewise polynomial (homogeneous of degree $\dim X$) with respect to the Mori chamber decomposition of the effective cone of $X$.
\end{corollary}
\begin{proof}
The conditions of Theorem~\ref{thm:volume pw poly} are quickly verified, the  cone $\mathcal{C}$ is the whole effective cone, the graded ring $R(X,\mathcal{C})$ is the Cox ring of $X$, which is finitely generated.
\end{proof}

\begin{remark}
In the case of Mori dream spaces  the statement already follows from Hu--Keel via the same argument as in Theorem~\ref{thm:volume pw poly}.
\end{remark}

\section{Automorphism groups of Fano's last Fanos}
\label{section:Aut}

Let $Y \subseteq \mathbb{P}^5$ be the smooth complete intersection of
two  quadrics $Q_1$ and $Q_2$ that intersect each other along a conic
$C$ on a plane $\Pi \subseteq \mathbb{P}^5$.
Denote by $X$ the strict transform of $Y$ via the blowup $\mathrm{Bl}_{\Pi}\mathbb{P}^5\to \mathbb{P}^5$  along the plane $\Pi$. We denote by ${\rm Aut}(X)$ and ${\rm Aut}(Y)$ the automorphism groups of the Fano varieties $X$ and $Y$, respectively.

Since
\[
{\rm Aut}(X) \simeq \{ \rho \in {\rm Aut}(Y)  :  \rho(C) = C \},
\]
it suffices to study the automorphisms of $Y$ that fix the conic $C$, in order to classify the automorphism groups of the smooth Fano $3$-folds in Family~\textnumero 2.16. Furthermore, every automorphism of $Y$ comes from an automorphism of $\mathbb{P}^5$ that leaves $Y$ fixed. Denote by $\rm{Lin}(Y)$ the subgroup of  automorphisms of $\mathbb{P}^5$ that leaves $Y$ fixed. We may identify  $\rm{Aut}(X)$ with the subgroup of $\rm{Lin}(Y)$ that consists of automorphisms that fix $\Pi$, i.e.,
\[
{\rm Aut}(X) =\{ \tau \in \rm{Lin}(Y) : \tau(\Pi)=\Pi \}.
\]

By~\cite[Lemma 8.6.1]{do}, up to projective equivalence, the ideal of $Y$ is generated by the two diagonal quadrics given by the equations
\[
  \sum_{i=0}^5 a_i x_i^2 = 0, \qquad \sum_{i=0}^5 x_i^2 = 0,
\]
where all the $a_i$ are distinct (ensuring that $Y$ is smooth).
We will refer to these two quadrics as $Q_1$ and $Q_2$ from now on.
Let $\mathcal{P}$ be the pencil of quadrics generated by $Q_1$ and $Q_2$ in $\mathbb{P}^5$.
Since $Y$ is smooth, the pencil $\mathcal{P}$ has a unique quadric $Q$ that contains the plane $\Pi$.
Because the quadric $Q$ can be written as $\alpha Q_1+\beta Q_2$, $[\alpha:\beta]\in\mathbb{P}^1$, a further
diagonal scaling $x_i\mapsto\lambda_i x_i$ (which keeps the pencil
unchanged) can be chosen so that $Q$ becomes a Fermat-type quadric of
rank~$6$ if it is smooth, or of rank~$5$ when it is singular.  Re-ordering coordinates if necessary we
may therefore assume
\[
Q_1 : \; \sum_{i=0}^{5} a_i x_i^{2}=0,\qquad
Q \;=\; Q_2 : \left\{\aligned
&\; x_0^{2}+x_1^{2}+x_2^{2}+x_3^{2}+x_4^{2}=0\  \text{if $Q$ is singular}, \\
&\; x_0^{2}+x_1^{2}+x_2^{2}+x_3^{2}+x_4^{2}+x_5^{2}=0\  \text{if $Q$ is smooth},
\endaligned
\right.
\]
where $a_0,a_1,\cdots, a_5$ are still pairwise distinct if $Q$ is smooth,  and $a_0,a_1,\cdots, a_4$ are  pairwise distinct with $a_5\ne0$ if $Q$ is singular.

The pencil $\mathcal{P}$ contains exactly six singular quadrics and they are cones over smooth quadric $3$-folds. The vertices of the six singular quadrics are the points represented by the following standard row vectors:
\begin{equation}\label{eq:st}
\begin{aligned}
e_0:=(1,0,0,0,0,0), \  e_1:=(0,1,0,0,0,0),  \ e_2:=(0,0,1,0,0,0), \\
e_3:=(0,0,0,1,0,0), \  e_4:=(0,0,0,0,1,0),  \ e_5:=(0,0,0,0,0,1).
\end{aligned}
\end{equation}

An automorphism $\tau \in {\rm Aut}(X)$ must preserve the pencil $\mathcal{P}$ and, in particular, the vertices $e_0, \dots, e_5$ of the six singular quadrics of the pencil.
Therefore, we can write
\begin{equation}\label{eq:Aut-decompostion}
\tau = \sigma \circ \eta,
\end{equation}
 where $\sigma$ is an automorphism given by a permutation matrix, and $\eta$ is an automorphism given by a non-singular diagonal matrix.
Moreover, $\tau$ must preserve the Fermat-type quadric $Q$.

We adopt the following standard notation for finite groups throughout this section:
\begin{itemize}
\item[-] $C_m$ is the cyclic group of order $m$;
\item[-] $\mathrm{D}_m$ is the dihedral  group of order $2m$;
\item[-] $\mathfrak{A}_m$ is the alternating group of degree $m$;
\item[-] $\mathfrak{S}_m$ is the symmetric  group of degree $m$.
\end{itemize}

\subsection{Diagonal part}

If $Q$ is smooth, then it immediately follows that $\eta$ in \eqref{eq:Aut-decompostion} is a diagonal matrix with diagonal entries $\pm 1$
because~$\tau$ must preserve the Fermat quadric $Q$ of rank $6$.

In the case when $Q$ is singular, $\tau$ preserves the Fermat quadric $Q$ of rank $5$. Therefore, the first five diagonal entries of $\eta$ are $\pm 1$ and
the permutation part $\sigma$ preserves the vertex given by $e_5$.
Determining the sixth entry of $\eta$ requires an additional argument. To this end, we first introduce the following lemma.

\begin{lemma}\label{lemma:Skew}
Let $B=(b_i-b_j)_{0\leq i,j\leq 4}$ be a $5\times 5$ skew-symmetric
matrix, where $b_0,\dots,b_4$ are pairwise distinct complex numbers. Let
$\nu\in\mathfrak S_5$, and let $P_\nu$ be the corresponding permutation
matrix, with the convention $P_\nu e_i=e_{\nu(i)}$. Suppose that there exists
a nonzero scalar $c$ such that
\[
P_\nu^tBP_\nu=cB.
\]
Then the only possible cycle types are:
\begin{itemize}
    \item $\nu=\mathrm{Id}$ and $c=1$;
    \item $\nu=(01)(23)$, up to conjugation, and $c=-1$;
    \item $\nu=(0123)$, up to conjugation, and $c=\zeta_4$, a primitive
    $4$-th root of unity;
    \item $\nu=(01234)$, up to conjugation, and $c=\zeta_5$, a primitive
    $5$-th root of unity.
\end{itemize}
\end{lemma}

\begin{proof}
The assumption gives
\[
b_{\nu(i)}-b_{\nu(j)}=c(b_i-b_j)
\]
for all $i,j$. Equivalently, there exists a constant $r$ such that
\[
b_{\nu(i)}=cb_i+r
\]
for all $i$.

First suppose that $c=1$. Then $b_{\nu(i)}=b_i+r$ for all $i$. If $N$ is
the order of $\nu$, applying this relation $N$ times gives
\[
b_i=b_{\nu^N(i)}=b_i+Nr.
\]
Thus $r=0$, and hence $b_{\nu(i)}=b_i$ for all $i$. Since the $b_i$ are
pairwise distinct, we get $\nu=\mathrm{Id}$.

Assume now that $c\neq 1$. Suppose that $\nu$ has a cycle of length
$m\geq 2$. Iterating the relation along the cycle gives
\[
b_{\nu^k(i)}
=
c^k b_i+\left(\sum_{j=0}^{k-1}c^j\right)r.
\]
Since $\nu^m(i)=i$, the same computation applied to every element of the
cycle gives, for all $h=0,\dots,m-1$,
\[
(c^m-1)b_{\nu^h(i)}
+
\left(\sum_{j=0}^{m-1}c^j\right)r=0.
\]
Subtracting two such equations and using the fact that the corresponding
$b$'s are distinct, we obtain $c^m=1$.
Moreover, $c$ is a primitive $m$-th root of unity. Indeed, if $c^k=1$ for
some $0<k<m$, then
\[
1+c+\cdots+c^{k-1}=0,
\]
and therefore $b_{\nu^k(i)}=b_i$. Since the $b_i$ are pairwise distinct, this
implies $\nu^k(i)=i$, contradicting the fact that $i$ belongs to a cycle of
length $m$.
Thus all nontrivial cycles of $\nu$ have the same length. Also, $\nu$ has at
most one fixed point: if $\nu(i)=i$ and $\nu(j)=j$ with $i\neq j$, then
\[
b_i-b_j=c(b_i-b_j),
\]
hence $c=1$, a contradiction.
The possible cycle types in $\mathfrak S_5$ are
\[
1^5,\quad 2+1+1+1,\quad 2+2+1,\quad 3+1+1,\quad 3+2,\quad 4+1,\quad 5.
\]
The types $2+1+1+1$ and $3+1+1$ have at least two fixed points, while the
type $3+2$ has nontrivial cycles of different lengths. Hence the only
nontrivial possibilities are
\[
2+2+1,\qquad 4+1,\qquad 5.
\]
The corresponding values of $c$ are respectively $-1$, a primitive fourth
root of unity, and a primitive fifth root of unity. This proves the claim.
\end{proof}

We derive the sixth diagonal entry of $\eta$ and the permutation matrix $\sigma$ from Lemma~\ref{lemma:Skew} as follows.

\begin{proposition}\label{proposition:singularP}
Suppose that the quadric $Q$ is singular. Then, in \eqref{eq:Aut-decompostion}, only the following four cases can occur:
\begin{enumerate}
\item the sixth diagonal entry of $\eta$ is $\pm 1$ and  $\sigma$ is the identity.
\item the sixth diagonal entry of $\eta$ is $\zeta_4$, a  primitive $4$-th root of unity, and $\sigma$ has two cycles of length $2$.
\item the sixth diagonal entry of $\eta$ is $ \zeta_8$, a  primitive $8$-th root of unity, and $\sigma$ has a cycle of length $4$.
\item  the sixth diagonal entry of $\eta$ is $\zeta_{10}$, a $10$-th root of unity with $\zeta_{10}\ne \pm 1$, and $\sigma$ has a cycle of length $5$.
\end{enumerate}
\end{proposition}
\begin{proof}

The automorphism $\tau$ preserves all six singular quadrics in the pencil, which are given by:
\begin{equation}
\begin{aligned}
\phantom{(a_0-a_0)}x_0^{2}+\phantom{(a_1-a_0)}x_1^{2}+\phantom{(a_2-a_0)}x_2^{2}+
\phantom{(a_3-a_0)}x_3^{2}+\phantom{(a_4-a_0)}x_4^{2}\phantom{+a_{5}x_5^2}\ =0;\\
(a_0-a_0)x_0^{2}+(a_1-a_0)x_1^{2}+(a_2-a_0)x_2^{2}+(a_3-a_0)x_3^{2}+(a_4-a_0)x_4^{2}+a_5x_5^2=0;\\
(a_0-a_1)x_0^{2}+(a_1-a_1)x_1^{2}+(a_2-a_1)x_2^{2}+(a_3-a_1)x_3^{2}+(a_4-a_1)x_4^{2}+a_5x_5^2=0;\\
(a_0-a_2)x_0^{2}+(a_1-a_2)x_1^{2}+(a_2-a_2)x_2^{2}+(a_3-a_2)x_3^{2}+(a_4-a_2)x_4^{2}+a_5x_5^2=0;\\
(a_0-a_3)x_0^{2}+(a_1-a_3)x_1^{2}+(a_2-a_3)x_2^{2}+(a_3-a_3)x_3^{2}+(a_4-a_3)x_4^{2}+a_5x_5^2=0;\\
(a_0-a_4)x_0^{2}+(a_1-a_4)x_1^{2}+(a_2-a_4)x_2^{2}+(a_3-a_4)x_3^{2}+(a_4-a_4)x_4^{2}+a_5x_5^2=0.\\
\end{aligned}
\end{equation}
The last five quadrics determine a unique skew-symmetric $5 \times 5$ matrix
\[\left(\frac{a_i-a_j}{a_5}\right)_{0\leq i,j\leq 4}\]
whose off-diagonal entries are all nonzero. Since this matrix must be preserved by $\tau$, Lemma~\ref{lemma:Skew} completes the proof.
\end{proof}

\subsection{Permutation part}
We now investigate which permutations, up to conjugations, in $\mathfrak{S}_6$ can appear in ${\rm Aut}(X)$ when the quadric $Q$ is smooth.

Note that the possible types of permutations for~$\sigma$ in \eqref{eq:Aut-decompostion}, up to conjugation, are:
\[
 (01), (012), (0123), (01234), (012345), (01)(23), (01)(234), (01)(2345), (012)(345), (01)(23)(45).
\]
in $\mathfrak{S}_6$.

\begin{proposition}\label{lemma:invariant-pencil}
Suppose that the quadric $Q$ is smooth.
A subgroup of $\mathfrak{S}_6$, up to conjugation,
that preserve the pencil $\mathcal{P}$ is one of  the following
cyclic groups
\[  \langle \rm{Id}\rangle,
 \quad
   \langle (01)(23)(45)\rangle,
   \quad
   \langle(012)(345)\rangle,
   \quad
   \langle(01234)\rangle,
   \quad
   \langle(012345)\rangle.
\]
For each  non-trivial subgroup $\mathfrak{G}$ in the list, the pencil $\mathcal{P}$ is $\mathfrak{G}$-invariant if and only if  the distinct
$a_i$'s satisfy the following relations:
\begin{itemize}
    \item \textbf{Case} $(01)(23)(45)$:
{\footnotesize
\[a_0 + a_1 - a_4 - a_5
       =a_2 + a_3 - a_4 - a_5=0.\]
       }
    \item \textbf{Case} $(012)(345)$:
    {\footnotesize
    \[a_0 - \zeta_3^2 a_2 + \zeta_3^2 a_4 - a_5=a_1 - \zeta_3 a_2 - a_4 + \zeta_3 a_5=a_3 + \zeta_3^2 a_4 + \zeta_3 a_5=0, \mbox{ or}\]
    \[a_0 - \zeta_3 a_2 + \zeta_3 a_4 - a_5=a_1 - \zeta_3^2 a_2 - a_4 + \zeta_3^2 a_5= a_3 + \zeta_3 a_4 + \zeta_3^2 a_5=0.  \phantom{\mbox{ or}}\]
    }
    \item \textbf{Case} $(01234)$:
 {\footnotesize
\[a_0 -\zeta_5^4a_4
  + (\zeta_5^4 - 1)a_5 = a_1 - \zeta_5^3a_4 + (\zeta_5^3 - 1)a_5 = a_2 - \zeta_5^2a_4 + (\zeta_5^2 - 1)a_5 = a_3 - \zeta_5a_4 + (\zeta_5 - 1)a_5 = 0,  \phantom{\mbox{ or}}\]
  \[a_0 - \zeta_5^3a_4 + (\zeta_5^3 - 1)a_5 = a_1 - \zeta_5a_4 + (\zeta_5 - 1)a_5 = a_2 -\zeta_5^4a_4
  + (\zeta_5^4 - 1)a_5 = a_3 - \zeta_5^2a_4 + (\zeta_5^2 - 1)a_5 = 0,  \phantom{\mbox{ or}}\]
  \[a_0 - \zeta_5^2a_4 + (\zeta_5^2 - 1)a_5 = a_1 -\zeta_5^4a_4
  + (\zeta_5^4 - 1)a_5 = a_2 - \zeta_5a_4 + (\zeta_5 - 1)a_5 = a_3 - \zeta_5^3a_4 + (\zeta_5^3 - 1)a_5 = 0,  \mbox{ or}
\]
\[a_0 - \zeta_5a_4 + (\zeta_5 - 1)a_5 = a_1 - \zeta_5^2a_4 + (\zeta_5^2 - 1)a_5 = a_2 - \zeta_5^3a_4 + (\zeta_5^3 - 1)a_5 =a_3 -\zeta_5^4a_4 + (\zeta_5^4 - 1)a_5 = 0.  \phantom{\mbox{ or}}
\]
  }
    \item \textbf{Case} $(012345)$:
{\footnotesize
\[a_0 -\zeta_3^2a_4 + (\zeta_3^2-1)a_5 = a_1 + (2\zeta_3 + 1)a_4 + 2\zeta_3^2a_5 = a_2 + 2\zeta_3 a_4 - (2\zeta_3 + 1)a_5 = a_3 + (\zeta_3 - 1)a_4 - \zeta_3 a_5 = 0,  \mbox{ or}\]
\[a_0 - \zeta_3 a_4 + (\zeta_3 - 1)a_5 = a_1 - (2\zeta_3 + 1)a_4 + 2\zeta_3 a_5 = a_2 + 2\zeta_3^2a_4 + (2\zeta_3 + 1)a_5 = a_3 + (\zeta_3^2-1)a_4 -\zeta_3^2a_5 = 0.  \phantom{\mbox{ or}}\]
}
\end{itemize}
Here $\zeta_m$ is a primitive $m$-th root of unity.
\end{proposition}

\begin{proof} Let $\sigma\in \mathfrak{S}_6$ be a permutation in $\{0,1,2,3,4,5\}$. The pencil $\mathcal{P}$ is $\sigma$-invariant if and only if  the $3 \times 6$ matrix
\begin{equation}\label{eq:matrix}
  \begin{pmatrix}
    1 & 1 & 1 & 1 & 1 & 1 \\
    a_0 & a_1 & a_2 & a_3 & a_4 & a_5 \\
    a_{\sigma^{-1}(0)} & a_{\sigma^{-1}(1)}  & a_{\sigma^{-1}(2)}  & a_{\sigma^{-1}(3)}  & a_{\sigma^{-1}(4)}  & a_{\sigma^{-1}(5)}
  \end{pmatrix}
\end{equation}
has rank $2$. This gives a subvariety $\Xi_\sigma$ of $\mathbb{P}^5$ defined by the equations in $[a_0:a_1:a_2:a_3:a_4:a_5]$ derived from the determinants of all the $3 \times 3$ minors of the above matrix. For a subgroup $\mathfrak{G}$ of $\mathfrak{S}_6$, we define
\[\Xi_{\mathfrak{G}}=\bigcap_{\sigma\in\mathfrak{G}}\Xi_\sigma.\]
Let $\Delta$ be the union of $15$ hyperplanes in $\mathbb{P}^5$ defined by the equations $a_i=a_j$. Then  each point
$[a_0:a_1:a_2:a_3:a_4:a_5]$ in the affine variety $\Xi_{\mathfrak{G}}\setminus \Delta$ characterizes a $\mathfrak{G}$-invariant pencil
generated by quadrics $\sum x_i^2$ and $\sum a_ix_i^2$ whose base locus is smooth.

The \texttt{Magma} code \texttt{Permutation Scheme}
in Appendix~\ref{app:aut}
implements these conditions to construct the affine variety $\Xi_{\mathfrak{G}}\setminus\Delta$.
It shows that, up to conjugation, only the cyclic groups generated by the following permutations yield a non-empty set $\Xi_{\mathfrak{G}}\setminus\Delta$: the identity $\mathrm{Id}$, $(01)(23)(45)$, $(012)(345)$, $(01234)$, and $(012345)$.

The second assertion is established by the \texttt{Magma} code.
\end{proof}

We next  consider the same question in the case when the quadric $Q$ is singular.

\begin{proposition}\label{prop:invariant-pencil-singular-lambda}
Suppose $Q$ is the singular Fermat quadric in $\mathbb{P}^{5}$.
Consider transformations of the form
\[
   (x_{0},x_{1},x_{2},x_{3},x_{4},x_{5})
   \;\longmapsto\;
   (x_{\sigma(0)},x_{\sigma(1)},x_{\sigma(2)},x_{\sigma(3)},x_{\sigma(4)},\lambda\,x_{5}),
\]
where $\sigma\in\mathfrak{S}_{6}$ fixes the sixth coordinate
($\sigma(5)=5$) and $\lambda\in\mathbb{C}^{\times}$.
Up to conjugation in $\mathfrak S_{5}$, the subgroups generated by such
transformations that preserve \emph{both} the pencil $\mathcal{P}$ and
the quadric $Q$ are the following cyclic groups:
\[
\begin{array}{ccl}
\langle\mathrm{Id}\rangle
    &\text{:}& \lambda^2=1,\\
\langle(01)(23)\rangle
    &\text{:}& \lambda^2=-1,\\
\langle(0123)\rangle
    &\text{:}& \lambda^2 \text{ is a primitive $4$-th root of unity},\\
\langle(01234)\rangle
    &\text{:}& \lambda^2 \text{ is a primitive $5$-th root of unity}.
\end{array}
\]
\end{proposition}

\begin{proof}
By a slight abuse of notation, we denote by $\mathfrak S_{5}$ the subgroup of $\mathfrak S_{6}$ consisting of
permutations that fix the last element, i.e., those $\sigma\in\mathfrak S_{6}$ such that $\sigma(5)=5$.

Let $\mathfrak G\subset \mathfrak S_{5}$ be a subgroup, and let
$\sigma_{1},\dots,\sigma_{r}$ be generators of $\mathfrak G$.
For each $j=1,\dots,r$, we impose the condition that the pencil $\mathcal P$
is preserved by $\sigma_j$ after scaling $x_5$ by a scalar $\lambda_j$.
This is equivalent to requiring that the matrix
\[
  \begin{pmatrix}
    1 & 1 & 1 & 1 & 1 & 0\\
    a_{0} & a_{1} & a_{2} & a_{3} & a_{4} & a_{5}\\
    a_{\sigma^{-1}_j(0)} & a_{\sigma^{-1}_j(1)} & a_{\sigma^{-1}_j(2)} &
    a_{\sigma^{-1}_j(3)} & a_{\sigma^{-1}_j(4)} & \lambda_j^2 a_{5}
  \end{pmatrix}
\]
has rank $2$.
Treating the coefficients $[a_{0}:\dots:a_{5}]$ as homogeneous
coordinates on $\mathbb P^{5}$ and the scalars
$\lambda_1,\dots,\lambda_r$ as affine coordinates on $\mathbb A^{r}$,
these minors cut out a subvariety
\[
  \Xi_{\mathfrak G}\subset \mathbb P^{5}\times \mathbb A^{r}.
\]
To discard degeneracies (coincident coefficients or vanishing
coordinates) we remove the divisor
\[
  \Delta=\bigcup_{0\leq i<j\leq 4}\{a_i=a_j\}\;\cup\;\{a_5=0\}.
\]
Thus $\Xi_{\mathfrak G}\setminus\Delta$ parameterises exactly the
non-trivial pencils that are preserved by every generator of
$\mathfrak G$, together with the corresponding dilations of $x_5$.

The \texttt{Magma} program \texttt{Permutation Scheme Singular}
in Appendix~\ref{app:aut} shows that $\Xi_{\mathfrak G}\setminus\Delta$
is non-empty if and only if $\mathfrak G$ is conjugate in $\mathfrak S_{5}$ to
\[
\langle \mathrm{Id}\rangle,\qquad
\langle(01)(23)\rangle,\qquad
\langle(0123)\rangle,\qquad
\langle(01234)\rangle.
\]
This yields precisely the four cyclic groups listed above.
\end{proof}

\subsection{Automorphism groups}
We adopt the following notation. Given a permutation $\sigma$ of the set $\{0,1,2,3,4,5\}$, we write
\[
[\pm \sigma(0),\pm \sigma(1),\pm \sigma(2),\pm \sigma(3),\pm \sigma(4),\pm \sigma(5)]
\]
to denote the $6\times 6$ matrix whose $i$th row is the row vector $\pm e_{\sigma(i)}$, where each sign is chosen accordingly and $e_j$ is as in \eqref{eq:st}. More generally, in the singular case we use the same notation for monomial matrices whose nonzero entries are roots of unity. For instance, if $\xi\in\mathbb C^\times$, then
\[
[\pm \sigma(0),\pm \sigma(1),\pm \sigma(2),\pm \sigma(3),\pm \sigma(4),\pm \xi\,\sigma(5)]
\]
denotes the $6\times 6$ matrix whose first five rows are the row vectors $\pm e_{\sigma(i)}$ and whose last row is $\pm \xi\, e_{\sigma(5)}$.
For permutations
$\sigma_1,\cdots, \sigma_r$ in $\{0,1,2,3,4,5\}$, the notation
{\footnotesize
\[
\|[\pm \sigma_1(0),\pm \sigma_1(1),\pm \sigma_1(2),\pm \sigma_1(3),\pm \sigma_1(4),\pm \sigma_1(5)],
\cdots, [\pm \sigma_r(0),\pm \sigma_r(1),\pm \sigma_r(2),\pm \sigma_r(3),\pm \sigma_r(4),\pm \sigma_r(5)] \|
\]}
denotes the group generated by the matrices in $\|\cdots\|$, modulo the subgroup $\langle-I\rangle$, where $I$ is the $6\times 6$ identity matrix. This group acts on~$\mathbb{P}^5$ in the standard way.

We are now ready to describe the automorphism group of the Fano $3$-fold $X$.
\begin{theorem}
\label{thm-aut}
If the  quadric $Q$  is smooth, then the automorphism group $\mathrm{Aut}(X)$ is either
trivial or isomorphic to one of the finite groups listed in the table below.
In each case, $\mathrm{Aut}(X)$ is generated by a
single permutation in $\mathfrak{S}_6$ and by sign
changes.

{\tiny
\begin{center}
\renewcommand{\arraystretch}{1.2}
\begin{longtable}{@{}l|c|c|p{7.8cm}@{}}
\toprule
Permutation $\sigma$ & $[a_0:a_1:a_2:a_3:a_4:a_5]$ & $\mathrm{Aut}(X)$ & Example of invariant plane $\Pi$ \& Group action \\
\midrule
\endfirsthead
\toprule
Permutation $\sigma$ & $[a_0:a_1:a_2:a_3:a_4:a_5]$ & $\mathrm{Aut}(X)$ & Example of invariant plane $\Pi$ \& Group action\\
\midrule
\endhead
\midrule \multicolumn{4}{r}{\textit{Continued on next page}} \\
\endfoot
\bottomrule
\endlastfoot

\multirow{2}{*}{\(\mathrm{Id}\)}
    & \multirow{2}{*}{$[1:2:3:4:5:6]$}
    & \(C_2 \) &
      \(\begin{array}{l}
        i\sqrt{2} x_0+x_2+x_3=i\sqrt{2} x_1+x_2-x_3=x_4+ix_5=0\\[1mm]
        \|[-0,-1,-2,-3,4,5],  -I \|\\[1mm]
        \end{array}\) \\ \cmidrule(lr){3-4}
    &  & \(C_2^2\) &
    \(\begin{array}{l}
    x_0 + ix_3 = x_1 + ix_2=x_4 + ix_5=0\\[1mm]
    \| [-0,1,2,-3,4,5],  [0,-1,-2,3,4,5], -I\|\\[1mm]
    \end{array}\)
\\ \midrule

\multirow{4}{*}{\((01)(23)(45)\)}
    & \multirow{4}{*}{$[2:-1:3:-2:0:1]$}
    & \(C_2\)   &
        \(\begin{array}{l}
     x_0 + x_3 + i\,x_4 - i\,x_5 =2x_2 - 2i\,x_3 + (1-i)\,x_4 - (1+i)\,x_5 = \\[1mm]
    2x_1 - 2i\,x_3 + (1+i)\,x_4 + (i-1)\,x_5 = 0\\[1mm]
    \bigl\|[-1,0,-3,2,-5,4],\; -I\bigr\|\\[1mm]
        \end{array}\) \\ \cmidrule(lr){3-4}
    &  & \(C_2^2\) &
     \(\begin{array}{l}
    i\sqrt{2}\,x_0 + x_1 + x_2
      \;=\;
    x_1 - x_2 + i\sqrt{2}\,x_3
      \;=\;
    x_4 + i\,x_5
      \;=\;0\\[1mm]
    \|[1,-0,-3,2,-5,4],\;
      [1,-0,-3,2,5,-4],\;
      -I\|\\[1mm]
        \end{array}\)
     \\ \cmidrule(lr){3-4}
    &  & \(C_2^3\) &
        \(\begin{array}{l}
        x_0+ix_1= x_2+ix_3=x_4+ix_5=0\\[1mm]
          \|[-1,0,-3,2,-5,4], [-0,-1,2,3,4,5], [0,1,-2,-3,4,5], -I \|\\[1mm]
        \end{array}\) \\ \cmidrule(lr){3-4}
     &  & \(D_4\)   &
        \(\begin{array}{l}
       x_0+ix_1= x_2+ix_4= x_3+ix_5=0\\[1mm]
         \|[-1,0,3,2,5,4], [-0,-1,2,3,4,5], [0,1,-2,3,-4,5], -I \|\\[1mm]
           \end{array}\) \\
        \midrule

\multirow{2}{*}{\((012)(345)\)}
    & \multirow{2}{*}{$[2\zeta_3^2+1 : \zeta_3 : 2 : -\zeta_3 : 0 : 1]$}
    & \(C_3\)   &
     \(\begin{array}{l}
        x_0+\alpha x_3+\beta x_4+\gamma x_5= x_1+\gamma x_3+\alpha x_4+\beta x_5=\\[1mm]
        x_2+\beta x_3+\gamma x_4+\alpha x_5=0\\[1mm]
         \mbox{where } \alpha^2+\beta^2+\gamma^2=-1, \alpha\beta+\beta\gamma+\gamma\alpha=0,\alpha\beta\gamma\ne0\\[1mm]
         \|[1,2,0,4,5,3],  -I \|\\[1mm]
           \end{array}\)     \\ \cmidrule(lr){3-4}
    &  & \(\mathfrak{A}_4\)   &
        \(\begin{array}{l}
        x_0+ix_5=x_1-ix_3= x_2-ix_4=0\\[1mm]
          \|[1,2,0,4,-5,-3], [-0,1,2,3,4,-5], [0,-1,2,-3,4,5], -I \|\\[1mm]
        \end{array}\) \\
        \midrule

\((01234)\)
    & $[\zeta_5^4 : \zeta_5^3 : \zeta_5^2 : \zeta_5 : 1 : 0]$
    & \(C_5\)
    &
        \(\begin{array}{l}
        x_0+\zeta_{20}^6x_3+(-\zeta_{20}^4+\zeta_{20}^2)x_4+\zeta_{20}^3x_5=\\[2mm]
        x_1+(\zeta_{20}^6-\zeta_{20}^4+\zeta_{20}^2)x_3+(\zeta_{20}^2-1)x_4+(-\zeta_{20}^7+\zeta_{20}^5-\zeta_{20}^3)x_5=\\[2mm]
        x_2+(-\zeta_{20}^4+\zeta_{20}^2-1)x_3+\zeta_{20}^4x_4+\zeta_{20}^7x_5=0\\[2mm]
          \|[1,2,3,4,-0,-5], -I\|\\[1mm]
        \end{array}\) \\
        \midrule

\multirow{2}{*}{\((012345)\)}
    & \multirow{2}{*}{$[1-\zeta_3^2 : -2\zeta_3^2 : 2\zeta_3+1 : \zeta_3 : 0 : 1]$}
    & \(C_6\)
    &
        \(\begin{array}{l}
        x_0-\zeta_{12}x_3+(-\zeta_{12}^2+1)x_4+\zeta_{12}^3x_5=x_1+(-\zeta_{12}^2+1)x_3-\zeta_{12}^2x_5=\\[2mm]
        x_2+\zeta_{12}^3x_3-\zeta_{12}^2x_4+(-\zeta_{12}^3+\zeta_{12})x_5=0\\[2mm]
          \|[1,2,3,4,5,-0],  -I \|\\[1mm]
        \end{array}\) \\ \cmidrule(lr){3-4}
    &  & \(C_2\times \mathfrak{A}_4\)
    &
        \(\begin{array}{l}
        x_0+ix_3=  x_1+ix_4= x_2+ix_5=0\\[1mm]
       \|[2,-3,-4,5,0,1], [-0,1,2,-3,4,5], [-3,-4,-5,0,1,2], -I \|\\[1mm]
        \end{array}\)
\end{longtable}
\end{center}
}

If the  quadric $Q$   is singular, then the automorphism group $\mathrm{Aut}(X)$ is either
$C_2$ or isomorphic to one of the finite groups listed in the table below.

{\tiny
\begin{center}
\renewcommand{\arraystretch}{1.2}
\begin{longtable}{@{}l|c|c|p{7.8cm}@{}}
\toprule
Permutation $\sigma$ & $[a_0:a_1:a_2:a_3:a_4:a_5]$ &
$\mathrm{Aut}(X)$ & Example of invariant plane $\Pi$ \& group action \\
\midrule
\endfirsthead
\toprule
Permutation $\sigma$ & $[a_0:a_1:a_2:a_3:a_4:a_5]$ &
$\mathrm{Aut}(X)$ & Example of invariant plane $\Pi$ \& group action \\
\midrule
\endhead
\midrule \multicolumn{4}{r}{\textit{Continued on next page}}\\
\endfoot
\bottomrule
\endlastfoot
\multirow{3}{*}{\rm Id}
    & \multirow{3}{*}{$[1:2:3:4:5:1]$} & $C_2$ &
      $\begin{array}{l}
         x_0 -x_3-\tfrac{i}{\sqrt3}x_4 =
         x_1 +\tfrac{1+i\sqrt5}{2}x_3-\tfrac{i}{\sqrt3}x_4 = \\
         x_2 +\tfrac{1-i\sqrt5}{2}x_3-\tfrac{i}{\sqrt3}x_4 = 0\\[1mm]
         \|[-0,-1,-2,-3,-4,5],\,-I\|
       \end{array}$ \\ \cmidrule(lr){3-4}
    &  & $C_2^{\,2}$ &
      $\begin{array}{l}
         x_0+\tfrac{i}{\sqrt2}x_2 =
         x_1+\tfrac{i}{\sqrt2}x_2 =
         x_4+i\,x_3 = 0\\[1mm]
         \|[-0,-1,-2,3,4,5],\,[0,1,2,-3,-4,5],\,-I\|
       \end{array}$ \\ \cmidrule(lr){3-4}
    & & $C_2^{\,3}$ &
      $\begin{array}{l}
         x_0+i\,x_4 = x_1 = x_2+i\,x_3 = 0\\[1mm]
         \|[-0,1,2,3,-4,5],\,[0,-1,2,3,4,5],\,[0,1,-2,-3,4,5],\,-I\|
       \end{array}$ \\
\midrule
\multirow{4}{*}{$(01)(23)$}
 & \multirow{4}{*}{$[1:-1:2:-2:0:1]$}
 & $C_4$ &
$\begin{array}{l}
  2x_0-i\,x_2-i\,x_3=
  2x_1-i\,x_2-i\,x_3=
  \sqrt2\,x_4-i\,x_2+i\,x_3=0\\[1mm]
  \|[1,0,3,2,-4,i\,5],-I\|
\end{array}$ \\ \cmidrule(lr){3-4}
 & & $C_2\times C_4$ &
$\begin{array}{l}
  x_0-i\,x_4=
  x_1-i\,x_4=
  x_2-i\,x_3=0\\[1mm]
  \|[1,0,-3,2,4,i\,5],[0,1,-2,-3,4,5],-I\|
\end{array}$ \\ \cmidrule(lr){3-4}
 & & $C_2^{\,2}\times C_4$ &
$\begin{array}{l}
  x_0-i\,x_1=
  x_2-i\,x_3=
  x_4=0\\[1mm]
  \|[-1,0,-3,2,4,i\,5],[-0,-1,2,3,4,5],[0,1,-2,-3,4,5],-I\|
\end{array}$ \\ \cmidrule(lr){3-4}
 & & $C_2^{\,2}\rtimes C_4$ &
$\begin{array}{l}
  x_0-i\,x_2=
  x_1-i\,x_3=
  x_4=0\\[1mm]
  \|[1,0,3,2,4,i\,5],[-0,1,-2,3,4,5],[0,-1,2,-3,4,5],-I\|
\end{array}$ \\

\midrule
\multirow{3}{*}{$(0123)$}
 & \multirow{3}{*}{$[1:-i:-1:i:0:1]$}
 & $C_8$ &
$\begin{array}{l}
  x_0 + \zeta_8^2x_3 + \dfrac{-\zeta_8^2-1}{2}\,x_4=0,\\[1mm]
  x_1 + x_3 + \zeta_8^2x_4=0,\\[1mm]
  x_2 - \zeta_8^2x_3 + \dfrac{-\zeta_8^2+1}{2}\,x_4=0\\[1mm]
  \|[3,0,1,2,-4,-\zeta_8\,5],-I\|
\end{array}$ \\ \cmidrule(lr){3-4}
 & & $C_2\times C_8$ &
$\begin{array}{l}
  x_0 + x_2 + (\zeta_8^3+\zeta_8)x_3=0,\\[1mm]
  x_1 + (\zeta_8^3+\zeta_8)x_2 - x_3=0,\\[1mm]
  x_4=0\\[1mm]
  \|[-3,0,1,2,4,\zeta_8\,5],[-0,-1,-2,-3,4,5],-I\|
\end{array}$ \\ \cmidrule(lr){3-4}
 & & $(C_2)^{\,2}\rtimes C_8$ &
$\begin{array}{l}
  x_0+\zeta_8^2x_2=
  x_1+\zeta_8^2x_3=
  x_4=0\\[1mm]
  \|[-3,0,1,2,4,\zeta_8\,5],[-0,1,-2,3,4,5],[0,-1,2,-3,4,5],-I\|
\end{array}$ \\
\midrule
$(01234)$
    & $[\zeta_5^3+1 : \zeta_5^3+\zeta_5+1 : -\zeta_5^2 : 0 : 1 : 1]$
    & $C_{10}$
    & $\begin{array}{l}
 \begin{array}{l}
  x_0 - \zeta_{5}^3 x_3 -(\zeta_{5}^2 + \zeta_{5})\,x_4 = \\
  x_1 - (\zeta_{5}^3 + \zeta_{5}^2 + \zeta_{5})\,x_3 - (\zeta_{5} +1)\,x_4 = \\
  x_2 - (\zeta_{5}^2 + \zeta_{5} + 1)\,x_3 + \zeta_{5}^2 x_4 = 0\\
  \|[-4,0,1,2,3,\zeta_5\,5],-I\|
\end{array}

\end{array}$
\end{longtable}
\end{center}
}
\end{theorem}
Here $\zeta_m$ is a primitive $m$-th root of unity. The last column of each table lists examples of planes
 stabilized by the corresponding automorphism group. These examples are not exhaustive; other stabilized planes may exist.

\begin{proof}
An automorphism $\tau$ of $X$ must also preserve the plane $\Pi$, and consequently, it must preserve the quadric  $Q$, since $Q$ is the only quadric in the pencil containing $\Pi$.

The function \texttt{FindLis} takes a permutation $g$, converts it into its corresponding $6 \times 6$ permutation matrix~$M$, and constructs a group~$H$ generated by $M$ together with diagonal sign-change matrices or appropriate diagonal matrices given by Proposition~\ref{prop:invariant-pencil-singular-lambda}.
It then filters the subgroups of $H$, selecting those that contain both an element whose entrywise absolute value equals $M$, and the matrix $-I$. The resulting list of subgroups is then returned.

The general part of the code defines a multi-projective space $\mathbb{P}^5 \times \mathbb{P}^5 \times \mathbb{P}^5$ and associates to each triple of points the plane they span, viewed as a point in the Grassmannian $\mathbb{G}(2,5)$. It then computes the subvariety $W \subset \mathbb{G}(2,5)$ consisting of planes that are contained in the Fermat quadric~$Q$.

Next, for each element
\[
g \in
\left\{\aligned
&\; \{\mathrm{Id}, (01)(23)(45), (012)(345), (01234), (012345)\}\  \text{if $Q$ is smooth}, \\
&\; \{\mathrm{Id}, (01)(23), (0123), (01234)\}\  \text{if $Q$ is singular},
\endaligned
\right.
\]
a dedicated part of the code performs the following tasks:
\begin{itemize}
    \item[-] creates the list \texttt{lis := FindLis($g$)};
    \item[-] for each group $A$ in $\texttt{lis}$, determines the subvariety $Z \subset W$ consisting of planes fixed  by every element of $A$, and collects these fixed loci into the list \texttt{LFix};
    \item[-] returns the dimension of each subvariety $Z$ in \texttt{LFix};
    \item[-] for each $Z$ in \texttt{LFix}, identifies the group of transformations that fix points in $Z$.
\end{itemize}
By construction, if $z \in Z$, the corresponding stabilizer group consists of transformations generated by $g$ and suitable diagonal matrices that preserve both the plane associated to $z$ and the pencil of quadrics generated by $Q$ and $Q_2$. Therefore, for each $z \in Z$, the associated group can be the automorphism group of the blow-up of the complete intersection $Q \cap Q_2$ along the plane corresponding to $z$.

The proof requires a final manual verification. We begin by inputting a filtered subgroup $\mathfrak{G} \subset H$ and then identifying the planes stabilized by $\mathfrak{G}$. Consequently, we must ensure that $\mathfrak{G}$ is not a proper subgroup of the full automorphism group of the corresponding smooth complete intersection that fixes the plane. This task arises when, for a given permutation $g$, the code outputs two distinct groups, one of which is a subgroup of the other, with both sharing the same locus of stabilized planes.

For example, in the case of the permutation $(01)(23)(45)$, the code outputs the group $C_4$ and finds exactly eight stabilized planes. The code also identifies $D_4$ and associates it with the same eight stabilized planes. It can be easily and manually  checked that the actual automorphism group of the corresponding smooth complete intersection, with the plane fixed, is the dihedral group $D_4$.  In this case, the group $C_4$ returned by the code is precisely a subgroup of this larger automorphism group.

The source code and further details are available at Appendix~\ref{app:aut}.
\end{proof}

\section{K-stability of general Fano's last Fanos}
\label{section:AZ-threefolds}

The goal of this section is to prove Main Theorem. First, we present results from \cite{AbbanZhuang,Book} that we need to prove Main Theorem. Then we use these results to estimate local $\delta$-invariants of smooth Fano 3-folds in Family~\textnumero 2.16.
Finally, we use these estimates to prove Main Theorem.

\subsection{Abban--Zhuang method}
\label{subsection:Kento}

Let $X$ be a smooth Fano threefold, and let $S$ be an~irreducible smooth surface in $X$.
Set
$$
\tau=\mathrm{sup}\Big\{u\in\mathbb{R}_{\geqslant 0}\ \big\vert\ \text{the divisor  $-K_X-uS$ is pseudo-effective}\Big\}.
$$
For~$u\in[0,\tau]$, let $P(u)$ be the~positive part of the~Zariski decomposition of the~divisor $-K_X-uS$,
and let $N(u)$ be its negative part. Set
$$
S_X(S)=\frac{1}{-K_X^3}\int_{0}^{\infty}\mathrm{vol}\big(-K_X-uS\big)du=\frac{1}{-K_X^3}\int_{0}^{\tau}\big(P(u)\big)^3du.
$$
Let $P$ be a point in $S$. For every prime divisor $F$ over $S$, we set
\begin{multline*}
\quad \quad \quad \quad S\big(W^S_{\bullet,\bullet};F\big)=\frac{3}{(-K_X)^3}\int_0^\tau\big(P(u)\cdot P(u)\cdot S\big)\cdot\mathrm{ord}_{F}\big(N(u)\big\vert_{S}\big)du+\\
+\frac{3}{(-K_X)^3}\int_0^\tau\int_0^\infty \mathrm{vol}\big(P(u)\big\vert_{S}-vF\big)dvdu.\quad \quad \quad \quad
\end{multline*}

\begin{theorem}[{\cite{AbbanZhuang,Book}}]
\label{theorem:Hamid-Ziquan-Kento}
One has
$$
\delta_P(X)\geqslant\min\Bigg\{\frac{1}{S_X(S)},\inf_{\substack{F/S\\P\in C_S(F)}}\frac{A_S(F)}{S\big(W^S_{\bullet,\bullet};F\big)}\Bigg\},
$$
where the~infimum is taken by all prime divisors over $S$ whose center on $S$ contains $P$.
\end{theorem}

This theorem can be used to show that $\delta_P(X)\geqslant 1$.
However, if $S(W^S_{\bullet,\bullet};F)>A_S(F)$ for at least one prime divisor $F$ over the~surface $S$ such that $P\in C_F(S)$, then
we cannot use Theorem~\ref{theorem:Hamid-Ziquan-Kento} to prove that $\delta_P(X)\geqslant 1$.
In this case, we can use a similar approach to estimate $\beta$-invariant of some prime divisors over $X$
whose centers on $X$ are curves contained in $X$.

Namely, let $C$ be an irreducible curve in $S$ such that $P\in C$. Write
$$
N(u)\big\vert_{S}=N^\prime(u)+\mathrm{ord}_{C}\big(N(u)\big\vert_{S}\big)C,
$$
so $N^\prime(u)$ is an effective $\mathbb{R}$-divisor on $S$ whose support does not contain $C$.
For $u\in[0,\tau]$, let
$$
t(u)=\mathrm{sup}\Big\{v\in\mathbb{R}_{\geqslant 0}\ \big\vert\ \text{the divisor  $P(u)\big\vert_{S}-vC$ is pseudo-effective}\Big\}.
$$
For~$v\in[0,t(u)]$, we let $P(u,v)$ be the~positive part of the~Zariski decomposition of~$P(u)\vert_{S}-vC$,
and we let $N(u,v)$ be the~negative part  of the~Zariski decomposition of~$P(u)\vert_{S}-vC$. Then
\begin{multline*}
\quad \quad \quad \quad S\big(W^S_{\bullet,\bullet};C\big)=\frac{3}{(-K_X)^3}\int_0^\tau\big(P(u)\cdot P(u)\cdot S\big)\cdot\mathrm{ord}_{C}\big(N(u)\big\vert_{S}\big)du+\\
+\frac{3}{(-K_X)^3}\int_0^\tau\int_0^{t(u)}\big(P(u,v)\big)^2dvdu.\quad \quad \quad \quad
\end{multline*}
Moreover, we have the~following very explicit estimate:

\begin{theorem}[{\cite{AbbanZhuang,Book}}]
\label{theorem:Hamid-Ziquan-Kento-1}
Let $E$ be a prime divisor over $X$ such that $C_X(E)=C$. Then
$$
\frac{A_X(E)}{S_X(E)}\geqslant\min\Bigg\{\frac{1}{S_X(S)},\frac{1}{S\big(W^S_{\bullet,\bullet};C\big)}\Bigg\}.
$$
In particular, if $S_X(S)<1$ and $S(W^S_{\bullet,\bullet};C)<1$, then $\beta(E)>0$.
\end{theorem}

Now, we suppose, in addition, that $C$ is smooth. Then, following \cite{AbbanZhuang,Book}, we let
$$
F_P\big(W_{\bullet,\bullet,\bullet}^{S,C}\big)=\frac{6}{(-K_X)^3} \int_0^\tau\int_0^{t(u)}\big(P(u,v)\cdot C\big)\cdot \mathrm{ord}_P\big(N^\prime(u)\big|_C+N(u,v)\big|_C\big)dvdu
$$
and
$$
S\big(W_{\bullet, \bullet,\bullet}^{S,C};P\big)=\frac{3}{(-K_X)^3}\int_0^\tau\int_0^{t(u)}\big(P(u,v)\cdot C\big)^2dvdu+F_P\big(W_{\bullet,\bullet,\bullet}^{S,C}\big).
$$
We have the~following estimate:

\begin{theorem}[{\cite{AbbanZhuang,Book}}]
\label{theorem:Hamid-Ziquan-Kento-2}
One has
$$
\delta_P(X)\geqslant\min\Bigg\{\frac{1}{S_X(S)},\frac{1}{S\big(W^S_{\bullet,\bullet};C\big)},\frac{1}{S\big(W_{\bullet, \bullet,\bullet}^{S,C};P\big)}\Bigg\}.
$$
In particular, if $S_X(S)<1$, $S(W^S_{\bullet,\bullet};C)<1$ and $S(W_{\bullet, \bullet,\bullet}^{S,C};P)<1$, then $\delta_P(X)>1$.
\end{theorem}

Now, let $f\colon\widetilde{S}\to S$ be a blow up of the~surface $S$ at the~point $P$, let $E$ be the~$f$-exceptional curve,
and let $\widetilde{N}^\prime(u)$ be the~proper transform on $\widetilde{S}$ of the~divisor $N(u)\vert_{S}$.
For $u\in[0,\tau]$, we let
$$
\widetilde{t}(u)=\sup\Big\{v\in \mathbb{R}_{\geqslant 0} \ \big| \ f^*\big(P(u)|_S\big)-vE \text{ is pseudo-effective}\Big\}.
$$
For $v\in [0,\widetilde{t}(u)]$, let $\widetilde{P}(u,v)$ be the~positive part of the~Zariski decomposition of $f^*(P(u)|_S)-vE$,
and let $\widetilde{N}(u,v)$ be the~negative part of this~Zariski decomposition. Set
$$
S\big(W_{\bullet,\bullet}^{S};E\big)=\frac{3}{(-K_X)^3}\int_0^\tau\big(P(u)\cdot P(u)\cdot S\big)\cdot\mathrm{ord}_{E}\big(f^*(N(u)\vert_{S})\big)du+\frac{3}{(-K_X)^3}\int_0^\tau\int_0^{\widetilde{t}(u)}\big(\widetilde{P}(u,v)\big)^2dvdu.\quad \quad \quad \quad
$$
Finally, for every point $O\in E$, we set
$$
F_O\big(W_{\bullet,\bullet,\bullet}^{S,E}\big)=
\frac{6}{(-K_X)^3}\int_0^\tau\int_0^{\widetilde{t}(u)}\big(\widetilde{P}(u,v)\cdot E\big)\times\mathrm{ord}_O\big(\widetilde{N}^\prime(u)\big|_E+\widetilde{N}(u,v)\big|_{E}\big)dvdu
$$
and
$$
S\big(W_{\bullet, \bullet,\bullet}^{S,E};O\big)=
\frac{3}{(-K_X)^3}\int_0^\tau\int_0^{\widetilde{t}(u)}\big(\widetilde{P}(u,v)\cdot E\big)^2dvdu+
F_O\big(W_{\bullet,\bullet,\bullet}^{S,E}\big).
$$
If $P\not\in\mathrm{Supp}(N(u))$ for every $u\in[0,\tau]$, the~formulas for $S(W_{\bullet, \bullet}^{S};E)$ and
$F_O(W_{\bullet,\bullet,\bullet}^{S,E})$ simplify as
\begin{align*}
S\big(W_{\bullet,\bullet}^{S};E\big)&=\frac{3}{(-K_X)^3}\int_0^\tau\int_0^{\widetilde{t}(u)}\big(\widetilde{P}(u,v)\big)^2dvdu,\\
F_O\big(W_{\bullet,\bullet,\bullet}^{S,E}\big)&=\frac{6}{(-K_X)^3}\int_0^\tau\int_0^{\widetilde{t}(u)}\big(\widetilde{P}(u,v)\cdot E\big)\times\mathrm{ord}_O\big(\widetilde{N}(u,v)\big|_{E}\big)dvdu.
\end{align*}
Moreover, Theorem~\ref{theorem:Hamid-Ziquan-Kento-2} can be generalized as follows:

\begin{theorem}[{\cite{AbbanZhuang,Book}}]
\label{theorem:Hamid-Ziquan-Kento-3}
One has
$$
\delta_P(X)\geqslant\min\Bigg\{\frac{1}{S_X(S)},\frac{2}{S\big(W^S_{\bullet,\bullet};E\big)},\inf_{O\in E}\frac{1}{S\big(W_{\bullet, \bullet,\bullet}^{S,E};O\big)}\Bigg\}.
$$
\end{theorem}

This result can be generalized for a purely log terminal blow up of the~surface $S$ at the~point $P$.
Namely, let $g\colon\widehat{S}\to S$ be a purely log terminal blow up of the~point $P$,
let $G$ be the~$g$-exceptional curve,
and let $\widehat{N}^\prime(u)$ be the~proper transform on $\widehat{S}$ of the~divisor $N(u)\vert_{S}$.
For $u\in[0,\tau]$, we let
$$
\widehat{t}(u)=\sup\Big\{v\in \mathbb{R}_{\geqslant 0} \ \big| \ g^*\big(P(u)|_S\big)-vG \text{ is pseudo-effective}\Big\}.
$$
For $v\in [0,\widehat{t}(u)]$, let $\widehat{P}(u,v)$ be the~positive part of the~Zariski decomposition of $g^*(P(u)|_S)-vG$,
and let $\widehat{N}(u,v)$ be the~negative part of this~Zariski decomposition. Set
$$
S\big(W_{\bullet, \bullet}^{S};G\big)=\frac{3}{(-K_X)^3}\int_0^\tau\big(P(u)\cdot P(u)\cdot S\big)\cdot\mathrm{ord}_{G}\big(g^*(N(u)\vert_{S})\big)du+\frac{3}{(-K_X)^3}\int_0^\tau\int_0^{\widehat{t}(u)}\big(\widehat{P}(u,v)\big)^2dvdu.\quad \quad \quad \quad
$$
As above, for every point $O\in G$, we set
$$
F_O\big(W_{\bullet, \bullet,\bullet}^{S,G}\big)=
\frac{6}{(-K_X)^3}\int_0^\tau\int_0^{\widehat{t}(u)}\big(\widehat{P}(u,v)\cdot G\big)\times\mathrm{ord}_O\big(\widehat{N}^\prime(u)\big|_G+\widehat{N}(u,v)\big|_{G}\big)dvdu
$$
and
$$
S\big(W_{\bullet, \bullet,\bullet}^{S,G};O\big)=
\frac{3}{(-K_X)^3}\int_0^\tau\int_0^{\widehat{t}(u)}\big(\widehat{P}(u,v)\cdot G\big)^2dvdu+F_O\big(W_{\bullet,\bullet,\bullet}^{S,G}\big).
$$

\begin{remark}
\label{remark:weighted-blow-up}
Note that $G\cong\mathbb{P}^1$, but $\widehat{S}$~is singular unless $g$ is a usual blow up.
Nevertheless, the~singularities of the~surface $\widehat{S}$ are very mild  --- they are cyclic quotient singularities.
\end{remark}

To present a generalization of Theorem~\ref{theorem:Hamid-Ziquan-Kento-3},
we have to equip the~curve $G$ with an additional structure of a log Fano curve.
This can be done as follows.
Since the~log pair $(\widehat{S},G)$ has purely log terminal singularities (by definition),
the subadjunction formula gives
$$
K_{\widehat{S}}+\Delta_{G}\sim_{\mathbb{Q}}\big(K_{\widehat{S}}+G\big)\big\vert_{G},
$$
where $\Delta_{G}$ is an effective $\mathbb{Q}$-divisor on the~curve $G$ known as the~different of the~log pair $(\widehat{S},E)$.
This divisor $\Delta_{G}$ can be computed as follows:
\begin{itemize}
\item if $O$ is a point in $G$ such that $\widehat{S}$ is smooth at $O$, then $O\not\in\mathrm{Supp}(\Delta_G)$.
\item if $O$ is a point in $G$ such that $\widehat{S}$ has singularity of type $\frac{1}{n}(1,m)$ at $O$, then
$$
\mathrm{ord}_O\big(\Delta_G\big)=\frac{n-1}{n},
$$
where $n\in\mathbb{N}$ and $m\in\mathbb{N}$ such that $\mathrm{gcd}(n,m)=1$.
\end{itemize}
Now, we are ready to state a generalization of Theorem~\ref{theorem:Hamid-Ziquan-Kento-3} for weighted blow ups.

\begin{theorem}[{\cite{AbbanZhuang,Book}}]
\label{theorem:Hamid-Ziquan-Kento-4}
One has
$$
\delta_P(X)\geqslant \min\left\{\frac{1}{S_X(S)},\frac{A_{S}(G)}{S(W_{\bullet,\bullet}^{S};G)},\inf_{O\in G}\frac{1-\mathrm{ord}_O(\Delta_{G})}{S(W_{\bullet, \bullet,\bullet}^{S,G};O)}\right\}.
$$
\end{theorem}

\subsection{Estimating the local \texorpdfstring{$\delta$-invariant}{}}
\label{subsection:delta}
Let us use notations introduced  in Section~\ref{subsection:Kento} with $X$ being a smooth Fano 3-fold in Family \textnumero 2.16.
Recall from Section~\ref{section:intro} that there is Sarkisov link
$$
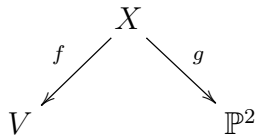
\xymatrix{
&X\ar[ld]_{f}\ar[rd]^{g}&\\
V&&\mathbb{P}^2}
$$
where $V$ is a smooth complete intersection of two quadrics in $\mathbb{P}^5$, $f$ is the~blow up of a smooth conic $C_2\subset V$,
and $g$ is a conic bundle. Let $H=f^{*}(\mathcal{O}_V(1))$, and let $E$ be the~$f$-exceptional surface.
Then either $E\cong\mathbb{P}^1\times\mathbb{P}^1$ or $E\cong\mathbb{F}_2$.

\begin{remark}
The effective cone of $X$ is generated by  the classes of $H-E$ and $E$.
The movable cone is generated by $H-E$ and $H$
and it coincides with the nef cone. The following  picture displays these cones
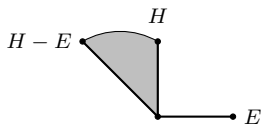
\begin{center}
\begin{tikzpicture}[scale=1]
\tkzDefPoint(0,0){O}
\tkzDefPoint(0,1){H}
\tkzDefPoint(1,0){E}
\tkzDefPoint(-1,1){F}
\filldraw[fill=gray!50] (0,0) -- (0,1) arc (60:120:1cm) -- cycle;
\tkzDrawPoints[size=2](O,H,E,F)
\tkzDrawSegments[thick](O,H O,E O,F)
\tkzLabelPoint[right](E){\tiny $E$}
\tkzLabelPoint[above](0,1.1){\tiny $H$}
\tkzLabelPoint[left](F){\tiny $H-E$}
\end{tikzpicture}
\end{center}
where the shaded region corresponds to the movable  cone. As a consequence a divisor $aH-bE$ is  linearly equivalent to an effective one exactly when $a\geqslant 0$ and $a\geqslant b$.
\end{remark}

Let $P$ be a point in $X$. Let us estimate~$\delta_P(X)$ from below. Let $C$ be the~scheme fiber of the~conic bundle $g\colon X\to  \mathbb{P}^2$ such that $P\in C$. Then
\begin{enumerate}
\item either $C$ is smooth and irreducible;
\item or $C=L_1+L_2$ for smooth rational curves $L_1$ and $L_2$ intersecting transversally at a point;
\item or $C=2L$ for a smooth rational curve $L$.
\end{enumerate}
Our goal is to show that $\delta_P(X)>1$ in the case when $C$ is reduced.

Let $S$ be a general surface in the~linear system $|g^*(\mathcal{O}_{\mathbb{P}^2}(1))|$ that contains $P$.
Then $C\subset S$ and
$$
S\sim H-E,
$$
where $H=f^{*}(\mathcal{O}_V(1))$, and $E$ the~exceptional divisor of the~blow up $f$.
Moreover, if $C$ is reduced, then $S$ is smooth. However, if $C$ is not reduced,
then $S$ has two isolated ordinary double singular points, and we can choose $S$ such that $S$ is smooth at $P$.

\begin{lemma}
\label{lemma:weak-dP4}
The surface $S$ is a del Pezzo surface of degree $4$ except the~following special case:
\begin{center}
$E\cong\mathbb{F}_2$ and $C$ is the~$(-2)$-curve in $E$.
\end{center}
In the~special case, the~surface $S$ is a weak smooth del Pezzo surface of degree $4$, and
$$
S\cap E=C+\mathbf{l}_1+\mathbf{l}_2
$$
for two distinct fibers $\mathbf{l}_1$ and $\mathbf{l}_2$ of the~natural projection $E\to C_2$,
which are $(-2)$-curves in $S$. Moreover, the~curves $\mathbf{l}_1$ and $\mathbf{l}_2$ are the~only $(-2)$-curves in $S$.
\end{lemma}

\begin{proof}
By adjunction formula, we have $-K_S\sim H$, which easily implies the required assertion.
\end{proof}

As in Section~\ref{subsection:Kento}, we let
$$
\tau=\mathrm{sup}\Big\{u\in\mathbb{R}_{\geqslant 0}\ \big\vert\ \text{the divisor  $-K_X-uS$ is pseudo-effective}\Big\}.
$$
Fix $u\in\mathbb{R}_{\geqslant 0}$. Then
$$
-K_X-uS\sim_{\mathbb{R}}2H-E-u(H-E)\sim_{\mathbb{R}}(2-u)H-(1-u)E,
$$
which shows that $\tau=2$.

For $u\in[0,2]$, let $P(u)$ be the~positive part of the~Zariski decomposition of the~divisor $-K_X-uS$,
and let $N(u)$ be the~negative part of the~Zariski decomposition of this divisor. Then
$$
P(u)=
\left\{\aligned
&(2-u)H-(1-u)E \ \text{if $0\leqslant u\leqslant 1$}, \\
&(2-u)H\ \text{if $1\leqslant u\leqslant 2$},
\endaligned
\right.
$$
and
$$
N(u)=\left\{\aligned
&0 \ \text{if $0\leqslant u\leqslant 1$}, \\
&(u-1)E\ \text{if $1\leqslant u\leqslant 2$}.
\endaligned
\right.
$$
Thus, we have
\begin{multline*}
S_X(S)=\frac{1}{22}\int_{0}^{2}\big(P(u)\big)^3du=\frac{1}{22}\int_{0}^{1}\big((2-u)H-(1-u)E\big)^3du+\frac{1}{22}\int_{1}^{2}\big((2-u)H\big)^3du=\\
=\frac{1}{22}\int_{0}^{1}6u^2-24u+22du+\frac{1}{22}\int_{1}^{2}4(2-u)^3du=\frac{13}{22}.\quad\quad\quad\quad\quad
\end{multline*}
In particular, we have $\beta(S)=A_X(S)-S_X(S)=1-\frac{13}{22}=\frac{9}{22}>0$.

\begin{proposition}
\label{proposition:flag-3-Nemuro-lemma}
Suppose that $C$ is reduced and $S$ is a smooth del Pezzo surface. Set
$$
\gamma=\left\{\aligned
&\frac{176}{161}\  \text{if $C$ is smooth and $P\not\in E$}, \\
&\frac{176}{169}\  \text{if $C$ is smooth and $P\in E$}, \\
&\frac{88}{85}\  \text{if $C$ is reducible and $P\not\in E$}, \\
&\frac{88}{89}\  \text{if $C$ is reducible  and $P\in E$}.
\endaligned
\right.
$$
Then $\delta_P(X)\geqslant\gamma$.
\end{proposition}

\begin{proof}
For every $t\in\mathbb{R}_{\geqslant 0}$, we let $D_t=-K_S+tC$. As in \cite[Appendix~A]{CheltsovFujitaKishimotoOkada}, we set
$$
\delta_P(S,D_t)=\inf_{\substack{F/S\\ P\in C_S(F)}}\frac{A_{S}(F)}{S_{D_t}(F)},
$$
where infimum is taken over all prime divisors $F$ over $S$ such that $P\in C_S(F)$,
and
$$
S_{D_t}(F)=\frac{1}{D_t^2}\int_0^\infty \mathrm{vol}\big(D_t-uF\big)du=\frac{1}{4-4t}\int_0^\infty \mathrm{vol}\big(D_t-uF\big)du.
$$
If $t=0$, then $\delta_P(S,D_t)=\delta_P(S)\geqslant\delta(S)\geqslant\frac{4}{3}$ by \cite[Lemma 2.12]{Book}.
If $C$ is smooth, then it follows from \cite[Lemma~24]{CheltsovFujitaKishimotoOkada} we claim that
\begin{equation}
\label{equation:dP4-C-smooth-delta}
\delta_P(S,D_t)\geqslant
\left\{\aligned
&\frac{24}{19+8t+t^2}\ \text{if $0\leqslant t\leqslant 1$}, \\
&\frac{6(1+t)}{5+6t+3t^2}\ \text{if $t\geqslant 1$}.
\endaligned
\right.
\end{equation}
Similarly, if $C$ is reducible (and reduced), then \cite[Lemma~24]{CheltsovFujitaKishimotoOkada} gives
\begin{equation}
\label{equation:dP4-C-singular-delta}
\delta_P(S,D_t)\geqslant \frac{24(1+t)}{19+30t+12t^2}.
\end{equation}
Let us use these estimates together with Theorem~\ref{theorem:Hamid-Ziquan-Kento} to estimate $\delta_P(X)$.

Since $S_X(S)=\frac{13}{22}$, Theorem~\ref{theorem:Hamid-Ziquan-Kento} gives
\begin{equation}
\label{equation:flag-3-Abban-Zhuang}
\delta_P(X)\geqslant\min\Bigg\{\frac{22}{13},\inf_{\substack{F/S\\P\in C_S(F)}}\frac{A_S(F)}{S\big(W^S_{\bullet,\bullet};F\big)}\Bigg\},
\end{equation}
where the~infimum is taken by all prime divisors over $S$ whose center on $S$ contains $P$.
Thus, it follows from \eqref{equation:flag-3-Abban-Zhuang} that to prove the~required assertion, it is enough to show that
$$
S(W^S_{\bullet,\bullet};F)\leqslant \gamma A_S(F)
$$
for any prime divisor $F$ over the~surface $S$ such that $P\in C_S(F)$. Let us do this.

Fix a prime divisor $F$  over $S$  such that $P\in C_S(F)$.
Set $\Gamma=E\vert_{S}$. Then $\Gamma$ is a smooth~rational curve.
Moreover, we have
$$
P(u)\big\vert_{S}=
\left\{\aligned
&-K_S+(1-u)C \ \text{if $0\leqslant u\leqslant 1$}, \\
&(2-u)(-K_S)\ \text{if $1\leqslant u\leqslant 2$},
\endaligned
\right.
$$
and
$$
N(u)\big\vert_{S}=\left\{\aligned
&0 \ \text{if $0\leqslant u\leqslant 1$}, \\
&(u-1)\Gamma\ \text{if $1\leqslant u\leqslant 2$}.
\endaligned
\right.
$$
Therefore, it follows from Theorem~\ref{theorem:Hamid-Ziquan-Kento} that
\begin{multline*}
S\big(W^S_{\bullet,\bullet};F\big)=\frac{3}{22}\int_1^24(2-u)^2(u-1)\mathrm{ord}_{F}\big(\Gamma)du+\frac{3}{22}\int_0^2\int_0^\infty \mathrm{vol}\big(P(u)\big\vert_{S}-vF\big)dvdu=\\
=\frac{3}{22}\int_1^24(2-u)^2(u-1)\mathrm{ord}_{F}\big(\Gamma)du+\frac{3}{22}\int_0^1\int_0^\infty \mathrm{vol}\big(-K_S+(1-u)C-vF\big)dvdu+\\
+\frac{3}{22}\int_1^2\int_0^\infty \mathrm{vol}\big((2-u)(-K_S)-vF\big)dvdu.\quad\quad\quad\quad
\end{multline*}
If $P\not\in E$, this formula simplifies as
$$
S\big(W^S_{\bullet,\bullet};F\big)=\frac{3}{22}\int_0^1\int_0^\infty \mathrm{vol}\big(-K_S+(1-u)C-vF\big)dvdu+\frac{3}{22}\int_1^2\int_0^\infty \mathrm{vol}\big((2-u)(-K_S)-vF\big)dvdu.
$$
If $P\in E$, then  $\mathrm{ord}_{F}\big(\Gamma)\leqslant A_S(F)$, since $(S,\Gamma)$ has log canonical singularities,
which gives
$$
\frac{3}{22}\int_1^24(2-u)^2(u-1)\mathrm{ord}_{F}\big(\Gamma)du\leqslant\frac{3}{22}\int_1^24(2-u)^2(u-1)A_S(F)du=\frac{A_S(F)}{22}.
$$
To estimate the~second term, let $t=1-u$. If $u\in[0,1]$, then
$$
\int_0^\infty \mathrm{vol}\big(-K_S+(1-u)C-vF\big)dvdu\leqslant \frac{D_t^2}{\delta_P(S,D_t)}A_S(F)=\frac{D_t^2}{\delta_P(S,D_t)}A_S(F)=\frac{8-4u}{\delta_P(S,D_t)}A_S(F).
$$
Moreover, using \eqref{equation:dP4-C-smooth-delta} and \eqref{equation:dP4-C-singular-delta}, we get
$$
\delta_P(S,D_t)\geqslant
\left\{\aligned
&\frac{24}{u^2-10u+28}\ \text{if $C$ is irreducible}, \\
&\frac{48-24u}{12u^2-54u+61}\ \text{if $C$ is reducible}. \\
\endaligned
\right.
$$
Hence, if $C$ is irreducible, then
$$
\frac{3}{22}\int_0^1\int_0^\infty \mathrm{vol}\big(-K_S+(1-u)C-vF\big)dvdu\leqslant\frac{3}{22}A_S(F)\int_0^1\frac{(8-4u)(u^2-10u+28)}{24}du=\frac{13}{16}A_S(F).
$$
Similarly, if $C$ is reducible, then
$$
\frac{3}{22}\int_0^1\int_0^\infty \mathrm{vol}\big(-K_S+(1-u)C-vF\big)dvdu\leqslant\frac{3}{22}A_S(F)\int_0^1\frac{(8-4u)(12u^2-54u+61)}{48-24u}du=\frac{19}{22}A_S(F).
$$
Finally, if $u\in[1,2]$, then
\begin{multline*}
\int_0^\infty \mathrm{vol}\big((2-u)(-K_S)-vF\big)dvdu=
\int_0^\infty (2-u)^2\mathrm{vol}\big(-K_S-\frac{v}{2-u}F\big)dvdu=\\
=(2-u)^3\int_0^\infty\mathrm{vol}\big(-K_S-vF\big)dv\leqslant (2-u)^3\frac{4}{\delta_P(S)}A_S(F)\leqslant 3(2-u)^3A_S(F),
\end{multline*}
because $\delta_P(S)\geqslant\delta(S)=\frac{4}{3}$. Thus, we see that
$$
\frac{3}{22}\int_1^2\int_0^\infty \mathrm{vol}\big((2-u)(-K_S)-vF\big)dvdu\leqslant\frac{3}{22}A_S(F)\int_1^23(2-u)^3du=\frac{9}{88}A_S(F).
$$
Now, we have to combine our estimates. If $P\not\in E$ and $C$ is irreducible, then
$$
S\big(W^S_{\bullet,\bullet};F\big)\leqslant\frac{13}{16}A_S(F)+\frac{9}{88}A_S(F)=\frac{161}{176}A_S(F).
$$
If $P\in E$ and $C$ is irreducible, then
$$
S\big(W^S_{\bullet,\bullet};F\big)\leqslant\frac{A_S(F)}{22}+\frac{13}{16}A_S(F)+\frac{9}{88}A_S(F)=\frac{169}{176}A_S(F).
$$
If $P\not\in E$ and $C$ is reducible, then
$$
S\big(W^S_{\bullet,\bullet};F\big)\leqslant\frac{19}{22}A_S(F)+\frac{9}{88}A_S(F)=\frac{85}{88}A_S(F).
$$
Finally, if $P\in E$ and $C$ is irreducible, then
$$
S\big(W^S_{\bullet,\bullet};F\big)\leqslant\frac{A_S(F)}{22}+\frac{19}{22}A_S(F)+\frac{9}{88}A_S(F)=\frac{89}{88}A_S(F).
$$
This shows that $S(W^S_{\bullet,\bullet};F)\leqslant \gamma A_S(F)$, which implies the required assertion.
\end{proof}

Now, arguing as in the proof of Proposition~\ref{proposition:flag-3-Nemuro-lemma}, we obtain the following result.

\begin{proposition}
\label{proposition:flag-3-final}
Suppose that the fiber $C$ is singular and reduced. Then
$$
\delta_P(X)\geqslant\frac{176}{171}.
$$
\end{proposition}

\begin{proof}
We have $C=L_1+L_2$, where $L_1$ and $L_2$ are smooth irreducible rational curves that intersect transversally at one point.
By Lemma~\ref{lemma:weak-dP4}, $S$ is a smooth del Pezzo surface of degree $4$ .

Without loss of generality, we may assume that $P\in L_1$.
Set $\Gamma=E\vert_{S}$.
Then it follows from Lemma~\ref{lemma:weak-dP4} that $\Gamma$ is a smooth~irreducible curve, and $\Gamma$ intersects $L_1$ transversally at one point.
Moreover, we have
$$
P(u)\big\vert_{S}\sim_{\mathbb{R}}
\left\{\aligned
&-K_S+(1-u)(L_1+L_2)\ \text{if $0\leqslant u\leqslant 1$}, \\
&(2-u)(-K_S)\ \text{if $1\leqslant u\leqslant 2$},
\endaligned
\right.
$$
and
$$
N(u)\big\vert_{S}=\left\{\aligned
&0 \ \text{if $0\leqslant u\leqslant 1$}, \\
&(u-1)\Gamma\ \text{if $1\leqslant u\leqslant 2$}.
\endaligned
\right.
$$
In particular, we see that $L_1\not\subset\mathrm{Supp}(N(u)\vert_{S})$ for every $u\in[0,2]$.
As in Section~\ref{subsection:Kento}, we let
$$
t(u)=\mathrm{sup}\Big\{v\in\mathbb{R}_{\geqslant 0}\ \big\vert\ \text{the $\mathbb{R}$-divisor  $P(u)\big\vert_{S}-vL_1$ is pseudo-effective}\Big\}
$$
for every real number $u\in[0,2]$. We will compute $t(u)$ later. In fact, we will see that
\begin{equation}
\label{equation:t-u}
t(u)=\left\{\aligned
&\frac{5-2u}{2}\ \text{if $u\in [0,1]$}, \\
&\frac{6-3u}{2}\ \text{if $u\in [1,2]$}.
\endaligned
\right.
\end{equation}
For~$v\in[0,t(u)]$, we let $P(u,v)$ be the~positive part of the~Zariski decomposition of~$P(u)\vert_{S}-vL_1$,
and we let $N(u,v)$ be the~negative part  of the~Zariski decomposition of~this $\mathbb{R}$-divisor. We set
\begin{align*}
S\big(W^S_{\bullet,\bullet};L_1\big)&=\frac{3}{22}\int_0^2\int_0^{t(u)}\big(P(u,v)\big)^2dvdu,\\
S\big(W_{\bullet, \bullet,\bullet}^{S,L_1};P\big)&=\frac{3}{22}\int_0^2\int_0^{t(u)}\big(P(u,v)\cdot L_1\big)^2dvdu+F_P\big(W_{\bullet,\bullet,\bullet}^{S,L_1}\big),
\end{align*}
where
\begin{multline*}
\quad\quad\quad F_P\big(W_{\bullet,\bullet,\bullet}^{S,L_1}\big)=\frac{6}{22}\int_1^2\int_0^{t(u)}(u-1)\big(P(u,v)\cdot L_1\big)\big(\Gamma\cdot L_1\big)_Pdvdu+\\
\quad\quad\quad\quad\quad\quad\quad\quad\quad\quad+\frac{6}{22}\int_0^2\int_0^{t(u)}\big(P(u,v)\cdot L_1\big)\cdot \mathrm{ord}_P\big(N(u,v)\big|_{L_1}\big)dvdu.
\end{multline*}
Then, since $S_X(S)=\frac{13}{22}$, it follows from Theorem~\ref{theorem:Hamid-Ziquan-Kento-2} that
\begin{equation}
\label{equation:proposition:flag-3-final}
\delta_P(X)\geqslant\min\Bigg\{\frac{22}{13},\frac{1}{S\big(W^S_{\bullet,\bullet};L_1\big)},\frac{1}{S\big(W_{\bullet, \bullet,\bullet}^{S,L_1};P\big)}\Bigg\}.
\end{equation}
Let us compute $S(W^S_{\bullet,\bullet};L_1)$ and $S(W_{\bullet, \bullet,\bullet}^{S,L_1};P)$.

There exists a birational map $\pi\colon S\to\mathbb{P}^2$ that contracts the~curve $L_2$ and four other $(-1)$-curves,
which we denote by $\mathbf{e}_1$, $\mathbf{e}_2$, $\mathbf{e}_3$, $\mathbf{e}_4$.
Then $\pi(L_1)$ is a conic, the~curves $L_2$, $\mathbf{e}_1$, $\mathbf{e}_2$, $\mathbf{e}_3$, $\mathbf{e}_4$ are disjoint,
and their images in $\mathbb{P}^2$ are $5$ distinct points in the~conic $\pi(L_1)$. Moreover, if $u\in[0,1]$, then
$$
P(u)\big\vert_{S}-vL_1\sim_{\mathbb{R}}\frac{5-2u-2v}{2}L_1+\frac{3-2u}{2}L_2+\frac{1}{2}\big(\mathbf{e}_1+\mathbf{e}_2+\mathbf{e}_3+\mathbf{e}_4\big).
$$
Similarly, if $u\in[1,2]$, then
$$
P(u)\big\vert_{S}-vL_1\sim_{\mathbb{R}}\frac{6-3u-2v}{2}L_1+\frac{2-u}{2}\big(L_2+\mathbf{e}_1+\mathbf{e}_2+\mathbf{e}_3+\mathbf{e}_4\big).
$$
This gives \eqref{equation:t-u}.
Moreover, intersecting the~divisor $P(u)\big\vert_{S}-vL_1$ with the~curves $L_2$, $\mathbf{e}_1$, $\mathbf{e}_2$, $\mathbf{e}_3$,~$\mathbf{e}_4$,
we can find  $P(u,v)$ and  $N(u,v)$ for every $u\in[0,2]$ and $v\in[0,t(u)]$.
Namely, if $u\in[0,1]$, then
$$
P(u,v)=
\left\{\aligned
&\frac{5-2u-2v}{2}L_1+\frac{3-2u}{2}L_2+\frac{1}{2}\big(\mathbf{e}_1+\mathbf{e}_2+\mathbf{e}_3+\mathbf{e}_4\big)\ \text{if $0\leqslant v\leqslant 1$}, \\
&\frac{5-2u-2v}{2}\big(L_1+L_2\big)+\frac{1}{2}\big(\mathbf{e}_1+\mathbf{e}_2+\mathbf{e}_3+\mathbf{e}_4\big)\ \text{if $1\leqslant v\leqslant 2-u$},\\
&\frac{5-2u-2v}{2}\big(L_1+L_2+\mathbf{e}_1+\mathbf{e}_2+\mathbf{e}_3+\mathbf{e}_4\big)\ \text{if $2-u\leqslant v\leqslant \frac{5-2u}{2}$},
\endaligned
\right.
$$
and
$$
N(u,v)=\left\{\aligned
&0\ \text{if $0\leqslant v\leqslant 1$}, \\
&(v-1)L_2\ \text{if $1\leqslant v\leqslant 2-u$},\\
&(v-1)L_2+(v+u-2)\big(\mathbf{e}_1+\mathbf{e}_2+\mathbf{e}_3+\mathbf{e}_4\big)  \ \text{if $2-u\leqslant v\leqslant \frac{5-2u}{2}$},
\endaligned
\right.
$$
which implies that
$$
\big(P(u,v)\big)^2=
\left\{\aligned
&8-v^2-4u-2v\ \text{if $0\leqslant v\leqslant 1$}, \\
&9-4u-4v\ \text{if $1\leqslant v\leqslant 2-u$},\\
&(5-2u-2v)^2\ \text{if $2-u\leqslant v\leqslant \frac{5-2u}{2}$},
\endaligned
\right.
$$
and
$$
P(u,v)\cdot L_1=
\left\{\aligned
&1+v\ \text{if $0\leqslant v\leqslant 1$}, \\
&2\ \text{if $1\leqslant v\leqslant 2-u$},\\
&10-4u-4v\ \text{if $2-u\leqslant v\leqslant \frac{5-2u}{2}$}.
\endaligned
\right.
$$
Similarly, if $u\in[1,2]$, then
$$
P(u,v)=
\left\{\aligned
&\frac{6-3u-2v}{2}L_1+\frac{2-u}{2}\big(L_2+\mathbf{e}_1+\mathbf{e}_2+\mathbf{e}_3+\mathbf{e}_4\big)\ \text{if $0\leqslant v\leqslant 2-u$}, \\
&\frac{6-3u-2v}{2}\big(L_1+L_2+\mathbf{e}_1+\mathbf{e}_2+\mathbf{e}_3+\mathbf{e}_4\big)\ \text{if $2-u\leqslant v\leqslant \frac{6-3u}{2}$},
\endaligned
\right.
$$
and
$$
N(u,v)=\left\{\aligned
&0\ \text{if $0\leqslant v\leqslant 2-u$}, \\
&(v+u-2)\big(L_2+\mathbf{e}_1+\mathbf{e}_2+\mathbf{e}_3+\mathbf{e}_4\big)  \ \text{if $2-u\leqslant v\leqslant \frac{6-3u}{2}$},
\endaligned
\right.
$$
which implies that
$$
\big(P(u,v)\big)^2=
\left\{\aligned
&4u^2+2uv-v^2-16u-4v+16\ \text{if $0\leqslant v\leqslant 2-u$}, \\
&(6-3u-2v)^2\ \text{if $2-u\leqslant v\leqslant \frac{6-3u}{2}$},
\endaligned
\right.
$$
and
$$
P(u,v)\cdot L_1=
\left\{\aligned
&2-u+v\ \text{if $0\leqslant v\leqslant 2-u$}, \\
&12-6u-4v\ \text{if $2-u\leqslant v\leqslant \frac{6-3u}{2}$}.
\endaligned
\right.
$$
We have
\begin{multline*}
S\big(W^S_{\bullet,\bullet};L_1\big)=\frac{3}{22}\int_0^2\int_0^{t(u)}\big(P(u,v)\big)^2dvdu=\frac{3}{22}\int_0^1\int_0^{1}8-v^2-4u-2vdvdu+\\
+\frac{3}{22}\int_0^1\int_1^{2-u}9-4u-4vdvdu+\frac{3}{22}\int_0^1\int_{2-u}^{\frac{5-2u}{2}}(5-2u-2v)^2dvdu+\\
+\frac{3}{22}\int_1^2\int_0^{2-u}4u^2+2uv-v^2-16u-4v+16dvdu+\frac{3}{22}\int_1^2\int_{2-u}^{\frac{6-3u}{2}}(6-3u-2v)^2dvdu=\frac{161}{176}.
\end{multline*}
We have
\begin{multline*}
S\big(W_{\bullet, \bullet,\bullet}^{S,L_1};P\big)=\frac{3}{22}\int_0^2\int_0^{t(u)}\big(P(u,v)\cdot L_1\big)^2dvdu+F_P\big(W_{\bullet,\bullet,\bullet}^{S,L_1}\big)=\\
=\frac{3}{22}\int_0^1\int_0^{1}(1+v)^2dvdu+\frac{3}{22}\int_0^1\int_1^{2-u}4dvdu+\frac{3}{22}\int_0^1\int_{2-u}^{\frac{5-2u}{2}}(10-4u-4v)^2dvdu+\\
+\frac{3}{22}\int_1^2\int_0^{2-u}(2-u+v)^2dvdu+\frac{3}{22}\int_1^2\int_{2-u}^{\frac{6-3u}{2}}(12-6u-4v)^2dvdu+F_P\big(W_{\bullet,\bullet,\bullet}^{S,L_1}\big)=
\frac{69}{88}+F_P\big(W_{\bullet,\bullet,\bullet}^{S,L_1}\big).
\end{multline*}
In particular, if $P\not\in\Gamma\cup L_2\cup\mathbf{e}_1\cup\mathbf{e}_2\cup\mathbf{e}_3\cup\mathbf{e}_4$, then $F_P(W_{\bullet,\bullet,\bullet}^{S,L_1})=0$,
so that
$$
S(W_{\bullet, \bullet,\bullet}^{S,L_1};P)=\frac{69}{88}<\frac{171}{176},
$$
and \eqref{equation:proposition:flag-3-final} gives $\delta_P(X)\geqslant\frac{176}{171}$.

Thus, to complete the proof, we may assume that $P\in\Gamma\cup L_2\cup\mathbf{e}_1\cup\mathbf{e}_2\cup\mathbf{e}_3\cup\mathbf{e}_4$.
If $P\not\in\Gamma$, then
$$
F_P\big(W_{\bullet,\bullet,\bullet}^{S,L_1}\big)=\frac{6}{22}\int_0^2\int_0^{t(u)}\big(P(u,v)\cdot L_1\big)\cdot \mathrm{ord}_P\big(N(u,v)\big|_{L_1}\big)dvdu.
$$
Therefore, if $P\not\in\Gamma$ and $P\in L_2$, then
\begin{multline*}
F_P\big(W_{\bullet,\bullet,\bullet}^{S,L_1}\big)=\frac{6}{22}\int_0^1\int_1^{2-u}2(v-1)dvdu+\frac{6}{22}\int_0^1\int_{2-u}^{\frac{5-2u}{2}}(10-4u-4v)(v-1)dvdu+\\
+\frac{6}{22}\int_1^2\int_{2-u}^{\frac{6-3u}{2}}(12-6u-4v)(v+u-2)dvdu=\frac{3}{16},\quad\quad\quad\quad\quad\quad\quad\quad\quad\quad
\end{multline*}
so that $S(W_{\bullet, \bullet,\bullet}^{S,L_1};P)=\frac{69}{88}+\frac{3}{16}=\frac{171}{176}$.
Similarly, if $P\not\in\Gamma$ and $P\in \mathbf{e}_1\cup\mathbf{e}_2\cup\mathbf{e}_3\cup\mathbf{e}_4$, then
\begin{multline*}
\quad\quad\quad\quad\quad\quad F_P\big(W_{\bullet,\bullet,\bullet}^{S,L_1}\big)=
\frac{6}{22}\int_0^1\int_{2-u}^{\frac{5-2u}{2}}(10-4u-4v)(v+u-2)dvdu+\\
+\frac{6}{22}\int_1^2\int_{2-u}^{\frac{6-3u}{2}}(12-6u-4v)(v+u-2)dvdu=\frac{5}{176},\quad\quad\quad\quad\quad\quad
\end{multline*}
so $S(W_{\bullet, \bullet,\bullet}^{S,L_1};P)=\frac{69}{88}+\frac{5}{176}=\frac{13}{16}<\frac{171}{176}$.
Hence, if $P\not\in\Gamma$, then \eqref{equation:proposition:flag-3-final} gives $\delta_P(X)\geqslant\frac{176}{171}$.

To complete the proof, we may assume that $P\in\Gamma$.
Then $(\Gamma\cdot L_1)_P=\Gamma\cdot L_1=1$. This gives
\begin{multline*}
\quad\quad\quad\quad\quad\quad
F_P\big(W_{\bullet,\bullet,\bullet}^{S,L_1}\big)=\frac{6}{22}\int_1^2\int_0^{\frac{6-3u}{2}}(u-1)\big(P(u,v)\cdot L_1\big)dvdu+\\
+\frac{6}{22}\int_0^2\int_0^{t(u)}\big(P(u,v)\cdot L_1\big)\cdot \mathrm{ord}_P\big(N(u,v)\big|_{L_1}\big)dvdu.\quad\quad\quad\quad\quad\quad
\end{multline*}
The first summand in this formula can be computed as follows:
$$
\frac{6}{22}\int_1^2\int_0^{2-u}(u-1)(2-u+v)dvdu+\frac{6}{22}\int_1^2\int_{2-u}^{\frac{6-3u}{2}}(u-1)(12-6u-4v)dvdu=\frac{1}{22}.
$$
But the~second summand we already computed earlier:
$$
\frac{6}{22}\int_0^2\int_0^{t(u)}\big(P(u,v)\cdot L_1\big)\cdot \mathrm{ord}_P\big(N(u,v)\big|_{L_1}\big)dvdu=
\left\{\aligned
&\frac{3}{16}\ \text{if $P\in L_2$}, \\
&\frac{5}{176}\ \text{if $P\in \mathbf{e}_1\cup\mathbf{e}_2\cup\mathbf{e}_3\cup\mathbf{e}_4$},\\
&0\ \text{if $P\not\in L_2\cup \mathbf{e}_1\cup\mathbf{e}_2\cup\mathbf{e}_3\cup\mathbf{e}_4$}.
\endaligned
\right.
$$
This gives
$$
F_P\big(W_{\bullet,\bullet,\bullet}^{S,L_1}\big)=\left\{\aligned
&\frac{41}{176}\ \text{if $P\in L_2$}, \\
&\frac{13}{176}\ \text{if $P\in \mathbf{e}_1\cup\mathbf{e}_2\cup\mathbf{e}_3\cup\mathbf{e}_4$},\\
&\frac{1}{22}\ \text{if $P\not\in L_2\cup \mathbf{e}_1\cup\mathbf{e}_2\cup\mathbf{e}_3\cup\mathbf{e}_4$}.
\endaligned
\right.
$$
Therefore, we see that
$$
S\big(W_{\bullet, \bullet,\bullet}^{S,L_1};P\big)=
\frac{69}{88}+F_P\big(W_{\bullet,\bullet,\bullet}^{S,L_1}\big)=
\left\{\aligned
&\frac{179}{176}\ \text{if $P\in L_2$}, \\
&\frac{151}{176}\ \text{if $P\in \mathbf{e}_1\cup\mathbf{e}_2\cup\mathbf{e}_3\cup\mathbf{e}_4$},\\
&\frac{73}{88}\ \text{if $P\not\in L_2\cup \mathbf{e}_1\cup\mathbf{e}_2\cup\mathbf{e}_3\cup\mathbf{e}_4$}.
\endaligned
\right.
$$
Thus, if $P\not\in L_2$, then $\delta_P(X)\geqslant\frac{176}{171}$ by \eqref{equation:proposition:flag-3-final}.

Therefore, to complete the proof, we may assume that $P\in L_2$. Then $P=L_1\cap L_2\cap \Gamma$,
and the~curves $L_1$, $L_2$, $\Gamma$ intersect each other pairwise transversally at $P$.
Moreover, we have
$$
L_1+L_2+\Gamma\sim -K_S.
$$
This gives
$$
P(u)\big\vert_{S}\sim_{\mathbb{R}}
\left\{\aligned
&\Gamma+(2-u)(L_1+L_2)\ \text{if $0\leqslant u\leqslant 1$}, \\
&(2-u)(L_1+L_2+\Gamma)\ \text{if $1\leqslant u\leqslant 2$}.
\endaligned
\right.
$$
Recall that
$$
N(u)\big\vert_{S}=\left\{\aligned
&0 \ \text{if $0\leqslant u\leqslant 1$}, \\
&(u-1)\Gamma\ \text{if $1\leqslant u\leqslant 2$}.
\endaligned
\right.
$$

Now, let $g\colon\widehat{S}\to S$ be a blow up of the~point $P$, let $G$ be the~exceptional curve of the~blow up~$g$,
and let $\widehat{L}_1$, $\widehat{L}_2$, $\widehat{\Gamma}$ be the~proper transforms on $\widehat{S}$ of the~curves $L_1$, $L_2$, $\Gamma$, respectively.
Set
$$
\widehat{t}(u)=\sup\Big\{v\in \mathbb{R}_{\geqslant 0} \ \big| \ \text{the $\mathbb{R}$-divisor}\ g^*\big(P(u)|_S\big)-vG \text{ is pseudo-effective}\Big\}
$$
for every $u\in[0,2]$. Later, we will see that
\begin{equation}
\label{equation:t-u-blow-up}
t(u)=\left\{\aligned
&5-2u\ \text{if $u\in [0,1]$}, \\
&6-3u\ \text{if $u\in [1,2]$}.
\endaligned
\right.
\end{equation}
For $v\in [0,\widehat{t}(u)]$, let $\widehat{P}(u,v)$ be the~positive part of the~Zariski decomposition of $g^*(P(u)|_S)-vG$,
and let $\widehat{N}(u,v)$ be the~negative part of this~Zariski decomposition. Set
\begin{multline*}
\quad\quad\quad\quad\quad S\big(W_{\bullet, \bullet}^{S};G\big)=\frac{3}{22}\int_1^24(2-u)^2(u-1)du+\frac{3}{22}\int_0^2\int_0^{\widehat{t}(u)}\big(\widehat{P}(u,v)\big)^2dvdu=\\
=\frac{1}{22}+\frac{3}{22}\int_0^2\int_0^{\widehat{t}(u)}\big(\widehat{P}(u,v)\big)^2dvdu.\quad\quad\quad\quad\quad
\end{multline*}
As above, for every point $O\in G$, we set
\begin{multline*}
\quad\quad\quad\quad\quad F_O\big(W_{\bullet, \bullet,\bullet}^{S,G}\big)=
\frac{6}{22}\int_1^2\int_0^{\widehat{t}(u)}(u-1)\big(\widehat{P}(u,v)\cdot G\big)\big(\widehat{\Gamma}\cdot G\big)_Odvdu+\\
+\frac{6}{22}\int_0^2\int_0^{\widehat{t}(u)}\big(\widehat{P}(u,v)\cdot G\big)\big(\widehat{N}(u,v)\cdot G\big)_Odvdu\quad\quad\quad\quad\quad
\end{multline*}
and
$$
S\big(W_{\bullet, \bullet,\bullet}^{S,G};O\big)=
\frac{3}{22}\int_0^2\int_0^{\widehat{t}(u)}\big(\widehat{P}(u,v)\cdot G\big)^2dvdu+F_O\big(W_{\bullet,\bullet,\bullet}^{S,G}\big).
$$
Then Theorem~\ref{theorem:Hamid-Ziquan-Kento-3} or Theorem~\ref{theorem:Hamid-Ziquan-Kento-4} gives
\begin{equation}
\label{equation:flag-3-singular-fiber-2-delta-P}
\delta_P(X)\geqslant \min\left\{\frac{22}{13},\frac{2}{S(W_{\bullet,\bullet}^{S};G)},\inf_{O\in G}\frac{1}{S(W_{\bullet, \bullet,\bullet}^{S,G};O)}\right\}.
\end{equation}

Note that the~curves $\widehat{L}_1$, $\widehat{L}_2$, $\widehat{\Gamma}$ are disjoint,
and $\widehat{L}_1^2=-2$, $\widehat{L}_2^2=-2$, $\widehat{\Gamma}^{2}=-1$ on the~surface~$\widehat{S}$.
On the~other hand, we have
$$
g^*(P(u)|_S)-vG\sim_{\mathbb{R}}
\left\{\aligned
&\widehat{\Gamma}+(2-u)\big(\widehat{L}_1+\widehat{L}_2\big)+(5-2u-v)G\ \text{if $0\leqslant u\leqslant 1$}, \\
&(2-u)\big(\widehat{L}_1+\widehat{L}_2+\widehat{\Gamma}\big)+(6-3u-v)G\ \text{if $1\leqslant u\leqslant 2$}.
\endaligned
\right.
$$
This gives \eqref{equation:t-u-blow-up}.
Moreover, intersecting the~$\mathbb{R}$-divisor $g^*(P(u)|_S)-vG$ with the~curves $\widehat{L}_1$, $\widehat{L}_2$,~$\widehat{\Gamma}$,
we compute $\widehat{P}(u,v)$ and $\widehat{N}(u,v)$ for every $u\in[0,2]$ and $v\in[0,\widehat{t}(u)]$.
Namely, if $u\in[0,1]$, then
$$
\widehat{P}(u,v)=
\left\{\aligned
&\widehat{\Gamma}+(2-u)\big(\widehat{L}_1+\widehat{L}_2\big)+(5-2u-v)G\ \text{if $0\leqslant v\leqslant 1$}, \\
&\widehat{\Gamma}+\frac{5-2u-v}{2}\big(\widehat{L}_1+\widehat{L}_2+2G\big) \ \text{if $1\leqslant v\leqslant 4-2u$},\\
&\frac{5-2u-v}{2}\big(2\widehat{\Gamma}+\widehat{L}_1+\widehat{L}_2+2G\big) \ \text{if $4-2u\leqslant v\leqslant 5-2u$},
\endaligned
\right.
$$
and
$$
\widehat{N}(u,v)=\left\{\aligned
&0\ \text{if $0\leqslant v\leqslant 1$}, \\
&\frac{v-1}{2}\big(\widehat{L}_1+\widehat{L}_2\big)\ \text{if $1\leqslant v\leqslant 4-2u$},\\
&\frac{v-1}{2}\big(\widehat{L}_1+\widehat{L}_2\big)+(v+2u-4)\widehat{\Gamma}\ \text{if $4-2u\leqslant v\leqslant 5-2u$},
\endaligned
\right.
$$
which implies that
$$
\big(\widehat{P}(u,v)\big)^2=
\left\{\aligned
&8-v^2-4u\ \text{if $0\leqslant v\leqslant 1$}, \\
&9-4u-2v\ \text{if $1\leqslant v\leqslant 4-2u$},\\
&(5-2u-v)^2\ \text{if $4-2u\leqslant v\leqslant 5-2u$},
\endaligned
\right.
$$
and
$$
\widehat{P}(u,v)\cdot G=
\left\{\aligned
&v\ \text{if $0\leqslant v\leqslant 1$}, \\
&1\ \text{if $1\leqslant v\leqslant 4-2u$},\\
&5-2u-v\ \text{if $4-2u\leqslant v\leqslant 5-2u$}.
\endaligned
\right.
$$
Similarly, if $u\in[1,2]$, then
$$
\widehat{P}(u,v)=
\left\{\aligned
&(2-u)\big(\widehat{L}_1+\widehat{L}_2+\widehat{\Gamma}\big)+(6-3u-v)G\ \text{if $0\leqslant v\leqslant 2-u$}, \\
&\frac{6-3u-v}{2}\big(\widehat{L}_1+\widehat{L}_2+2G\big)+(2-u)\widehat{\Gamma}\ \text{if $2-u\leqslant v\leqslant 4-2u$},\\
&\frac{6-3u-v}{2}\big(\widehat{L}_1+\widehat{L}_2+2G+2\widehat{\Gamma}\big)\ \text{if $4-2u\leqslant v\leqslant 6-3u$},
\endaligned
\right.
$$
and
$$
\widehat{N}(u,v)=
\left\{\aligned
&0\ \text{if $0\leqslant v\leqslant 2-u$}, \\
&\frac{v+u-2}{2}\big(\widehat{L}_1+\widehat{L}_2\big)\ \text{if $2-u\leqslant v\leqslant 4-2u$},\\
&\frac{v+u-2}{2}\big(\widehat{L}_1+\widehat{L}_2\big)+(v+2u-4)\widehat{\Gamma}\ \text{if $4-2u\leqslant v\leqslant 6-3u$},
\endaligned
\right.
$$
which implies that
$$
\big(\widehat{P}(u,v)\big)^2=
\left\{\aligned
&4u^2-v^2-16u+16\ \text{if $0\leqslant v\leqslant 2-u$}, \\
&(2-u)(10-5u-2v) \ \text{if $2-u\leqslant v\leqslant 4-2u$},\\
&(6-3u-v)^2\ \text{if $4-2u\leqslant v\leqslant 6-3u$},
\endaligned
\right.
$$
and
$$
\widehat{P}(u,v)\cdot G=
\left\{\aligned
&v\ \text{if $0\leqslant v\leqslant 2-u$}, \\
&2-u\ \text{if $2-u\leqslant v\leqslant 4-2u$},\\
&6-3u-v\ \text{if $4-2u\leqslant v\leqslant 6-3u$}.
\endaligned
\right.
$$
This gives
\begin{multline*}
\quad \quad \quad \quad \quad \quad \quad \quad \quad \quad \quad S\big(W_{\bullet, \bullet}^{S};G\big)=\frac{1}{22}+\frac{3}{22}\int_0^2\int_0^{\widehat{t}(u)}\big(\widehat{P}(u,v)\big)^2dvdu=\\
=\frac{1}{22}+\frac{3}{22}\int_0^1\int_0^{1}8-v^2-4udvdu+\frac{3}{22}\int_0^1\int_1^{4-2u}9-4u-2vdvdu+\\
+\frac{3}{22}\int_0^1\int_{4-2u}^{5-2u}(5-2u-v)^2dvdu+\frac{3}{22}\int_1^2\int_0^{2-u}4u^2-v^2-16u+16dvdu+\\
+\frac{3}{22}\int_1^2\int_{2-u}^{4-2u}(2-u)(10-5u-2v)dvdu+\frac{3}{22}\int_1^2\int_{4-2u}^{6-3u}(6-3u-v)^2dvdu=\frac{85}{44}. \quad \quad
\end{multline*}
Then \eqref{equation:flag-3-singular-fiber-2-delta-P} gives $\delta_P(X)\geqslant\frac{88}{85}\geqslant\frac{176}{171}$
provided that $S(W_{\bullet, \bullet,\bullet}^{S,G};O)\leqslant\frac{85}{88}$ for every point $O\in G$.

Fix $O\in G$. Let us show that $S(W_{\bullet, \bullet,\bullet}^{S,G};O)\leqslant\frac{85}{88}$. This would finish the~proof.
We have
\begin{multline*}
\quad \quad \quad \quad \quad \quad S\big(W_{\bullet, \bullet,\bullet}^{S,G};O\big)=
\frac{3}{22}\int_0^2\int_0^{\widehat{t}(u)}\big(\widehat{P}(u,v)\cdot G\big)^2dvdu+F_O\big(W_{\bullet,\bullet,\bullet}^{S,G}\big)=\\
=\frac{3}{22}\int_0^1\int_0^{1}v^2dvdu+\frac{3}{22}\int_0^1\int_1^{4-2u}1dvdu+\frac{3}{22}\int_0^1\int_{4-2u}^{5-2u}(5-2u-v)^2dvdu+\\
+\frac{3}{22}\int_1^2\int_0^{2-u}v^2dvdu+\frac{3}{22}\int_1^2\int_{2-u}^{4-2u}(2-u)^2dvdu+\\
+\frac{3}{22}\int_1^2\int_{4-2u}^{6-3u}(6-3u-v)^2dvdu+F_O\big(W_{\bullet,\bullet,\bullet}^{S,G}\big)=\frac{37}{88}+F_O\big(W_{\bullet,\bullet,\bullet}^{S,G}\big).\quad \quad \quad\quad \quad
\end{multline*}
Moreover, if $O\not\in\widehat{\Gamma}\cup\widehat{L}_1\cup\widehat{L}_1$, then $F_O(W_{\bullet,\bullet,\bullet}^{S,G})=0$, so that  $S(W_{\bullet, \bullet,\bullet}^{S,G};O)=\frac{37}{88}<\frac{85}{88}$ as required.
Similarly, if $O\in\widehat{L}_1\cup\widehat{L}_1$, then $O\not\in\widehat{\Gamma}$, which gives
\begin{multline*}
F_O\big(W_{\bullet, \bullet,\bullet}^{S,G}\big)=\frac{6}{22}\int_0^2\int_0^{\widehat{t}(u)}\big(\widehat{P}(u,v)\cdot G\big)\big(\widehat{N}(u,v)\cdot G\big)_Odvdu=\\
=\frac{6}{22}\int_0^1\int_1^{4-2u}\frac{v-1}{2}dvdu+\frac{6}{22}\int_0^1\int_{4-2u}^{5-2u}\frac{(5-2u-v)(v-1)}{2}dvdu+\\
+\frac{6}{22}\int_1^2\int_{2-u}^{4-2u}\frac{(2-u)(v+u-2)}{2}dvdu+\frac{6}{22}\int_1^2\int_{4-2u}^{6-3u}\frac{(6-3u-v)(v+u-2)}{2}dvdu=\frac{87}{176},
\end{multline*}
so $S(W_{\bullet, \bullet,\bullet}^{S,G};O)=\frac{37}{88}+\frac{87}{176}=\frac{161}{176}<\frac{85}{88}$.
Finally, if $O\in\widehat{\Gamma}$, then $O\not\in\widehat{L}_1\cup\widehat{L}_1$ and $(\widehat{\Gamma}\cdot G)_O=1$, so
\begin{multline*}
F_O\big(W_{\bullet, \bullet,\bullet}^{S,G}\big)=\frac{6}{22}\int_1^2\int_0^{6-3u}(u-1)\big(\widehat{P}(u,v)\cdot G\big)dvdu+\frac{6}{22}\int_0^2\int_0^{\widehat{t}(u)}\big(\widehat{P}(u,v)\cdot G\big)\big(\widehat{N}(u,v)\cdot G\big)_Odvdu=\\
=\frac{6}{22}\int_1^2\int_0^{2-u}v(u-1)dvdu+\frac{6}{22}\int_1^2\int_{2-u}^{4-2u}(2-u)(u-1)dvdu+\frac{6}{22}\int_1^2\int_{4-2u}^{6-3u}(6-3u-v)(u-1)dvdu+\\
+\frac{6}{22}\int_0^1\int_{4-2u}^{5-2u}(5-2u-v)(v+2u-4)dvdu+\frac{6}{22}\int_1^2\int_{4-2u}^{6-3u}(6-3u-v)(v+2u-4)dvdu=\frac{9}{88},
\end{multline*}
which gives $S(W_{\bullet, \bullet,\bullet}^{S,G};O)=\frac{37}{88}+\frac{9}{88}=\frac{23}{4}<\frac{85}{88}$.
This completes the~proof of the~proposition.
\end{proof}

\begin{remark}
\label{remark:non-reduced-fiber-flag-3}
Suppose that $C$ is not reduced. Then~$C=2L$, where $L$ is a smooth rational curve.
Recall that $S$ is a singular del Pezzo surface of degree $4$ that has two isolated ordinary double points,
which both are contained in $L$. In this case, we have
$$
S(W^S_{\bullet,\bullet};L)=\frac{31}{22}.
$$
Indeed, the linear system $|-K_S-2L|$ is free and gives a conic bundle $S\to\mathbb{P}^1$.
Let $Z$ be a general curve in the~pencil $|-K_S-2L|$, let $u$ be a real number in $[0,2]$, and let $v$ be a non-negative real number.
Then
$$
P(u)\big\vert_{S}-vL\sim_{\mathbb{R}}
\left\{\aligned
&(4-2u-v)L+Z\ \text{if $0\leqslant u\leqslant 1$}, \\
&(4-2u-v)L+(2-u)Z\ \text{if $1\leqslant u\leqslant 2$}.
\endaligned
\right.
$$
Thus, since $Z^2=L^2=0$, we see that
\begin{center}
the $\mathbb{R}$-divisor $P(u)\big\vert_{S}-vL$ is pseudoeffective $\iff$ it is nef $\iff$ $v\leqslant 4-2u$.
\end{center}
On the~other hand, we have $L\not\subset\mathrm{Supp}(N(u)\vert_{S})$ for every $u\in[0,2]$.
Thus, we compute
\begin{multline*}
S\big(W^S_{\bullet,\bullet};L\big)=\frac{3}{22}\int_0^2\int_0^{4-2u}\mathrm{vol}\big(P(u)\big\vert_{S}-vL\big)dvdu=\\
=\frac{3}{22}\int_0^1\int_0^{4-2u} \big((4-2u-v)L+Z\big)^2dvdu+\frac{3}{22}\int_1^2\int_0^{4-2u} \big((4-2u-v)L+(2-u)Z\big)^2dvdu=\\
=\frac{3}{22}\int_0^1\int_0^{4-2u}8-4u-2vdvdu+\frac{3}{22}\int_1^2\int_0^{4-2u}2(2-u)(4-2u-v)dvdu=\frac{31}{22}.
\end{multline*}
Hence, we cannot use the~surface $S$ in Theorem~\ref{theorem:Hamid-Ziquan-Kento} to show that $\delta_P(X)>1$.
\end{remark}

Let us conclude this section by proving the following result:

\begin{proposition}
\label{proposition:Iskovskikh-surface}
Suppose that $E\cong\mathbb{F}_2$ and $C$ is the~$(-2)$-curve in $E$.
Then $\delta_P(X)\geqslant \frac{22}{19}$.
\end{proposition}

\begin{proof}
It follows from Lemma~\ref{lemma:weak-dP4} that $S$ is a smooth weak del Pezzo surface of degree $K_S^2=4$, and
$$
E\vert_{S}=C+\mathbf{l}_1+\mathbf{l}_2,
$$
where $\mathbf{l}_1$ and $\mathbf{l}_2$ are two disjoint $(-2)$-curves in $S$,
which are fibers of the~natural projection $E\to C_2$.
Moreover, the~curves $\mathbf{l}_1$ and $\mathbf{l}_2$ are the~only $(-2)$-curves in $S$.
Furthermore, we may assume that
\begin{equation}
\label{equation:Iskovskikh-P-L1-L2}
P\not\in\mathbf{l}_1\cup\mathbf{l}_2
\end{equation}
due to the~generality in the~choice of the~surface $S\in|g^*(\mathcal{O}_{\mathbb{P}^2}(1))|$.

The birational morphism $f\colon X\to V$ induces a birational morphism $S\to f(S)$ that contracts
the curves $\mathbf{l}_1$ and $\mathbf{l}_2$ to two isolated ordinary double points.
The surface $f(S)$ is a quasismooth hypersurface in $\mathbb{P}(1,1,2,2)$ of degree $4$,
and $S$ is often called an Iskovskikh surface.

Let us apply Theorem~\ref{theorem:Hamid-Ziquan-Kento-2} to estimate $\delta_P(X)$ from below.
Fix $u\in[0,2]$ and $v\in\mathbb{R}_{\geqslant 0}$.
Then
$$
P(u)\big\vert_{S}-vC\sim_{\mathbb{R}}
\left\{\aligned
&(3-u-v)C+\mathbf{l}_1+\mathbf{l}_2\ \text{if $0\leqslant u\leqslant 1$}, \\
&\frac{4-2u-v}{2}\big(2C+\mathbf{l}_1+\mathbf{l}_2\big)\ \text{if $1\leqslant u\leqslant 2$},
\endaligned
\right.
$$
and
$$
N(u)\big\vert_{S}=\left\{\aligned
&0 \ \text{if $0\leqslant u\leqslant 1$}, \\
&(u-1)(C+\mathbf{l}_1+\mathbf{l}_2)\ \text{if $1\leqslant u\leqslant 2$}.
\endaligned
\right.
$$
Observe that $P(u)\big\vert_{S}-vC$ is pseudoeffective $\iff$ $v\leqslant t(u)$, where
$$
t(u)=\left\{\aligned
&3-u\ \text{if $u\in [0,1]$}, \\
&4-2u\ \text{if $u\in [1,2]$}.
\endaligned
\right.
$$
Moreover, in the~notations of Section~\ref{subsection:Kento}, we have
$$
N^\prime(u)=\left\{\aligned
&0\ \text{if $0\leqslant u\leqslant 1$}, \\
&(u-1)(\mathbf{l}_1+\mathbf{l}_2)\ \text{if $1\leqslant u\leqslant 2$}.
\endaligned
\right.
$$
Namely, we have $N(u)\vert_{S}=N^\prime(u)+\mathrm{ord}_{C}\big(N(u)\big\vert_{S}\big)C$.

For~$v\in[0,t(u)]$, we let $P(u,v)$ be the~positive part of the~Zariski decomposition of~$P(u)\vert_{S}-vL_1$,
and we let $N(u,v)$ be the~negative part  of this Zariski decomposition.
If $u\in[0,1]$, then
$$
P(u,v)=
\left\{\aligned
&(3-u-v)C+\mathbf{l}_1+\mathbf{l}_2\ \text{if $0\leqslant v\leqslant 1-u$}, \\
&\frac{3-u-v}{2}\big(2C+\mathbf{l}_1+\mathbf{l}_2\big)\ \text{if $1-u\leqslant v\leqslant 3-u$},\\
\endaligned
\right.
$$
and
$$
N(u,v)=\left\{\aligned
&0\ \text{if $0\leqslant v\leqslant 1-u$}, \\
&\frac{v+u-1}{2}\big(\mathbf{l}_1+\mathbf{l}_2\big)\ \text{if $1-u\leqslant v\leqslant 3-u$},\\
\endaligned
\right.
$$
which implies that
$$
\big(P(u,v)\big)^2=
\left\{\aligned
&8-4u-4v\ \text{if $0\leqslant v\leqslant 1-u$}, \\
&(3-u-v)^2\ \text{if $1-u\leqslant v\leqslant 3-u$},\\
\endaligned
\right.
$$
and
$$
P(u,v)\cdot C=
\left\{\aligned
&2\ \text{if $0\leqslant v\leqslant 1-u$}, \\
&3-u-v\ \text{if $1-u\leqslant v\leqslant 3-u$},\\
\endaligned
\right.
$$
Similarly, if $u\in[1,2]$ and $v\in[0,4-2u]$, then
$$
P(u,v)=\frac{4-2u-v}{2}\big(2C+\mathbf{l}_1+\mathbf{l}_2\big)
$$
and $N(u,v)=\frac{v}{2}(\mathbf{l}_1+\mathbf{l}_2)$, so  $(P(u,v))^2=(4-2u-v)^2$ and $P(u,v)\cdot C=4-2u-v$.

Recall that $S_X(S)=\frac{13}{22}$.
Thus, it follows from Theorem~\ref{theorem:Hamid-Ziquan-Kento-2} that
$$
\delta_P(X)\geqslant\min\Bigg\{\frac{22}{13},\frac{1}{S\big(W^S_{\bullet,\bullet};C\big)},\frac{1}{S\big(W_{\bullet, \bullet,\bullet}^{S,C};P\big)}\Bigg\},
$$
where $S(W^S_{\bullet,\bullet};C)$ and $S(W_{\bullet, \bullet,\bullet}^{S,C};P)$ can be computed using
the formulas presented in Section~\ref{subsection:Kento}.
In our case, these formulas give
\begin{multline*}
\quad\quad S\big(W^S_{\bullet,\bullet};C\big)=\frac{3}{22}\int_0^2\big(P(u)\big\vert_{S}\big)^2\cdot\mathrm{ord}_{C}\big(N(u)\big\vert_{S}\big)du+\frac{3}{22}\int_0^2\int_0^{t(u)}\big(P(u,v)\big)^2dvdu=\\
=\frac{3}{22}\int_1^2\big(P(u,0)\big)^2(u-1)du+\frac{3}{22}\int_0^1\int_0^{3-u}\big(P(u,v)\big)^2dvdu+\frac{3}{22}\int_1^2\int_0^{4-2u}\big(P(u,v)\big)^2dvdu=\\
=\frac{3}{22}\int_1^24(2-u)^2(u-1)du+\frac{3}{22}\int_0^1\int_0^{1-u}8-4u-4vdvdu+\\
+\frac{3}{22}\int_0^1\int_{1-u}^{3-u}(3-u-v)^2dvdu+\frac{3}{22}\int_1^2\int_0^{4-2u}(4-2u-v)^2dvdu=\frac{19}{22}
\end{multline*}
and
\begin{multline*}
\quad\quad\quad\quad S\big(W_{\bullet, \bullet,\bullet}^{S,C};P\big)=
\frac{3}{22}\int_0^2\int_0^{t(u)}\big(P(u,v)\cdot C\big)^2dvdu+\\
+\frac{6}{22}\int_0^2\int_0^{t(u)}\big(P(u,v)\cdot C\big)\cdot\mathrm{ord}_P\big(N^\prime(u)\big|_C+N(u,v)\big|_C\big)dvdu=\frac{3}{22}\int_0^2\int_0^{t(u)}\big(P(u,v)\cdot C\big)^2dvdu=\\
=\frac{3}{22}\int_0^1\int_0^{3-u}\big(P(u,v)\cdot C\big)^2dvdu+\frac{3}{22}\int_1^2\int_0^{4-2u}\big(P(u,v)\cdot C\big)^2dvdu=\\
=\frac{3}{22}\int_0^1\int_0^{1-u}4dvdu+\frac{3}{22}\int_0^1\int_{1-u}^{3-u}(3-u-v)^2dvdu+\frac{3}{22}\int_1^2\int_0^{4-2u}(4-2u-v)^2dvdu=\frac{8}{11},
\end{multline*}
because $P\not\in\mathrm{Supp}(N^\prime(u))$ and $P\not\in\mathrm{Supp}(N(u,v))$ by \eqref{equation:Iskovskikh-P-L1-L2}.
Thus, applying Theorem~\ref{theorem:Hamid-Ziquan-Kento-2}, we get
$$
\delta_P(X)\geqslant\min\Bigg\{\frac{22}{13},\frac{22}{19},\frac{11}{8}\Bigg\}=\frac{22}{19},
$$
which complete the proof of the proposition.
\end{proof}

\subsection{The proof}
\label{subsection:proof}

Let $X$ be a smooth Fano 3-fold in Family \textnumero 2.16 defined over a subfield $\Bbbk\subset\mathbb{C}$,
let $X_\mathbb{C}$ be its geometric model, and let $G$ be a (finite) subgroup in $\mathrm{Aut}(X)$.
Recall from Section~\ref{section:intro} that there exists $G$-equivariant diagram
$$
\xymatrix{
&X\ar[ld]_{f}\ar[rd]^{g}&\\
V&&\mathbb{P}^2}
$$
where $V$ is a smooth complete intersection of two quadrics $Q_1$ and $Q_2$ in $\mathbb{P}^5$,
$f$ is a blow up of a smooth conic $C_2\subset V$, and $g$ is a conic bundle.
Let $\Delta$ be the~discriminant curve of the~conic bundle $g$.
Suppose that $G$ does not fix $\Bbbk$-points in $\mathrm{Sing}(\Delta)$.
To prove Main Theorem, we have to show that $X_{\mathbb{C}}$ is K-stable.

Suppose that $X_{\mathbb{C}}$ is not K-polystable.
Then it follows from \cite{Fu19,Li17,Zhuang21} that there is a~$G$-invariant geometrically irreducible prime divisor $\mathbf{E}$ over $X$ such that $\mathbf{E}$ is defined over $\Bbbk$ and  $\beta(\mathbf{E})\leqslant 0$.
Let $Z$ be the~center of the~divisor $\mathbf{E}$ on the~3-fold $X$.
Then
$$
\delta_P(X_{\mathbb{C}})\leqslant 1
$$
for every $\mathbb{C}$-point $P\in Z_{\mathbb{C}}\subset X_{\mathbb{C}}$.
On the other hand, it follows from \cite{Fujita2016} that $Z$ is not a surface. Hence, either $Z$ is a $\Bbbk$-point,
or $Z$ is a geometrically irreducible curve defined over $\Bbbk$.
If $Z$ is a~point, then $g(Z)$ is a $G$-fixed $\Bbbk$-point in $\mathbb{P}^2$,
so $g(Z)$ is not a singular point of the curve $\Delta$, because $G$ does not fix $\Bbbk$-points in $\mathrm{Sing}(\Delta)$.
Similarly, if $Z$ is a curve, then
\begin{itemize}
\item either $g(Z)$ is a geometrically irreducible curve in $\mathbb{P}^2$,
\item or $g(Z)$ is a $\Bbbk$-point in $\mathbb{P}^2$, which is not a singular point of the curve $\Delta$.
\end{itemize}
In every case, $Z_{\mathbb{C}}$ contains a $\mathbb{C}$-point $P$ such that $g(P)$ is not a singular point of the curve $\Delta_{\mathbb{C}}$.
Then, as we mentioned, $\delta_P(X_{\mathbb{C}})\leqslant 1$,
which is impossible by Propositions~\ref{proposition:flag-3-Nemuro-lemma}, \ref{proposition:flag-3-final} and \ref{proposition:Iskovskikh-surface}.
The obtained contradiction completes the proof of Main Theorem.

\section{K-moduli space and GIT}\label{section:moduli}

In this section, we outline a strategy to describe the K-moduli space of \textnumero2.16 by identifying it with a GIT quotient. In general, GIT (semi/poly)stability is comparatively easier to analyze than K-stability, primarily due to the availability of the Hilbert--Mumford criterion.

Our method is called the \emph{moduli continuity method}. To the best of the authors' knowledge, the spirit of the moduli continuity method can be found in the work of Mumford when he proved the equivalence of Deligne-Mumford stability and Chow (semi)stability of curves. This method was successful in identifying K-moduli spaces and GIT moduli spaces; see \cite{OSS16,LX19,SS17,LZ25,Liu22,Zha24}. The moduli continuity method proceeds as follows.
\begin{enumerate}
    \item First, using equivariant K-stability, one can identify a K-stable member within a given family of Fano varieties. By the openness of K-(semi)stability (Theorem~\ref{theorem:openness}), this implies that general members of the family are also K-stable.
    \item Understand the geometry of the K-semistable limits, especially make sure that all the limits appear in some appropriate parameter space. This requires the singularity estimate introduced in \cite{Liu18}.
    \item Next, by using the ampleness of the Chow-Mumford (CM) line bundle, one can propose a candidate for the K-moduli space based on GIT stability with respect to the CM line bundle over the parameter space.
    \item Finally, establish an isomorphism between the K-moduli space and the GIT moduli space, using the separatedness and properness of K-moduli spaces.
    \item[(4+)] Additionally, one can further study the GIT stability of the objects and the corresponding GIT moduli space. Typically, smooth objects are GIT stable, which allows us to prove the K-stability of all smooth Fano varieties in that family.
\end{enumerate}

This approach has been successfully applied to certain cases; see \cite{LZ25,Zha24}.

\subsection{K-semistable limits of a one-parameter family}

The most technical step is to prove the following.

\begin{conjecture}\label{conj:K-ss degeneration}
    Any K-semistable $\bQ$-Gorenstein degeneration $X$ of a smooth \textnumero2.16 is isomorphic to the blow-up of a $(2,2)$-complete intersection in $\bP^5$ along a conic curve.
\end{conjecture}

To reach this, we would like to construct a globally generated line bundle $L$ on $X$ inducing a morphism $X\rightarrow \bP^5$, which is birational onto its image, and its image is a $(2,2)$-complete intersection. Take a one-parameter family $f:\mts{X}\rightarrow T$ over a pointed smooth curve $0\in T$, where $\mts{X}_t$ is a smooth K-stable Fano \textnumero2.16 for any $t\neq0$ and $\mts{X}_0$ is a (singular) K-semistable degeneration. After a base change, we may assume that $\mts{X}|_{T^\circ}\simeq \Bl_{\mts{C}^{\circ}}\mts{V}^{\circ}$, where $\mts{V}^{\circ}\hookrightarrow \bP^5\times T^\circ$ is a family of smooth $(2,2)$-complete intersections over the punctured curve $T^\circ:=T\setminus\{0\}$, and $\mts{C}^{\circ}\rightarrow T^{\circ}$ is a family of smooth conic curves.

A natural candidate for such a line bundle $L$ is the degeneration (or limit) of the pull-back on $\mts{X}|_{T^\circ}$, denoted by $\mtc{L}^{\circ}$, of the line bundle $\mtc{O}_{\bP^5}(1)|_{\mts{V}^{\circ}}$. However, the degeneration exists only as a (Weil!) divisor: it is even not $\bQ$-Cartier. To overcome this, we can perform a small birational modification of the total space $g:\widetilde{\mts{X}}\rightarrow \mts{X}$ such that

\begin{itemize}
    \item $g$ is an isomorphism over $T^{\circ}$;
    \item $\widetilde{\mts{X}}$ is $\bQ$-factorial; and
    \item the central fiber $\widetilde{X}:=\widetilde{\mts{X}}_0$ is a Gorenstein canonical weak Fano  variety, and $\widetilde{X}\rightarrow X$ is crepant.
\end{itemize}

Although we sacrifice the positivity of the anti-canonical line bundle, the central fiber is still K-semistable, i.e., its $\delta$-invariant is at least $1$, by the crepant condition, and most desired properties still hold true. In particular, we have the following.

\begin{theorem}[ref. \cite{LZ25,LX19}]\label{thm:nonvanishing}
Let $X$ be a $\bQ$-Gorenstein smoothable K-semistable (weak) $\bQ$-Fano threefold with volume $V:=(-K_X)^3\geqslant 20$. Then the following holds.
\begin{enumerate}
    \item The variety $X$ is Gorenstein canonical;
    \item There exists a divisor $S\in |-K_X|$ such that $(X,S)$ is a plt pair, and that $(S,-K_X|_S)$ is a (quasi-) polarized K3 surface of degree $V$.
    \item If $D$ is a $\bQ$-Cartier Weil divisor on $X$ which deforms to a $\bQ$-Cartier Weil divisor on a $\bQ$-Gorenstein smoothing of $X$, then $D$ is Cartier.
\end{enumerate}
\end{theorem}

Let $\widetilde{L}$ be a $\bQ$-Cartier divisor on $\widetilde{X}$, which is the degeneration of $\mtc{L}^{\circ}$. Then $\widetilde{L}$ is, by Theorem \ref{thm:nonvanishing}(3), a Cartier divisor. We split our goal into the following steps.

\begin{enumerate}
    \item For a general K3 surface $\widetilde{S}\in|-K_{\widetilde{X}}|$, show that $\widetilde{L}|_{\widetilde{S}}$ is nef (and big).
    \item Show that $\widetilde{L}$ is nef by lifting sections from $\widetilde{L}|_{\widetilde{S}}$: suppose $\Bs|2\widetilde{L}|$ contains a curve $C$, then $C$ intersects $\widetilde{S}$ at base points of $\big|2\widetilde{L}|_{\widetilde{S}}\big|$.
    \item By base-point-free Theorem, $\widetilde{\mtc{L}}$ is semiample and big over $T$, and induces a birational morphism $\widetilde{\mts{X}}\rightarrow \mts{Y}$ over $T$ such that $\mts{Y}|_{T^{\circ}}\simeq \mts{V}^{\circ}$.
    \item The central fiber $\mts{Y}_0$ is a Gorenstein canonical Fano variety. Show that $\mts{Y}_0$ is of Fano index (at least) $2$, and hence we can apply the classification in \cite{Fuj90} to conclude that $\mts{Y}_0$ is a $(2,2)$-complete intersection in $\bP^5$.
    \item Argue that $\mts{X}$ coincides with $\widetilde{\mts{X}}$ and $\mts{X}_0$ is isomorphic the blow-up of $\mts{Y}_0$ along some (possibly singular) conic curve.
\end{enumerate}

\subsection{Parameter space and variation of GIT}

 Let $\bG(2,5)$ be the Grassmannian parametrizing 2-planes in $\bP^5$, and $p:\bH\rightarrow \bG(2,5)$ be the $\bP^5$-bundle, whose fiber over $[\Lambda]$ parametrizes conics in $\Lambda$. Let $\bG(1,20)$ be the parameter space of pencils of quadric hypersurfaces in $\bP^5$. For a pencil $\mtc{P}$ of quadrics, let $V$ be the base locus of $\mtc{P}$, which is a $(2,2)$-complete intersection if $\mtc{P}$ is general. Let $W\subseteq \mathbb{H}\times \bG(1,20)$ be the incidence variety $$W\ := \ \big\{(\Gamma,V)\ | \ \Gamma \subseteq V\big\},$$ and $U\subseteq W$ be the big (codimension of $W \setminus U$ in $W$ is at least 2) open subset of $W$ consisting of pairs $(\Gamma,V)$ such that $V$ is a $(2,2)$-complete intersection with canonical singularities and $\Gamma$ is not contained in the singular locus of $V$. Notice that $$W \ \longrightarrow \ \bH$$ is a $\bG(1,15)$-bundle, hence the Picard group of $W$ is free of rank three, generated by $$\zeta:=p_1^{*}\mtc{O}_{\bH}(1),\ \ \   \xi:=p_1^{*}p^{*}\mtc{O}_{\bG(2,5)}(1),\ \ \ \textup{and}\ \ \  \eta:=p_2^{*}\mtc{O}_{\bG(1,20)}(1).$$ Moreover, since $U$ is a big open subset of $W$, then we can identify $\Pic(U)$ with $\Pic(W)$.

Let $(\mts{Y},\mts{C})\rightarrow U$ be the universal family of $(2,2)$-complete intersections and conic curves, and $f:\mts{X}:=\Bl_{\mts{C}}\mts{Y}\rightarrow U$ be the family of possibly singular Fano threefolds \textnumero2.16 obtained by blowing up $\mts{Y}$ along $\mts{C}$.

\begin{definition}
   Keep the notation as above. For $m\gg0$, by \cite{KM76} one can write $$\det f_{*}\mtc{O}_{\mts{X}}(-mK_{\mts{X}/U}) \ \simeq \ \bigotimes_{i=0}^{4}\mtc{L}_i^{\otimes\binom{m}{i}}.$$ The \emph{Chow-Mumford (CM) $\mathbb{Q}$-line bundle of the family of log Fano pairs} $f : \mts{X} \rightarrow U$ is defined as $$\lambda_{\CM,f}:=-\frac{1}{m^4}\mtc{L}_4,$$ which does not depend on the choice of $m$.
\end{definition}

As $\Pic(U)\simeq\mathbb{Z}\cdot\zeta\oplus\mathbb{Z}\cdot\xi\oplus\mathbb{Z}\cdot\eta$, one can write $$\lambda_{\CM,f}\ \simeq \ a\zeta+b\xi+c\eta,$$ and compute the coefficients $a,b,c$ using testing curves. The functoriality of CM $\mathbb{Q}$-line bundles and the following intersection formula play an important role.

\begin{proposition}
    Let $T\subseteq U$ be a smooth proper curve contained in $U$, and $f_T:\mts{X}_T\rightarrow T$ be the restriction of the family $f:\mts{X}\rightarrow U$. Then $$\deg(\lambda_{\CM,f}|_{T}) \ = \ -(K_{\mts{X}_T/T})^4.$$
\end{proposition}

\begin{proof}
    As $f:\mts{X}\rightarrow U$ is a locally stable family, then by the functoriality of the CM $\bQ$-line bundle, one has that $\lambda_{\CM,f}|_T \simeq \lambda_{\CM,f_T}$. Thus the statement follows from the intersection formula; see \cite[Section 2.7]{XZ20}.
\end{proof}

To conclude this subsection, we briefly discuss the variations of GIT. Although it will not play a role in the remainder of the paper, it is of independent interest.

Consider the induced $G:=\PGL(6)$-action on $W$, which leaves the open subscheme $U$ invariant. The $G$-ample cone $\Amp^{G}(W)$ of $W$ is a strictly convex cone in $\Pic^G(W)_{\bQ}$. For any $\mathbb{Q}$-line bundle $\mtc{L}\in \Amp^G(W)$, one defines the GIT quotient $$\overline{M}^{\GIT}(\mtc{L}) \ := \ W\sslash_{\mtc{L}}\PGL(6).$$

It is not hard to check the following.

\begin{lemma}
    If $\mts{Q}\rightarrow \bP^1$ is a pencil of quadrics in $\bP^5$ such that $Y:=\mts{Q}_0\cap \mts{Q}_{\infty}$ is not a complete intersection, then $(Y,\Gamma)\in W$ is GIT unstable for any polarization $\mtc{L}\in \Amp^G(W)$.
\end{lemma}

\subsection{Identify K-moduli space with the GIT quotient}

We aim for the following result.

\begin{conjecture}
    Let $\lambda_{\CM}$ be the CM line bundle associated with the family $\mts{X}\rightarrow U$. Then there is an isomorphism $$\overline{M}^K_{\textup{\textnumero2.16}} \ \simeq \ U^{\sst}\sslash_{\lambda_{\CM}}\PGL(6).$$
\end{conjecture}

\begin{proof}[Outline of a proof]
One first shows that the K-moduli stack $\mtc{M}^K_{\textup{\textnumero2.16}}$ is isomorphic to the quotient stack $[U^K/\PGL(6)]$. One can achieve this by first producing a morphism $[U^K/\PGL(6)]\rightarrow \mtc{M}^K_{\textup{\textnumero2.16}}$ via the universality of the K-moduli stack, and constructing its inverse by constructing a $\PGL(6)$-torsor over $\mtc{M}^K_{\textup{\textnumero2.16}}$. Here $U^{K}\subseteq U$ is the open subscheme parametrizing K-semistable Fano varieties.

Then one needs to show that $U^K$ is contained in the $\lambda_{\CM}$-GIT semistable locus $U^{\sst}(\lambda_{\CM})$. Suppose $[X]\in U^K$ represents a K-semistable Fano which is GIT unstable. Then by Hilbert--Mumford criterion, one can construct a $\bG_m$-equivariant degeneration of $[X=\Bl_{\Gamma}Y]=[(Y,\Gamma)]$ to a point $[(\mtc{P}_0,\Gamma_0)]$. If the pencil $\mtc{P}_0$ defines also a complete intersection, then by a standard argument, one can identify the Hilbert--Mumford index of this degeneration with the Futaki invariant of the test configuration and apply the K-semistability condition. A similar argument applies to show that the K-polystability of $X=\Bl_{\Gamma}Y$ implies the $\lambda_{\CM}$-GIT polystability of $[(Y,\Gamma)]$.

Finally, one proves the equivalence of the K-semistability and $\lambda_{\CM}$-GIT semistability so that $U^K$ coincides with $U^{\sst}(\lambda_{\CM})$. We want to emphasize that the \emph{properness} of the K-moduli space $\overline{M}^K_{\textup{\textnumero2.16}}$ is crucial. For any $\lambda_{\CM}$-GIT semistable point $[(Y,\Gamma)]$, take a one-parameter family $\{[(Y_b,\Gamma_b)]\}_{b\in B}$ over a smooth pointed curve $0\in B$ such that $(Y_0,\Gamma_0)=(Y,\Gamma)$ and $(Y_b,\Gamma_b)$ are K-semistable for any $0\neq b\in B$. By properness of $\overline{M}^K_{\textup{\textnumero2.16}}$, there exists a K-polystable limit $X'$ of $\Bl_{\Gamma_b}Y_b$ as $b \to 0$, after possibly a finite base change of $(0\in B)$. By Conjecture \ref{conj:K-ss degeneration} and the above paragraph, $X'$ is isomorphic to $\Bl_{\Gamma'_0}Y'_0$ for some $\lambda_{CM}$-GIT polystable point $(Y'_0,\Gamma_0')$. Therefore, there is a $\bG_m$-equivariant degeneration from $[[(Y,\Gamma)]]$ to $g\cdot [(Y'_0,\Gamma_0')]$ for some $g\in \PGL(6)$, which proves that $X$ is K-semistable by the openness of K-semistability, see Theorem~\ref{theorem:openness}.

To conclude, one has that $$\mtc{M}^K_{\textup{\textnumero2.16}} \ \simeq \ [U^K/\PGL(6)] \ = \ [U^{\sst}(\lambda_{\CM})/\PGL(6)] \ =: \ \mtc{M}^{\GIT}(\lambda_{\CM}),$$ which descends to an isomorphism of their good moduli spaces $$\overline{M}^K_{\textup{\textnumero2.16}} \ \simeq \ U^{\sst}\sslash_{\lambda_{\CM}}\PGL(6).$$
\end{proof}


\section*{Appendix A: \texttt{Magma} code for the automorphism computations}
\label{app:aut}

\vspace{1em}
\noindent\textbf{Automorphism Groups of Fano Threefolds of Family 2.16}

\medskip
This repository contains Magma code used to study the automorphism groups of
Fano threefolds in family~2.16.  These varieties are obtained by blowing up
the complete intersection of two smooth quadrics in $\mathbb{P}^5$ along a
conic.  The scripts compute the possible group actions that preserve both the
intersection and the conic, and identify the finite groups that can arise as
automorphism groups of the resulting blow\-ups.

\bigskip
\noindent\textbf{Contents}
\begin{itemize}
  \item \texttt{library.m}\,: a utility library with general-purpose functions
        used across all scripts:
    \begin{itemize}
      \item \texttt{FindLis(g,\,FF)} - for a permutation $g$, returns the
            subgroups of $\mathrm{GL}_6(FF)$ that preserve the absolute-value
            matrix and contain $-I$;
      \item \texttt{WedgePower} - $k^{\text{th}}$ exterior power of a matrix;
      \item \texttt{Fiber} - fibre of a morphism over a given subscheme.
    \end{itemize}

  \item \texttt{Grassmannian\_subvarieties.txt}\,: ambient data for the
        computations:
    \begin{itemize}
      \item sets up $\mathrm{Gr}(3,6)\subset\mathbb{P}^{19}$ via the Pl\"ucker
            embedding;
      \item defines a subvariety $W\subset\mathrm{Gr}(3,6)$ by Pl\"ucker
            relations plus extra constraints;
      \item provides coordinate rings, ambient maps, and projections used in
            fixed-point calculations.
    \end{itemize}

  \item \textbf{Case folders}\,: one folder per conjugacy class in $S_6$,
        each containing a script that computes the fixed loci of planes under
        the corresponding subgroup action:
        \begin{quote}
          \texttt{case(12)(34)(56)}, \quad
          \texttt{case(123)(456)}, \quad
          \texttt{case(12345)}, \quad
          \texttt{case(123456)}
        \end{quote}
        Each script
        \begin{enumerate}
          \item initialises a cyclotomic or quadratic field suited to the
                symmetry;
          \item uses \texttt{FindLis} to find the relevant
                $\mathrm{GL}_6$-subgroups;
          \item computes the fixed loci of these subgroups on
                $W\subset\mathrm{Gr}(3,6)$;
          \item filters the loci by non-emptiness;
          \item reports the dimension of each locus and the subgroup acting
                trivially on it;
          \item (for selected loci) constructs the corresponding fixed planes
                in $\mathbb{P}^5$.
        \end{enumerate}

  \item \texttt{Permutation\_scheme}\,: given a permutation $\rho$, constructs
        the scheme of points $(a_0:\dots:a_5)\in\mathbb{P}^5$ such that the
        pencil of two diagonal quadrics is invariant under~$\rho$.
\end{itemize}

\noindent The complete \texttt{Magma} codebase is publicly available at \href{https://github.com/alaface/Fano-2.16}{github.com/alaface/Fano-2.16}.

\newpage

{\texttt{library.m}}
\begin{lstlisting}[caption={ \texttt{library.m} - shared helper routines }]
// Computes the intersection product of two divisors using an intersection matrix
// Input: vectors a, b in a vector space; symmetric matrix M
// Output: scalar value of the intersection product
qua := function(a,b,M)
    K := BaseRing(Parent(M));
    u := Eltseq(a);
    v := Eltseq(b);
    return (Matrix(K,1,#u,u)*M*Matrix(K,#v,1,v))[1,1];
end function;

// Returns the matrix of absolute values of the entries of M
// Input: matrix M over a finite field FF
// Output: matrix with entries |M_{i,j}| in FF
AbsMatrix := function(M,FF)
    return Matrix([[Norm(FF!M[i,j]) : j in [1..Ncols(M)]] : i in [1..Nrows(M)]]);
end function;

// Finds subgroups of H whose elements have absolute value matrices equal to M
// and that contain -I (minus the identity matrix)
// Input: permutation g generating a subgroup of GL(n, FF)
// Output: list of subgroups satisfying the condition
FindLis := function(g,FF)
    G := Parent(g);
    M := PermutationMatrix(FF, g);
    Glin := GL(Nrows(M), FF);

    // Generate diagonal sign-change matrices
    Dgens := [
        DiagonalMatrix(FF, [(j eq i select -1 else 1) : j in [1..6]])
        : i in [1..6]
    ];

    H := sub< Glin | Dgens, M >;

    // Extract subgroups of H
    lis := [U`subgroup : U in Subgroups(H)];

    // Keep only those containing elements with absolute value matrix equal to M
    lis := [U : U in lis | &or[AbsMatrix(s,FF) eq M : s in U]];

    I := IdentityMatrix(FF, 6);

    // Return only those subgroups containing -I
    return [U : U in lis | -I in U];
end function;

// Computes the k-th exterior power of a square matrix P
// Input: matrix P and integer k
// Output: matrix representing the k-th exterior power
WedgePower := function(P, k)
    K := BaseRing(P);
    n := NumberOfRows(P);
    indices := Subsets({1..n}, k);
    Wedge_P := ZeroMatrix(K, #indices, #indices);
    index_list := Sort([Setseq(u) : u in indices]);

    for i in [1..#index_list] do
        for j in [1..#index_list] do
            rows := index_list[i];
            cols := index_list[j];
            Wedge_P[i, j] := Determinant(Submatrix(P, rows, cols));
        end for;
    end for;

    return Wedge_P;
end function;

// Computes the fiber of a morphism f over a scheme Z
// Input: morphism f: X -> Y, scheme Z in Y
// Output: scheme in X representing the fiber over Z
Fiber := function(f,Z)
    equ := DefiningEquations(f);
    bas := MinimalBasis(Z);
    P := Domain(f);
    X := Scheme(P,[Evaluate(g,equ) : g in bas]);
    return Complement(X,Scheme(P,equ));
end function;

// Returns the name of the projective quotient modulo -I
// Input: subgroup H of GL(n,K)
// Output: GroupName of H / < -I >
GroupNameProj := function(H)
    return GroupName(quo< H | -IdentityMatrix(BaseRing(H.1), NumberOfRows(H.1)) >);
end function;

// Checks whether a pair <scheme, string> already appears in a list
// Input: pair = <L,nm>, seen = list of such pairs
// Output: true if pair is already in seen, false otherwise
AlreadySeen := function(pair,seen)
    for q in seen do
        if pair[1] eq q[1] and pair[2] eq q[2] then
            return true;
        end if;
    end for;
    return false;
end function;
\end{lstlisting}

{\texttt{Discriminant\_curve}}
\begin{lstlisting}[caption={ \texttt{Discriminant\_curve} - discriminant of the quadric pencil },label={lst:disc}]
// Define the affine space with 12 variables over the rationals
A<a0,a1,a2,a3,b0,b1,b2,b3,x4,x5,x6,x7> := AffineSpace(Rationals(), 12);

// Define two polynomials in the ring
f1 := a0*x7 + a1*x4 + a2*x5 + a3*x6;
f2 := b0*x7^2 + (b1*x4 + b2*x5 + b3*x6)*x7 + x4^2 - x5*x6;

// Define the scheme X cut out by f1 and f2
X := Scheme(A, [f1, f2]);

// Eliminate parameters a_i, b_i and variables x5, x6, x7 to obtain an implicit equation in x4
g := Basis(EliminationIdeal(Ideal(X), {a0,a1,a2,a3,b0,b1,b2,b3,x5,x6,x7}))[1];

// Construct the Hessian matrix of g with respect to x5, x6, x7
M := Matrix(3, 3, [Derivative(Derivative(g, u), v) : u, v in [x5, x6, x7]]);

// Compute the determinant of the Hessian matrix
det := Determinant(M);

// Factor the determinant and extract the irreducible factor
Factorization(det)[2][1];
// Output:
// a0^2 - a0*a1*b1 + 2*a0*a2*b3 + 2*a0*a3*b2 + a1^2*b0 + a1^2*b2*b3 - a1*a2*b1*b3 - 
//     a1*a3*b1*b2 + a2^2*b3^2 - 4*a2*a3*b0 + a2*a3*b1^2 - 2*a2*a3*b2*b3 + a3^2*b2^2

// Compute the second prime component of the locus where the Hessian matrix has rank <=
PrimeComponents(Scheme(A, Minors(M, 2)))[2];
// Output:
// Scheme over Rational Field defined by
// a0 - 1/2*a1*b1 + a2*b3 + a3*b2,
// a1^2 - 4*a2*a3,
// b0 - 1/4*b1^2 + b2*b3
\end{lstlisting}

{\texttt{Permutation\_scheme\_smooth}}
\begin{lstlisting}[caption={ \texttt{Permutation\_scheme\_smooth} - invariant scheme (smooth case) },label={lst:permSmooth}]
// Define a projective scheme invariant under the subgroup H of S_6
PermutationScheme := function(H)
    P<a0,a1,a2,a3,a4,a5> := ProjectiveSpace(Rationals(), 5);
    eqns := [];
    for g in Generators(H) do
        M := Matrix([
            [1, 1, 1, 1, 1, 1],
            [a0, a1, a2, a3, a4, a5],
            [a0^g, a1^g, a2^g, a3^g, a4^g, a5^g]
        ]);
        eqns cat:= Minors(M, 3);
    end for;
    X := Scheme(P, eqns);
    
    // Remove trivial or degenerate components
    Y := Scheme(P, &*[a0,a1,a2,a3,a4,a5]) 
           join
         Scheme(P, &*[P.i - P.j : i,j in [1..6] | i lt j]);
    return Complement(X, Y);
end function;

// Filter subgroups of G based on the dimension of the associated scheme
G := Sym(6);
lis := [];
for H in Subgroups(G) do
    H := H`subgroup;
    X := PermutationScheme(H);
    if Dimension(X) ge 0 then 
        Append(~lis, H); 
    end if;
end for;

// List the relevant subgroups
lis;
\end{lstlisting}

{\texttt{Permutation\_scheme\_singular}}
\begin{lstlisting}[caption={ \texttt{Permutation\_scheme\_singular} - invariant scheme (singular case) },label={lst:permSing}]
// Define a projective scheme invariant under the subgroup H of S_6
PermutationScheme := function(H)
    gens := Setseq(Generators(H));
    r := #gens;
    A := AffineSpace(Rationals(), r);
    P5 := ProjectiveSpace(Rationals(), 5);
    P := P5 * A;
    R := CoordinateRing(P);

    coords := [P.i : i in [1..6]];
    pars := [P.(6+i) : i in [1..r]];

    eqns := [];
    for k in [1..r] do
        g := gens[k];
        perm := [j eq 6 select pars[k]^2*coords[j^g] else coords[j^g] : j in [1..6]];
        eqns cat:= Minors(Matrix(R, [[1,1,1,1,1,0], coords, perm]), 3);
    end for;

    bad := Scheme(P, coords[6]) join
           Scheme(P, &*[coords[i]-coords[j] : i,j in [1..6] | i lt j]);

    return Complement(Scheme(P, eqns), bad);
end function;

// Filter subgroups of G based on the dimension of the associated scheme
G := Stabilizer(Sym(6), 6);
lis := [];
for K in Subgroups(G) do
    H := K`subgroup;
    X := PermutationScheme(H);
    if Dimension(X) ge 0 then 
        Append(~lis, H); 
    end if;
end for;

// List the relevant subgroups
lis;
\end{lstlisting}

{\texttt{Grassmannian\_subvarieties}}
\begin{lstlisting}[caption={ \texttt{Grassmannian\_subvarieties} - $\mathrm{Gr}(3,6)\subset\mathbb{P}^{19}$ ambient data },label={lst:grass}]
// Define projective spaces
P5<x0,x1,x2,x3,x4,x5> := ProjectiveSpace(FF, 5);
P<a0,a1,a2,a3,a4,a5,b0,b1,b2,b3,b4,b5,c0,c1,c2,c3,c4,c5> := P5 * P5 * P5;
R := CoordinateRing(P);
pr := map<P->P5 | [a0,a1,a2,a3,a4,a5]>;

// Matrix of coordinates
M := Matrix(R, [
    [a0,a1,a2,a3,a4,a5],
    [b0,b1,b2,b3,b4,b5],
    [c0,c1,c2,c3,c4,c5]
]);

// Generate the Plucker coordinates (3x3 minors)
B := Minors(M, 3);

// Define the Plucker embedding into P^19
P19<z012,z013,z014,z015,z023,z024,z025,z034,z035,z045,
    z123,z124,z125,z134,z135,z145,z234,z235,z245,z345> := ProjectiveSpace(FF, 19);

pl := map<P -> P19 | B>;
Gr := Image(pl); // Grassmannian image

// Quadrics defining the base locus (orthogonality relations)
G1 := a0^2 + a1^2 + a2^2 + a3^2 + a4^2 + a5^2;
G2 := 2*a0*b0 + 2*a1*b1 + 2*a2*b2 + 2*a3*b3 + 2*a4*b4 + 2*a5*b5;
G3 := 2*a0*c0 + 2*a1*c1 + 2*a2*c2 + 2*a3*c3 + 2*a4*c4 + 2*a5*c5;
G4 := b0^2 + b1^2 + b2^2 + b3^2 + b4^2 + b5^2;
G5 := 2*b0*c0 + 2*b1*c1 + 2*b2*c2 + 2*b3*c3 + 2*b4*c4 + 2*b5*c5;
G6 := c0^2 + c1^2 + c2^2 + c3^2 + c4^2 + c5^2;

// Scheme defined by these quadratic forms
WF := Scheme(P, [G1, G2, G3, G4, G5, G6]);

// Define the subvariety of planes contained in the Fermat quadric
W := pl(WF);

// Define a vector of coordinate functions in P^19
v := [P19.i : i in [1..20]];
x := Vector(v);
\end{lstlisting}

{\texttt{Grassmannian-subvarieties-singular}}
\begin{lstlisting}[caption={ \texttt{Grassmannian\_subvarieties\_singular} - singular ambient data for $\mathrm{Gr}(3,6)\subset\mathbb{P}^{19}$ },label={lst:grass_sing}]
// Define projective spaces
P5<x0,x1,x2,x3,x4,x5> := ProjectiveSpace(FF, 5);
P<a0,a1,a2,a3,a4,a5,b0,b1,b2,b3,b4,b5,c0,c1,c2,c3,c4,c5> := P5 * P5 * P5;
R := CoordinateRing(P);
pr := map<P->P5 | [a0,a1,a2,a3,a4,a5]>;

// Matrix of coordinates
M := Matrix(R, [
    [a0,a1,a2,a3,a4,a5],
    [b0,b1,b2,b3,b4,b5],
    [c0,c1,c2,c3,c4,c5]
]);

// Generate the Plucker coordinates (3x3 minors)
B := Minors(M, 3);

// Define the Plucker embedding into P^19
P19<z012,z013,z014,z015,z023,z024,z025,z034,z035,z045,
    z123,z124,z125,z134,z135,z145,z234,z235,z245,z345> := ProjectiveSpace(FF, 19);

pl := map<P -> P19 | B>;
Gr := Image(pl); // Grassmannian image


// Quadrics for the singular Fermat quadric  x0^2+...+x4^2 = 0
G1 :=  a0^2 + a1^2 + a2^2 + a3^2 + a4^2;
G2 := 2*(a0*b0 + a1*b1 + a2*b2 + a3*b3 + a4*b4);
G3 := 2*(a0*c0 + a1*c1 + a2*c2 + a3*c3 + a4*c4);
G4 :=  b0^2 + b1^2 + b2^2 + b3^2 + b4^2;
G5 := 2*(b0*c0 + b1*c1 + b2*c2 + b3*c3 + b4*c4);
G6 :=  c0^2 + c1^2 + c2^2 + c3^2 + c4^2;

// Scheme defined by these quadratic forms
WF := Scheme(P, [G1, G2, G3, G4, G5, G6]);

// Define the subvariety of planes contained in the Fermat quadric
W := ReducedSubscheme(pl(WF));

// Define a vector of coordinate functions in P^19
v := [P19.i : i in [1..20]];
x := Vector(v);
\end{lstlisting}

{\texttt{case(12)(34)(56)}}
\begin{lstlisting}[caption={ Fixed loci for permutation $(12)(34)(56)$ },label={lst:case12_34_56}]
// Define the base field
FF<i> := QuadraticField(-1);

// Load utility functions and subvariety data
load "library.m";
load "Grassmannian subvarieties.txt";

// Define the symmetric group on 6 elements and a representative permutation
G := Sym(6);
g := G!(1,2)(3,4)(5,6);

// Find all subgroups of GL(6, FF) associated to g that satisfy certain conditions
lis := FindLis(g, FF);

// Initialize lists to store data
Grp := [];     // List of subgroups
PGrp := [];    // Corresponding projective groups
LFix := [];    // Fixed loci

// Loop over all candidate subgroups
for A in lis do
    lFix := [];

    // Check fixed points under each element of the subgroup A
    for a in A do
        // Compute the 3rd wedge power action on the Plucker coordinates
        Wa := Matrix(CoordinateRing(P19), WedgePower(a, 3));

        // Matrix for fixed point scheme computation
        N := Matrix(CoordinateRing(P19), [v, Eltseq(x * Wa)]);

        // Define the fixed point scheme in the Grassmannian
        Fix := Scheme(P19, Minors(N, 2)) meet Gr;

        // Intersect with the subvariety W
        Z := Fix meet W;

        lFix := Append(lFix, Z);
        ZFix := &meet lFix;

        // If intersection is empty, break early
        if Dimension(ZFix) eq -1 then
            break;
        end if;
    end for;

    // If the intersection is non-empty, store the subgroup and its data
    if Dimension(ZFix) ne -1 then
        LFix := Append(LFix, ZFix);
        PA := quo< A | -IdentityMatrix(FF, 6) >;
        Grp := Append(Grp, A);
        PGrp := Append(PGrp, PA);
    end if;
end for;

// Output: Dimensions of the fixed loci
[Dimension(Z) : Z in LFix];

// Output: Names of the corresponding projective groups
[GroupName(G) : G in PGrp];

// Compute the fibers over selected indices
ind := [2, 5, 8, 9];
for i in ind do
    PrimeComponents(pr(Fiber(pl, LFix[i])));
end for;
\end{lstlisting}

{\texttt{case(123)(456)}}
\begin{lstlisting}[caption={ Fixed loci for permutation $(123)(456)$ },label={lst:case123_456}]
// Define the base field: Cyclotomic field of order 12
FF<e> := CyclotomicField(12);

// Load utility library and Grassmannian subvariety data
load "library.m";
load "Grassmannian subvarieties.txt";

// Define the symmetric group on 6 elements and a representative permutation
G := Sym(6);
g := G!(1,2,3)(4,5,6);

// Find all subgroups of GL(6, FF) associated to g that satisfy certain conditions
lis := FindLis(g, FF);

// Initialize lists to store data
Grp := [];     // List of subgroups
PGrp := [];    // Corresponding projective groups
LFix := [];    // Fixed loci

// Loop over all candidate subgroups
for A in lis do
    lFix := [];

    // Check fixed points under each element of the subgroup A
    for a in A do
        // Compute the 3rd wedge power action on the Plutcker coordinates
        Wa := Matrix(CoordinateRing(P19), WedgePower(a, 3));

        // Matrix for fixed point scheme computation
        N := Matrix(CoordinateRing(P19), [v, Eltseq(x * Wa)]);

        // Define the fixed point scheme in the Grassmannian
        Fix := Scheme(P19, Minors(N, 2)) meet Gr;

        // Intersect with the subvariety W
        Z := Fix meet W;

        lFix := Append(lFix, Z);
        ZFix := &meet lFix;

        // If intersection is empty, break early
        if Dimension(ZFix) eq -1 then
            break;
        end if;
    end for;

    // If the intersection is non-empty, store the subgroup and its data
    if Dimension(ZFix) ne -1 then
        LFix := Append(LFix, ZFix);
        PA := quo< A | -IdentityMatrix(FF, 6) >;
        Grp := Append(Grp, A);
        PGrp := Append(PGrp, PA);
    end if;
end for;

// Output: Dimensions of the fixed loci
[Dimension(Z) : Z in LFix];

// Output: Names of the corresponding projective groups
[GroupName(G) : G in PGrp];

// Compute fibers over selected indices
ind := [1, 2];
for i in ind do
    PrimeComponents(pr(Fiber(pl, LFix[i])));
end for;
\end{lstlisting}

{\texttt{case(12345)}}
\begin{lstlisting}[caption={ Fixed loci for permutation $(12345)$ },label={lst:case12345}]
// Define the base field: Cyclotomic field of order 20
FF<e> := CyclotomicField(20);

// Load helper functions and Grassmannian data
load "library.m";
load "Grassmannian subvarieties.txt";

// Define the symmetric group on 6 elements and a representative permutation
G := Sym(6);
g := G!(1,2,3,4,5);

// Find relevant subgroups of GL(6, FF) associated with g
lis := FindLis(g, FF);

// Initialize lists for subgroups and fixed loci
Grp := [];     // Full subgroups
PGrp := [];    // Projective versions (mod + or -Id)
LFix := [];    // Fixed loci

// Loop over all candidate subgroups
for A in lis do
    lFix := [];

    // Compute fixed loci under action of each element in the subgroup
    for a in A do
        // Compute the 3rd wedge power action matrix on Plucker coordinates
        Wa := Matrix(CoordinateRing(P19), WedgePower(a, 3));

        // Matrix defining the fixed point condition
        N := Matrix(CoordinateRing(P19), [v, Eltseq(x * Wa)]);

        // Fixed point scheme in the Grassmannian
        Fix := Scheme(P19, Minors(N, 2)) meet Gr;

        // Intersect with the subvariety W
        Z := Fix meet W;

        lFix := Append(lFix, Z);
        ZFix := &meet lFix;

        // If intersection is empty, skip this subgroup
        if Dimension(ZFix) eq -1 then
            break;
        end if;
    end for;

    // If the fixed locus is non-empty, record the subgroup and data
    if Dimension(ZFix) ne -1 then
        LFix := Append(LFix, ZFix);
        PA := quo< A | -IdentityMatrix(FF, 6) >;
        Grp := Append(Grp, A);
        PGrp := Append(PGrp, PA);
    end if;
end for;

// Output: Dimensions of the fixed loci
[Dimension(Z) : Z in LFix];

// Output: Names of the corresponding projective groups
[GroupName(G) : G in PGrp];

// Compute fibers over selected fixed loci
ind := [1];
for i in ind do
    PrimeComponents(pr(Fiber(pl, LFix[i])));
end for;
\end{lstlisting}

{\texttt{case(123456)}}
\begin{lstlisting}[caption={ Fixed loci for permutation $(123456)$ },label={lst:case123456}]
// Define the base field: Cyclotomic field of order 6
FF<e> := CyclotomicField(6);

// Load external scripts
load "library.m";
load "Grassmannian subvarieties.txt";

// Define the symmetric group on 6 elements and a permutation of type (123456)
G := Sym(6);
g := G!(1,2,3,4,5,6);

// Find subgroups of GL(6, FF) associated with g that satisfy filtering conditions
lis := FindLis(g, FF);

// Initialize storage for subgroups and fixed loci
Grp := [];
PGrp := [];
LFix := [];

// Loop over candidate subgroups
for A in lis do
    lFix := [];

    // Compute intersection of fixed loci under all elements of A
    for a in A do
        // Compute the induced action on Plucker coordinates
        Wa := Matrix(CoordinateRing(P19), WedgePower(a, 3));

        // Matrix to define fixed point scheme
        N := Matrix(CoordinateRing(P19), [v, Eltseq(x * Wa)]);

        // Fixed points in the Grassmannian
        Fix := Scheme(P19, Minors(N, 2)) meet Gr;

        // Intersect with subvariety W
        Z := Fix meet W;

        lFix := Append(lFix, Z);
        ZFix := &meet lFix;

        // Skip this subgroup if intersection is empty
        if Dimension(ZFix) eq -1 then
            break;
        end if;
    end for;

    // If fixed locus is non-empty, store subgroup and data
    if Dimension(ZFix) ne -1 then
        LFix := Append(LFix, ZFix);
        PA := quo< A | -IdentityMatrix(FF, 6) >;
        Grp := Append(Grp, A);
        PGrp := Append(PGrp, PA);
    end if;
end for;

// Output: Dimensions of fixed loci
[Dimension(Z) : Z in LFix];

// Output: Names of corresponding projective groups
[GroupName(G) : G in PGrp];

// Compute fibers over selected fixed loci
ind := [1, 2];
comp := [];
for i in ind do
    Append(~comp, PrimeComponents(pr(Fiber(pl, LFix[i]))));
end for;
comp;
\end{lstlisting}

{\texttt{Case-id-singular}}
\begin{lstlisting}[caption={ \texttt{Case-id-singular} - singular fixed locus for identity },label={lst:caseId_sing}]
FF<i>:=CyclotomicField(4);
load "library.m";
load "Grassmannian-subvarieties-singular";

FindLis := function(g,FF)
    n:=6;
    Glin:=GL(n,FF);
    Dgens:=[DiagonalMatrix(FF,[ (j eq i) select -1 else 1 : j in [1..n] ]):i in [1..n]];
    Pmat:=PermutationMatrix(FF,g);
    H:=sub< Glin | Dgens, Pmat >;
    return [U`subgroup : U in Subgroups(H) | -IdentityMatrix(FF,n) in U`subgroup];
end function;

Stab := function(L,FF,g)
    n:=6;
    Glin:=GL(n,FF);
    Dgens:=[DiagonalMatrix(FF,[ (j eq i) select -1 else 1 : j in [1..n] ]):i in [1..n]];
    Pmat:=PermutationMatrix(FF,g);
    H:=sub< Glin | Dgens, Pmat >;
    P5 := Ambient(L);
    sta := [];
    for A in H do
       phi := map<P5->P5|[&+[A[i][j]*P5.j : j in [1..6]] : i in [1..6]]>;
       if phi(L) eq L then Append(~sta,A); end if;
    end for;
    return sub<H|sta>;
end function;

FixedLocus := function(H)
    lFix:=[];
    for h in Generators(H) do
        Wh:=Matrix(CoordinateRing(P19),WedgePower(h,3));
        N:=Matrix(CoordinateRing(P19),[v,Eltseq(x*Wh)]);
        Fix:=Scheme(P19,Minors(N,2)) meet Gr;
        Z:=Fix meet W;
        Append(~lFix,Z);
        if #lFix gt 0 and Dimension(&meet lFix) eq -1 then
            return Scheme(Ambient(W),[1]);
        end if;
    end for;
    return (#lFix eq 0) select W else &meet lFix;
end function;

SamplePointOnGrassmannComponent := function(C)
    A := Ambient(C);
    Y := C;
    while Dimension(Y) gt 0 do
        f := &+[ Random([-2..2])*A.i : i in [1..20] ];
        if f ne 0 then
            Y := Y meet Scheme(A,f);
        end if;
    end while;
    return PrimeComponents(Y)[1];
end function;

PlaneFromGrassmannPoint := function(Z)
    return PrimeComponents(pr(Fiber(pl,Z)))[1];
end function;

g:=Identity(Sym(6));
lis:=FindLis(g,FF);

Grp:=[]; PGrp:=[]; LFix:=[]; Dime:=[]; GrName:=[];
for H in lis do
    ZFix:=FixedLocus(H);
    if Dimension(ZFix) ne -1 then
        Append(~LFix,ZFix);
        Append(~Grp,H);
        Append(~PGrp,quo< H | -IdentityMatrix(FF,6) >);
        Append(~Dime,Dimension(ZFix));
        Append(~GrName,GroupName(PGrp[#PGrp]));
    end if;
end for;

Comps := [];
for i in [1..#Dime] do
    if Dime[i] ge 0 then
       for C in PrimeComponents(LFix[i]) do
          Append(~Comps,<i,C>);
       end for;
    end if;
end for;

res := [];
seen := [];

// zero-dimensional components
for t in Comps do
    i := t[1];
    C := t[2];
    if Dimension(C) eq 0 then
        L := PlaneFromGrassmannPoint(C);
        nm := GroupNameProj(Stab(L,FF,g));
        if not nm in seen then
            Append(~res,<L,nm>);
            Append(~seen,nm);
        end if;
    end if;
end for;

// positive-dimensional components: generic point + special intersections
for t in Comps do
    i := t[1];
    C := t[2];

    if Dimension(C) gt 0 then
        Z0 := SamplePointOnGrassmannComponent(C);
        L := PlaneFromGrassmannPoint(Z0);
        nm := GroupNameProj(Stab(L,FF,g));
        if not nm in seen then
            Append(~res,<L,nm>);
            Append(~seen,nm);
        end if;

        for s in Comps do
            j := s[1];
            D := s[2];
            if #Grp[j] gt #Grp[i] then
                Y := C meet D;
                if Dimension(Y) eq 0 then
                    for Z1 in PrimeComponents(Y) do
                        L1 := PlaneFromGrassmannPoint(Z1);
                        nm1 := GroupNameProj(Stab(L1,FF,g));
                        if not nm1 in seen then
                            Append(~res,<L1,nm1>);
                            Append(~seen,nm1);
                        end if;
                    end for;
                end if;
            end if;
        end for;
    end if;
end for;

Seqset([r[2] : r in res]);
\end{lstlisting}

{\texttt{Case(12)(34)-singular}}
\begin{lstlisting}[caption={ \texttt{Case(12)(34)-singular} - singular fixed locus for $(12)(34)$ },label={lst:case12_34_sing}]
FF<i>:=CyclotomicField(4);
load "library.m";
load "Grassmannian-subvarieties-singular";

FindLis := function(g,FF)
    n:=6;
    Glin:=GL(n,FF);
    Dgens:=[DiagonalMatrix(FF,[ (j eq i) select -1 else 1 : j in [1..n] ]):i in [1..n]];
    Pmat:=PermutationMatrix(FF,g);
    zeta:=FF.1;
    scale:=DiagonalMatrix(FF,[1 : k in [1..n-1]] cat [zeta]);
    H:=sub< Glin | Dgens cat [Pmat,scale] >;
    return [U`subgroup : U in Subgroups(H) |
        -IdentityMatrix(FF,n) in U`subgroup and
        &or[AbsMatrix(s,FF) eq Pmat : s in U`subgroup]
    ];
end function;

FixedLocus := function(H)
    lFix:=[];
    for h in Generators(H) do
        Wh:=Matrix(CoordinateRing(P19),WedgePower(h,3));
        N:=Matrix(CoordinateRing(P19),[v,Eltseq(x*Wh)]);
        Fix:=Scheme(P19,Minors(N,2)) meet Gr;
        Z:=Fix meet W;
        Append(~lFix,Z);
        if #lFix gt 0 and Dimension(&meet lFix) eq -1 then
            return Scheme(Ambient(W),[1]);
        end if;
    end for;
    return (#lFix eq 0) select W else &meet lFix;
end function;

Stab := function(L,FF,g)
    v := [1,5,2,4,3,7];
    P5 := Ambient(L);
    Q0 := Scheme(P5,&+[P5.i^2 : i in [1..5]]);
    Q := Scheme(P5,&+[v[i]*P5.i^2 : i in [1..6]]);
    mon := Sections(LinearSystem(P5,2));
    Glin:=GL(6,FF);
    Dgens:=[DiagonalMatrix(FF,[ (j eq i) select -1 else 1 : j in [1..6] ]):i in [1..6]];
    Pmat:=PermutationMatrix(FF,g);
    zeta:=FF.1;
    scale:=DiagonalMatrix(FF,[1 : k in [1..5]] cat [zeta]);
    H:=sub< Glin | Dgens cat [Pmat,scale] >;
    sta := [];
    for A in H do
       phi := map<P5->P5|[&+[A[i][j]*P5.j : j in [1..6]] : i in [1..6]]>;
       lis:=[[MonomialCoefficient(Equation(X),m) : m in mon] : X in [Q0,Q,phi(Q0),phi(Q)]];
       M := Matrix(lis);
       if phi(L) eq L and Rank(M) eq 2 then Append(~sta,A); end if;
    end for;
    return sub<H|sta>;
end function;

SamplePointOnGrassmannComponent := function(C)
    A := Ambient(C);
    Y := C;
    repeat
        f := &+[ Random([-2..2])*A.i : i in [1..20] ];
        if f ne 0 then
            Y := Y meet Scheme(A,f);
        end if;
    until Dimension(Y) eq 0 and Degree(Y) gt 0;
    return PrimeComponents(Y)[1];
end function;

PlaneFromGrassmannPoint := function(Z)
    return PrimeComponents(pr(Fiber(pl,Z)))[1];
end function;

g:=Sym(6)!(1,2)(3,4);
lis:=FindLis(g,FF);

Grp:=[]; PGrp:=[]; LFix:=[]; Dime:=[]; GrName:=[];
for H in lis do
    ZFix:=FixedLocus(H);
    if Dimension(ZFix) ne -1 then
        Append(~LFix,ZFix);
        Append(~Grp,H);
        Append(~PGrp,quo< H | -IdentityMatrix(FF,6) >);
        Append(~Dime,Dimension(ZFix));
        Append(~GrName,GroupName(PGrp[#PGrp]));
    end if;
end for;

Comps := [];
for i in [1..#Dime] do
    if Dime[i] ge 0 then
       for C in PrimeComponents(LFix[i]) do
          Append(~Comps,<i,C>);
       end for;
    end if;
end for;

res := [];
seen := [];

// zero-dimensional components
for t in Comps do
    C := t[2];
    if Dimension(C) eq 0 then
        L := PlaneFromGrassmannPoint(C);
        nm := GroupNameProj(Stab(L,FF,g));
        if not nm in seen then
            Append(~res,<L,nm>);
            Append(~seen,nm);
        end if;
    end if;
end for;

// positive-dimensional components: generic point + special intersections
for t in Comps do
    i := t[1];
    C := t[2];

    if Dimension(C) gt 0 then
        Z0 := SamplePointOnGrassmannComponent(C);
        L := PlaneFromGrassmannPoint(Z0);
        nm := GroupNameProj(Stab(L,FF,g));
        if not nm in seen then
            Append(~res,<L,nm>);
            Append(~seen,nm);
        end if;

        for s in Comps do
            j := s[1];
            D := s[2];
            if #Grp[j] gt #Grp[i] then
                Y := C meet D;
                if Dimension(Y) eq 0 then
                    for Z1 in PrimeComponents(Y) do
                        L1 := PlaneFromGrassmannPoint(Z1);
                        nm1 := GroupNameProj(Stab(L1,FF,g));
                        if not nm1 in seen then
                            Append(~res,<L1,nm1>);
                            Append(~seen,nm1);
                        end if;
                    end for;
                end if;
            end if;
        end for;
    end if;
end for;

Seqset([r[2] : r in res]);
\end{lstlisting}

{\texttt{Case(1234)-singular}}
\begin{lstlisting}[caption={ \texttt{Case(1234)-singular} - singular fixed locus for $(1234)$ },label={lst:case1234_sing}]
FF<e>:=CyclotomicField(8);
load "library.m";
load "Grassmannian-subvarieties-singular";

FindLis := function(g,FF)
    n:=6;
    Glin:=GL(n,FF);
    Dgens:=[DiagonalMatrix(FF,[ (j eq i) select -1 else 1 : j in [1..n] ]):i in [1..n]];
    Pmat:=PermutationMatrix(FF,g);
    zeta:=FF.1;
    scale:=DiagonalMatrix(FF,[1 : k in [1..n-1]] cat [zeta]);
    H:=sub< Glin | Dgens cat [Pmat,scale] >;
    return [U`subgroup:U in Subgroups(H)|
        -IdentityMatrix(FF,n) in U`subgroup and
        &or[AbsMatrix(s,FF) eq Pmat : s in U`subgroup]
    ];
end function;

FixedLocus := function(H)
    lFix:=[];
    for h in Generators(H) do
        Wh:=Matrix(CoordinateRing(P19),WedgePower(h,3));
        N:=Matrix(CoordinateRing(P19),[v,Eltseq(x*Wh)]);
        Fix:=Scheme(P19,Minors(N,2)) meet Gr;
        Z:=Fix meet W;
        Append(~lFix,Z);
        if #lFix gt 0 and Dimension(&meet lFix) eq -1 then
            return Scheme(Ambient(W),[1]);
        end if;
    end for;
    return (#lFix eq 0) select W else &meet lFix;
end function;

Stab := function(L,FF,g)
    v := [1,e^2,-1,-e^2,0,1];
    P5 := Ambient(L);
    Q0 := Scheme(P5,&+[P5.i^2 : i in [1..5]]);
    Q := Scheme(P5,&+[v[i]*P5.i^2 : i in [1..6]]);
    mon := Sections(LinearSystem(P5,2));
    Glin:=GL(6,FF);
    Dgens:=[DiagonalMatrix(FF,[ (j eq i) select -1 else 1 : j in [1..6] ]):i in [1..6]];
    Pmat:=PermutationMatrix(FF,g);
    zeta:=FF.1;
    scale:=DiagonalMatrix(FF,[1 : k in [1..5]] cat [zeta]);
    H:=sub< Glin | Dgens cat [Pmat,scale] >;
    sta := [];
    for A in H do
       phi := map<P5->P5|[&+[A[i][j]*P5.j : j in [1..6]] : i in [1..6]]>;
       lis:=[[MonomialCoefficient(Equation(X),m) : m in mon] : X in [Q0,Q,phi(Q0),phi(Q)]];
       M := Matrix(lis);
       if phi(L) eq L and Rank(M) eq 2 then Append(~sta,A); end if;
    end for;
    return sub<H|sta>;
end function;

SamplePointOnGrassmannComponent := function(C)
    A := Ambient(C);
    Y := C;
    repeat
        f := &+[ Random([-2..2])*A.i : i in [1..20] ];
        if f ne 0 then
            Y := Y meet Scheme(A,f);
        end if;
    until Dimension(Y) eq 0 and Degree(Y) gt 0;
    return PrimeComponents(Y)[1];
end function;

PlaneFromGrassmannPoint := function(Z)
    return PrimeComponents(pr(Fiber(pl,Z)))[1];
end function;

g:=Sym(6)!(1,2,3,4);
lis:=FindLis(g,FF);

Grp:=[];PGrp:=[];LFix:=[];Dime:=[];GrName:=[];
for H in lis do
    ZFix:=FixedLocus(H);
    if Dimension(ZFix) ne -1 then
        Append(~LFix,ZFix);
        Append(~Grp,H);
        Append(~PGrp,quo< H | -IdentityMatrix(FF,6) >);
        Append(~Dime,Dimension(ZFix));
        Append(~GrName,GroupName(PGrp[#PGrp]));
    end if;
end for;

Comps := [];
for i in [1..#Dime] do
    if Dime[i] ge 0 then
       for C in PrimeComponents(LFix[i]) do
          Append(~Comps,<i,C>);
       end for;
    end if;
end for;

res := [];
seen := [];

// zero-dimensional components
for t in Comps do
    C := t[2];
    if Dimension(C) eq 0 then
        L := PlaneFromGrassmannPoint(C);
        nm := GroupNameProj(Stab(L,FF,g));
        if not nm in seen then
            Append(~res,<L,nm>);
            Append(~seen,nm);
        end if;
    end if;
end for;

// positive-dimensional components: generic point + special intersections
for t in Comps do
    i := t[1];
    C := t[2];

    if Dimension(C) gt 0 then
        Z0 := SamplePointOnGrassmannComponent(C);
        L := PlaneFromGrassmannPoint(Z0);
        nm := GroupNameProj(Stab(L,FF,g));
        if not nm in seen then
            Append(~res,<L,nm>);
            Append(~seen,nm);
        end if;

        for s in Comps do
            j := s[1];
            D := s[2];
            if #Grp[j] gt #Grp[i] then
                Y := C meet D;
                if Dimension(Y) eq 0 then
                    for Z1 in PrimeComponents(Y) do
                        L1 := PlaneFromGrassmannPoint(Z1);
                        nm1 := GroupNameProj(Stab(L1,FF,g));
                        if not nm1 in seen then
                            Append(~res,<L1,nm1>);
                            Append(~seen,nm1);
                        end if;
                    end for;
                end if;
            end if;
        end for;
    end if;
end for;

Seqset([r[2] : r in res]);
\end{lstlisting}

{\texttt{Case(12345)-singular}}
\begin{lstlisting}[caption={ \texttt{Case(12345)-singular} - singular fixed locus for $(12345)$ },label={lst:case12345_sing}]
FF<e> := CyclotomicField(5);
load "library.m";
load "Grassmannian-subvarieties-singular";

FindLis := function(g,FF)
    n:=6;
    Glin:=GL(n,FF);
    Dgens:=[DiagonalMatrix(FF,[ (j eq i) select -1 else 1 : j in [1..n] ]):i in [1..n]];
    Pmat:=PermutationMatrix(FF,g);
    zeta:=FF.1;
    scale:=DiagonalMatrix(FF,[1 : k in [1..n-1]] cat [zeta]);
    H:=sub< Glin | Dgens cat [Pmat,scale] >;
    return [U`subgroup : U in Subgroups(H) |
        -IdentityMatrix(FF,n) in U`subgroup and
        &or[AbsMatrix(s,FF) eq Pmat : s in U`subgroup]
    ];
end function;

FixedLocus := function(H)
    lFix := [];
    for h in Generators(H) do
        Wh := Matrix(CoordinateRing(P19),WedgePower(h,3));
        N  := Matrix(CoordinateRing(P19),[v,Eltseq(x*Wh)]);
        Z  := (Scheme(P19,Minors(N,2)) meet Gr) meet W;
        Append(~lFix,Z);
        if #lFix gt 0 and Dimension(&meet lFix) eq -1 then
            return Scheme(Ambient(W),[1]);
        end if;
    end for;
    return (#lFix eq 0) select W else &meet lFix;
end function;

Stab := function(L,FF,g)
    v := [4*e^3 + 2,4*e^3 + 4*e + 2,-4*e^2 - 2,-2,2,1];
    P5 := Ambient(L);
    Q0 := Scheme(P5,&+[P5.i^2 : i in [1..5]]);
    Q  := Scheme(P5,&+[v[i]*P5.i^2 : i in [1..6]]);
    mon := Sections(LinearSystem(P5,2));
    Glin:=GL(6,FF);
    Dgens:=[DiagonalMatrix(FF,[ (j eq i) select -1 else 1 : j in [1..6] ]):i in [1..6]];
    Pmat:=PermutationMatrix(FF,g);
    zeta:=FF.1;
    scale:=DiagonalMatrix(FF,[1 : k in [1..5]] cat [zeta]);
    H:=sub< Glin | Dgens cat [Pmat,scale] >;
    sta := [];
    for A in H do
       phi := map<P5->P5|[&+[A[i][j]*P5.j : j in [1..6]] : i in [1..6]]>;
       lis := [[MonomialCoefficient(Equation(X),m) : m in mon] : X in [Q0,Q,phi(Q0),phi(Q)]];
       M := Matrix(lis);
       if phi(L) eq L and Rank(M) eq 2 then Append(~sta,A); end if;
    end for;
    return sub<H|sta>;
end function;

SamplePointOnGrassmannComponent := function(C)
    A := Ambient(C);
    Y := C;
    repeat
        f := &+[ Random([-2..2])*A.i : i in [1..20] ];
        if f ne 0 then
            Y := Y meet Scheme(A,f);
        end if;
    until Dimension(Y) eq 0 and Degree(Y) gt 0;
    return PrimeComponents(Y)[1];
end function;

PlaneFromGrassmannPoint := function(Z)
    return PrimeComponents(pr(Fiber(pl,Z)))[1];
end function;

g   := Sym(6)!(1,2,3,4,5);
lis := FindLis(g,FF);

Grp:=[]; PGrp:=[]; LFix:=[]; Dime:=[]; GrName:=[];
for H in lis do
    ZFix := FixedLocus(H);
    if Dimension(ZFix) ne -1 then
        Append(~LFix,ZFix);
        Append(~Grp,H);
        Append(~PGrp,quo< H | -IdentityMatrix(FF,6) >);
        Append(~Dime,Dimension(ZFix));
        Append(~GrName,GroupName(PGrp[#PGrp]));
    end if;
end for;

Comps := [];
for i in [1..#Dime] do
    if Dime[i] ge 0 then
       for C in PrimeComponents(LFix[i]) do
          Append(~Comps,<i,C>);
       end for;
    end if;
end for;

res := [];
seen := [];

// zero-dimensional components
for t in Comps do
    C := t[2];
    if Dimension(C) eq 0 then
        L := PlaneFromGrassmannPoint(C);
        nm := GroupNameProj(Stab(L,FF,g));
        if not nm in seen then
            Append(~res,<L,nm>);
            Append(~seen,nm);
        end if;
    end if;
end for;

// positive-dimensional components: generic point + special intersections
for t in Comps do
    i := t[1];
    C := t[2];

    if Dimension(C) gt 0 then
        Z0 := SamplePointOnGrassmannComponent(C);
        L := PlaneFromGrassmannPoint(Z0);
        nm := GroupNameProj(Stab(L,FF,g));
        if not nm in seen then
            Append(~res,<L,nm>);
            Append(~seen,nm);
        end if;

        for s in Comps do
            j := s[1];
            D := s[2];
            if #Grp[j] gt #Grp[i] then
                Y := C meet D;
                if Dimension(Y) eq 0 then
                    for Z1 in PrimeComponents(Y) do
                        L1 := PlaneFromGrassmannPoint(Z1);
                        nm1 := GroupNameProj(Stab(L1,FF,g));
                        if not nm1 in seen then
                            Append(~res,<L1,nm1>);
                            Append(~seen,nm1);
                        end if;
                    end for;
                end if;
            end if;
        end for;
    end if;
end for;

Seqset([r[2] : r in res]);
\end{lstlisting}


\section*{Appendix B: \texttt{Magma} code for Zariski decompositions}
\label{app:zdec}

This appendix reproduces the \texttt{Magma} routines that compute Zariski
decompositions of divisors.  All scripts reside in the folder
\texttt{zariski-decomposition/} uploaded with this project.

\noindent The full repository is publicly available at
\href{https://github.com/alaface/zariski-decomposition}{github.com/alaface/zariski-decomposition}.

\bigskip
\noindent\textbf{Directory structure}
\begin{itemize}
  \item \texttt{surfaces/} - algorithms that work directly with the
        intersection matrix of a smooth projective surface.
  \item \texttt{cox-rings/} - algorithms that exploit the graded structure of
        a finitely generated Cox ring.
\end{itemize}
Each subdirectory contains its own \texttt{library.m} with the core
functions.

\vspace{1em}
\noindent\textbf{Zariski Decomposition on Surfaces}

\medskip
This repository provides a \texttt{Magma} implementation of the \emph{Zariski
decomposition} of an effective divisor on an algebraic surface, based on the
intersection matrix of the irreducible curves in the support of~$D$.

\bigskip
\noindent\textbf{Contents}
\begin{itemize}
  \item \texttt{library.m} - core library implementing the decomposition algorithm.
\end{itemize}

\bigskip
\noindent\textbf{Description}

Given a surface $X$, an effective divisor $D$, and irreducible curves
$C_1,\dots,C_n$ forming the support of $D$, the Zariski decomposition writes
\[
  D \;=\; P + N,
\]
where
\begin{itemize}
  \item $P$ is \emph{nef} ($P\!\cdot\! C_i \ge 0$ for all $i$),
  \item $N = \sum b_i C_i$ is effective and its intersection matrix
        $\bigl(C_i\!\cdot\!C_j\bigr)$ is negative definite,
  \item $P\!\cdot\!N = 0$.
\end{itemize}

\bigskip
\noindent\textbf{Reference}

The algorithm follows Theorem 7.7 of
\begin{quote}
  O.~Zariski, \emph{The Theorem of Riemann-Roch for High Multiples of an
  Effective Divisor on an Algebraic Surface}, Ann.\ Math.\ \textbf{76}
  (1962), 560-615.
\end{quote}

\bigskip
\noindent\textbf{How it works}

\begin{enumerate}
  \item \textbf{Negative curves.}  Find each $C_i$ with $D\!\cdot\!C_i < 0$.
  \item \textbf{Construct $N$.}  Let $C_1,\dots,C_r$ be those curves.  Seek
        $N=\sum_{j=1}^{r} b_jC_j$ such that
        $N\!\cdot\!C_i = D\!\cdot\!C_i$ for every $i$.
        Because the matrix $\bigl(C_i\!\cdot\!C_j\bigr)$ is negative
        definite, this linear system has a unique solution
        $(b_1,\dots,b_r)\in\mathbb{Q}_{\ge0}^r$.
  \item \textbf{Update $D$.}  Replace $D$ by $D-N$, which now satisfies
        $(D-N)\!\cdot\!C_i=0$ for all curves in the support of $N$.
  \item \textbf{Terminate.}  Repeat until the updated divisor is nef
        ($D\!\cdot\!C_i\ge0$ for every curve).  The final pair $(P,N)$ is the
        Zariski decomposition.
\end{enumerate}

\bigskip
\noindent\textbf{How to use}

\begin{lstlisting}[language=Magma,basicstyle=\ttfamily\small,
                   frame=single,captionpos=b]
load "library.m";
\end{lstlisting}

\bigskip
\noindent\textbf{Example: dimension 3}

\begin{lstlisting}[language=Magma,basicstyle=\ttfamily\small,
                   frame=single,captionpos=b]
load "library.m";

R := Rationals();
V := VectorSpace(R, 3);

// Negative definite intersection matrix
M := Matrix(R, 3, 3, [ -2,  1, 0,
                        1, -2, 1,
                        0,  1, -2 ]);

// Curve basis
C1 := V![1,0,0];
C2 := V![0,1,0];
C3 := V![0,0,1];
Curves := [C1, C2, C3];

// Divisor with non-negative coefficients
D := V![2,1,0];

// Compute Zariski decomposition
P, N := ZariskiDecomposition(D, Curves, M);

print "D =", D;
print "P =", P;
print "N =", N;

// Check intersection with each curve
for i in [1..#Curves] do
  printf "P.C%o = %o\n", i,
         IntersectionNumber(P, Curves[i], M);
end for;
\end{lstlisting}
{\texttt{surfaces/library.m}}
{\texttt{library.m}}
\begin{lstlisting}[caption={ \texttt{library.m} - shared helper routines }]
///////////////////////////////////////////////////////////////////////////
// Support functions for Zariski decomposition over a rational lattice
///////////////////////////////////////////////////////////////////////////

///////////////////////////////////////////////////////////////////////////
// IntersectionNumber(D, C, M)
// Input: 
//   - D: vector representing a divisor
//   - C: vector representing a curve
//   - M: symmetric intersection matrix
// Output:
//   - The intersection number D.C computed as Transpose(D)*M*C
///////////////////////////////////////////////////////////////////////////
IntersectionNumber := function(D, C, M)
    n := Nrows(M);
    U := Matrix(1,n,Eltseq(D))*M*Matrix(n,1,Eltseq(C));
    return U[1,1];
end function;

///////////////////////////////////////////////////////////////////////////
// NegativeCurves(D, Curves, M)
// Input:
//   - D: divisor vector
//   - Curves: list of curve vectors
//   - M: intersection matrix
// Output:
//   - List of curves C in Curves such that D.C < 0
///////////////////////////////////////////////////////////////////////////
NegativeCurves := function(D, Curves, M)
    return [ C : C in Curves | IntersectionNumber(D, C, M) lt 0 ];
end function;

///////////////////////////////////////////////////////////////////////////
// SolveNegativePart(D, NegC, M)
// Input:
//   - D: divisor vector
//   - NegC: list of curves with D.C < 0
//   - M: intersection matrix
// Output:
//   - List of rational coefficients b_i such that 
//     N = sum b_i C_i satisfies N.C_i = D.C_i
///////////////////////////////////////////////////////////////////////////
SolveNegativePart := function(D, NegC, M)
    q := #NegC;
    if q eq 0 then return []; end if;

    A := ZeroMatrix(Rationals(), q, q);
    b := ZeroMatrix(Rationals(), q, 1);
    for i in [1..q] do
        for j in [1..q] do
            A[i,j] := IntersectionNumber(NegC[i], NegC[j], M);
        end for;
        b[i,1] := IntersectionNumber(D, NegC[i], M);
    end for;
    return Eltseq(Solution(A, Transpose(b)));
end function;

///////////////////////////////////////////////////////////////////////////
// IsNef(D, Curves, M)
// Input:
//   - D: divisor vector
//   - Curves: list of curve vectors
//   - M: intersection matrix
// Output:
//   - true if D.C >= 0 for all C in Curves (i.e., D is nef), false otherwise
///////////////////////////////////////////////////////////////////////////
IsNef := function(D, Curves, M)
    return &and[ IntersectionNumber(D, C, M) ge 0 : C in Curves ];
end function;

///////////////////////////////////////////////////////////////////////////
// ZariskiDecomposition(D, Curves, M)
// Input:
//   - D: pseudo-effective divisor vector
//   - Curves: list of irreducible curve vectors
//   - M: intersection matrix
// Output:
//   - A pair <P, N> such that D = P + N is the Zariski decomposition,
//     where P is nef and N is effective with negative definite support
///////////////////////////////////////////////////////////////////////////
ZariskiDecomposition := function(D, Curves, M)
    V := Parent(D);
    P := D;
    Ntot := V!0;

    Neg := [];
    repeat
        Neg := Sort(Setseq(Set(Neg cat NegativeCurves(P, Curves, M))));
        if #Neg eq 0 then
            return P, Ntot;
        end if;

        coeff := SolveNegativePart(P, Neg, M); 
        Nstep := V!0;
        for i in [1..#Neg] do
            Nstep +:= coeff[i]*Neg[i];
        end for;

        P    -:= Nstep;
        Ntot +:= Nstep; 
    until IsNef(P, Curves, M);

    return P, Ntot;
end function;
\end{lstlisting}

\vspace{1em}
\noindent\textbf{Zariski Decomposition from Cox Rings}

\medskip
This directory contains a \texttt{Magma} implementation of the
\emph{Zariski decomposition} of a divisor class on a projective variety
whose Cox ring is finitely generated.  The algorithm relies on the degrees
of the homogeneous generators.

\bigskip
\noindent\textbf{Contents}
\begin{itemize}
  \item \texttt{library.m} - core library implementing the decomposition
        algorithm based on the Cox ring.
\end{itemize}

\bigskip
\noindent\textbf{Description}

Let $X$ be a projective variety whose Cox ring is finitely generated by
homogeneous elements $f_1,\dots,f_r$ with degrees
$w_1,\dots,w_r \in \operatorname{Cl}(X)$.  For a divisor class
$w_D \in \operatorname{Cl}(X)$ the Zariski decomposition writes
\[
  w_D \;=\; w_P + w_N,
\]
where
\begin{itemize}
  \item $w_P$ is the class of a \emph{nef} divisor;
  \item $w_N = \sum \mu_i\,w_i$ is the class of an effective combination of
        the generator degrees;
  \item $\mu_i \in \mathbb{Q}_{\ge 0}$.
\end{itemize}

\bigskip
\noindent\textbf{Reference}

The algorithm follows Theorem 3.3.4.8 of
\begin{quote}
  I.~Arzhantsev, U.~Derenthal, J.~Hausen, A.~Laface,
  \emph{Cox Rings},
  Cambridge Studies in Advanced Mathematics 144,
  Cambridge University Press, 2015.
\end{quote}

\bigskip
\noindent\textbf{How it works}

For each index $i\in\{1,\dots,r\}$:
\begin{enumerate}
  \item Form the cone
        \[
          \tau_i \;=\;
          \operatorname{cone}(w_D,\,-w_i)\;\cap\;
          \operatorname{cone}(w_j \mid j\ne i).
        \]
  \item For every extremal ray $R$ of $\tau_i$, solve
        \[
          w_D \;=\; x\,w_i + y\,R.
        \]
  \item Set $\mu_i$ to the \emph{smallest non-negative} $x$ among all
        solutions.
  \item Define $w_N := \sum \mu_i\,w_i$ and $w_P := w_D - w_N$.
\end{enumerate}

\bigskip
\noindent\textbf{How to use}

\begin{lstlisting}[language=Magma,basicstyle=\ttfamily\small,
                   frame=single,captionpos=b]
load "library.m";

L := ToricLattice(2);
W := [ L![1,0], L![0,1], L![0,1], L![1,3] ];
wD := L![2,5];

wP, wN, mu := ZariskiDecomposition(wD, W);

print "Positive part P =", wP;
print "Negative part N =", wN;
print "Vector of coefficients mu =", mu;
\end{lstlisting}

{\texttt{cox-rings/library.m}}
{\texttt{library.m}}
\begin{lstlisting}[caption={ \texttt{library.m} - shared helper routines }]
// ZariskiDecomposition
// Input:
//   - wD: an element of a rational vector space, representing the class of a divisor D
//   - W: a list [w1, ..., wr] of vectors in the same space, representing the degrees of the generators f_i of Cox(X)
//
// Output:
//   - wP: the class of the positive part P of the Zariski decomposition of D
//   - wN: the class of the negative part N = sum mu_i w_i
//   - mu: the list of rational numbers [mu_1,..., mu_r] such that D = P + sum mu_i D_i
//
// The algorithm computes mu_i by finding the minimal non-negative x such that 
// wD = x * w_i + y * R for some extremal ray R of the cone tau_i = cone(wD, -w_i) meet cone(w_j : j not i)

ZariskiDecomposition := function(wD, W)
    r := #W;
    mu := [Rationals()!0 : i in [1..r]];

    for i in [1..r] do
        cone1 := Cone([wD, -W[i]]);
        cone2 := Cone([W[j] : j in [1..r] | j ne i]);
        tau := cone1 meet cone2;
        rays := Rays(tau);

        candidates := [];

        for R in rays do
            M := Matrix(Rationals(), [Eltseq(W[i]), Eltseq(R)]);
            success, sol := IsConsistent(M, Vector(wD));
            if success then
                x := sol[1];
                if x ge 0 then
                    Append(~candidates, x);
                end if;
            end if;
        end for;

        mu[i] := Minimum(candidates);
    end for;

    wN := &+[mu[i]*W[i] : i in [1..r]];
    wP := wD - wN;

    return wP, wN, mu;
end function;
\end{lstlisting}


\begin{thebibliography}{33}




\bibitem{AbbanCheltsovKishimotoMangolte1}
H.~Abban, I.~Cheltsov, T.~Kishimoto, F.~Mangolte, \emph{K-stability of pointless del Pezzo surfaces and Fano 3-folds}, preprint, arXiv:2411.00767, 2024.

\bibitem{AbbanCheltsovKishimotoMangolte2}
H.~Abban, I.~Cheltsov, T.~Kishimoto, F.~Mangolte, \emph{K-stability of Fano 3-folds in the World of Null-A}, preprint, arXiv:2505.04330, 2025.

\bibitem{AbbanZhuang}
H.~Abban, Z.~Zhuang, \emph{K-stability of Fano varieties via admissible flags}, Forum of Mathematics Pi \textbf{10} (2022), 1--43.

\bibitem{AbbanZhuangSeshadri}
H.~Abban, Z.~Zhuang, \emph{Seshadri constants and K-stability of Fano manifolds}, Duke Mathematical Journal  \textbf{172} (2023) 1109--1144.

\bibitem{ABHLX20}
J.~Alper, H.~Blum, D.~Halpern-Leistner, Ch.~Xu, \emph{Reductivity of the automorphism group of {$K$}-polystable {F}ano varieties}, Invent. Math. \textbf{222} (2020), 995--1032.

\bibitem{AndPig}
M.~Andreatta, R.~Pignatelli, \emph{Fano's last Fano}, Rendiconti Acc. Naz. Lincei \textbf{34} (2023), 359--381.

\bibitem{Book}
C.~Araujo, A.-M.~Castravet, I.~Cheltsov, K.~Fujita, A.-S.~Kaloghiros, J.~Martinez-Garcia, C.~Shramov, H.~S\"u\ss, N.~Viswanathan, \emph{The Calabi problem for Fano 3-folds},
Lecture Notes in Mathematics, Cambridge University Press, \textbf{485} (2023).

\bibitem{ArtebaniLaface}
M. Artebani, A. Laface, \emph{Hypersurfaces in {M}ori dream spaces}, J. Algebra \textbf{371} (2012), 26--37.

\bibitem{adhl}
I.~Arzhantsev, U.~Derenthal, J.~Hausen, A.~Laface, \emph{Cox rings}, Cambridge Studies in Advanced Mathematics, Vol.~144,  Cambridge University Press, 2015.

\bibitem{Avilov}
A.~Avilov, \emph{Automorphisms of threefolds that can be represented as an intersection of two quadrics}, Sbornik: Mathematics \textbf{207} (2016), 315--330.

\bibitem{Au78}
Th.~Aubin, \emph{\'Equations du type {M}onge-{A}mp\`ere sur les vari\'et\'es k\"ahl\'eriennes compactes}, Bull. Sci. Math. \textbf{102} (1978), 63--95.

\bibitem{Bauer}
Th.~Bauer, \emph{A simple proof for the existence of Zariski decompositions on surfaces},  J. Algebraic Geom. \textbf{18} (2009), 789--793.

\bibitem{BKS04}
Th.~Bauer, A.~K\"{u}ronya, T.~Szemberg, \emph{Zariski chambers, volumes, and stable base loci}, Journal f\"ur die reine und angewandte Mathematik \textbf{576} (2004), 209--233.

\bibitem{BelousovLoginov}
G.~Belousov, K.~Loginov, \emph{K-stability of Fano threefolds of rank $4$ and degree $24$}, European Journal of Mathematics  \textbf{9} (2023), no. 3, Paper No. 80.

\bibitem{BBJ21}
R.~Berman, S.~Boucksom, M.~Jonsson, \emph{A variational approach to the {Y}au-{T}ian-{D}onaldson conjecture}, J. Amer. Math. Soc. \textbf{34} (2021), 605--652.

\bibitem{BCHM}
C.~Birkar, P.~Cascini, Paolo, C.~D.~Hacon, J.~McKernan, \emph{Existence of minimal models for varieties of log general type},
 {Journal of the American Mathematical Society} \textbf{23} (2010), no. 2, 405--468.

\bibitem{BHLLX21}
H.~Blum, D.~Halpern-Leistner, Y.~Liu, and C.~Xu, \emph{On properness of K-moduli spaces and optimal degenerations of Fano varieties}, Selecta Math. \textbf{27} (2021), article number 73.

\bibitem{BJ20}
H.~Blum, M.~Jonsson, \emph{Thresholds, valuations, and K-stability}, Advances in Mathematics \textbf{365} (2020), article ID 107062.

\bibitem{BL22}
H.~Blum, Y.~Liu, \emph{Openness of uniform K-stability in families of $\mathbb{Q}$-Fano varieties}, Annales Scientifiques de l'Ecole Normale Superieure \textbf{55} (2022), 1--41.

\bibitem{BLX19}
H.~Blum, Y.~Liu, and C.~Xu, \emph{Openness of K-semistability for Fano varieties}, Duke Math. J. \textbf{171} (2022), 2753--2797.

\bibitem{BX19}
H.~Blum, Ch.Xu, \emph{Uniqueness of {K}-polystable degenerations of {F}ano varieties}, Ann. of Math. \textbf{190} (2019), 609--656.

\bibitem{Bo97}
J.~P.~Bourguignon, \emph{M\'etriques d'{E}instein-{K}\"ahler sur les vari\'et\'es de {F}ano: obstructions et existence}, Ast\'erisque \textbf{245} (1997), 277--305, S\'eminaire Bourbaki, Exp. No. 830.

\bibitem{CasciniLazic}
P.~Cascini, V.~Lazi{\'c}, \emph{New outlook on the minimal model program. {I}.}, Duke Mathematical Journal \textbf{161} (2012), no. 12., 2415--2467.

\bibitem{Cheltsov}
I.~Cheltsov, \textit{Log canonical thresholds on hypersurfaces}, Sbornik: Mathematics \textbf{192} (2001), 1241--1257.

\bibitem{CheltsovGAFA}
I.~Cheltsov, \textit{Log canonical thresholds of del Pezzo surfaces}, Geometric and Functional Analysis \textbf{11} (2008), 1118--1144.

\bibitem{CheltsovDenisovaFujita}
I.~Cheltsov, E.~Denisova, K.~Fujita, \emph{K-stable smooth Fano threefolds of Picard rank two}, Forum Math. Sigma \textbf{12} (2024), Paper No. e41.

\bibitem{CheltsovFujitaKishimotoOkada}
I.~Cheltsov, K.~Fujita, T.~Kishimoto, T.~Okada, \emph{K-stable divisors in $\mathbb{P}^1\times\mathbb{P}^1\times\mathbb{P}^2$ of degree $(1,1,2)$}, Nagoya Mathematical Journal \textbf{251} (2023), 686--714.

\bibitem{CP}
I.~Cheltsov, J.~Park, \emph{Total log-canonical thresholds and generalized Eckardt points}, Sbornik: Mathematics \textbf{193} (2002), 779--789.

\bibitem{CheltsovPark}
I.~Cheltsov, J.~Park, \emph{K-stable Fano threefolds of rank $2$ and degree $30$}, European Journal of Mathematics \textbf{8} (2022), no. 3, 834-852.

\bibitem{CPS}
I.~Cheltsov, V.~Przyjalkowski, C.~Shramov, \emph{Fano threefolds with infinite automorphism groups}, Izvestiya: Mathematics \textbf{83} (2019), 860--907.

\bibitem{CS}
I.~Cheltsov, C.~Shramov, \emph{Log canonical thresholds of smooth Fano threefolds}, Russian Mathematical Surveys \textbf{63} (2008), 859--958.

\bibitem{CDS15}
X.~Chen, S.~Donaldson, S.~Sun, \emph{K\"ahler-{E}instein metrics on {F}ano manifolds. {I, II, III}}, J. Amer. Math. Soc. \textbf{28} (2015), 183--278.

\bibitem{CP21}
G.~Codogni, Z.~Patakfalvi, \emph{Positivity of the CM line bundle for families of K-stable klt Fano varieties}, Invent. Math. \textbf{223} (2021), 811--894.

\bibitem{Cutkosky}
S.~Cutkosky, \emph{Zariski decomposition of divisors on algebraic varieties}, Duke Math. J. \textbf{53}, 149--156 (1986).

\bibitem{DeBiaseFatighentiTanturri}
L. De Biase, E.~Fatighenti, F.~Tanturri, \emph{Fano 3-folds from homogeneous vector bundles over Grassmannians}, Revista Matemetica Complutense \textbf{35} (2022), 649--710.

\bibitem{dFKL}
T.~de~Fernex, A.~K\"uronya, R.~Lazarsfeld, \emph{Higher cohomology of divisors on a projective variety}, Math. Ann. \textbf{337} (2007), no. 2., 443--445.

\bibitem{Delcroix2022}
T.~Delcroix, \emph{Examples of K-unstable Fano manifolds}, Annales de l'Institut Fourier \textbf{72} (2022), 2079--2108.

\bibitem{Denisova}
E.~Denisova, \emph{On K-stability of $\mathbb{P}^3$ blown up along the~disjoint union of a~twisted cubic curve and~a~line}, preprint, arXiv:2202.04421, 2022.

\bibitem{dhhk}
U.~Derenthal, J.~Hausen, A.~Heim, S.~Keicher, A.~Laface, \emph{Cox rings of cubic surfaces and {F}ano threefolds}, J. Algebra \textbf{436} (2015), 228--276.

\bibitem{DolgachevIskovskikh}
I.~Dolgachev, V.~Iskovskikh, \emph{Finite subgroups of the~plane Cremona group}, Progress in Mathematics \textbf{269} (2009), 443--548.

\bibitem{do}
I.~Dolgachev, \emph{Classical algebraic geometry. A modern view}, Cambridge University Press, 2012.

\bibitem{Do02}
S.~Donaldson, \emph{Scalar curvature and stability of toric varieties}, J. Differential Geom. \textbf{62} (2002), 289--349.

\bibitem{Do15}
S.~Donaldson, \emph{Algebraic families of constant scalar curvature K\"ahler metrics}, Surveys in Differential Geometry \textbf{19} (2015), 111--137.

\bibitem{EGZ09}
Ph.~Eyssidieux, V.~Guedj, A.Zeriahi, \emph{Singular {K}\"ahler--{E}instein metrics}, J. Amer. Math. Soc. \textbf{22} (2009), 607--639.

\bibitem{Es16}
Ph.~Eyssidieux, \emph{M\'etriques de {K}\"ahler-{E}instein sur les vari\'et\'es de {F}ano [d'apr\`es {C}hen-{D}onaldson-{S}un et {T}ian]}, Ast\'erisque (2016), no.~380, Exp. No. 1095, 207--229.

\bibitem{Fano}
G.~Fano, \emph{Su una particolare varieta algebrica a tre dimensioni aventi curve sezioni canoniche}, Rendiconti Acc. Naz. Lincei \textbf{6} (1949), 151--156.

\bibitem{fanography}
P.~Belmans, \emph{Fanography}, \url{https://fanography.info}, 2025.

\bibitem{Fuj90}
T.~Fujita, \emph{On singular Del Pezzo varieties}, In Algebraic Geometry. Lecture Notes in Mathematics, vol 1417. Springer, Berlin, Heidelberg, (1990) 117--128.

\bibitem{Fujita2016}
K.~Fujita, \emph{On K-stability and the~volume functions of $\mathbb{Q}$-Fano varieties}, Proceedings of the~LMS \textbf{113} (2016), 541--582.

\bibitem{Fu19}
K.~Fujita, \emph{A valuative criterion for uniform K-stability of $\mathbb{Q}$-Fano varieties}, Journal f\"ur die Reine und Angewandte Mathematik \textbf{751} (2019), 309--338.

\bibitem{Fujita2021}
K.~Fujita, \emph{On K-stability for Fano threefolds of rank $3$ and degree $28$}, International Mathematics Research Notices, IMRN 2023, no. 15, 12601-12784.

\bibitem{FO18}
K.~Fujita, Y.~Odaka, \emph{On the {K}-stability of {F}ano varieties and anticanonical divisors}, Tohoku Math. J. \textbf{70} (2018), 511--521.

\bibitem{hlm}
J.~Hausen, A.~Laface, Ch.~Mauz, \emph{On smooth Fano fourfolds of Picard number two}, Revista Matematica Iberoamericana \textbf{38} (2022), 53--93.

\bibitem{Jia20}
C.~Jiang, \emph{Boundedness of $\mathbb{Q}$-Fano varieties with degrees and alpha-invariants bounded from below}, Ann. Sci. Ec. Norm. Super. (4) \textbf{53} (2020), 1235--1248.

\bibitem{KKL12}
A.-S.~Kaloghiros, A.~K\"uronya, V.~Lazi\'c, \emph{Finite generation and geography of models}, Adv. Stud. Pure Math. \textbf{70} (2016), 215--245.

\bibitem{KLPZ26}
A.-S.~Kaloghiros, Y.~Liu, A.~Petracci, J~.Zhao, \emph{The boundary of K-moduli of prime Fano threefolds of genus twelve}, preprint, available at arXiv:2603.29827.



\bibitem{KMM}
Yu.~Kawamata, K.~Matsuda, K.~Matsuki, \emph{Introduction to the minimal model problem}, Algebraic geometry, Proc. Symp., Sendai/Jap. 1985, Adv. Stud. Pure Math. 10, 283-360 (1987).

\bibitem{KM76}
F.~Knudsen, D.~Mumford, \emph{The projectivity of the moduli space of stable curves. I:  Preliminaries on ``det'' and ``Div''}, Math. Scand. \textbf{39}, 19-55 (1976).

\bibitem{KM98}
J.~Koll\'{a}r, Sh.~Mori, \emph{Birational geometry of algebraic varieties}, Cambridge Tracts in Mathematics, vol. 134, Cambridge University Press, Cambridge, 1998, With the collaboration of C. H. Clemens and A. Corti, Translated from the 1998 Japanese original.

\bibitem{Ko97}
J.~Koll\'ar, \emph{Singularities of pairs}, Algebraic geometry---{S}anta {C}ruz 1995, Proc. Sympos. Pure Math., vol. 62, Part 1, Amer. Math. Soc., Providence, RI, 1997, pp.~221--287.

\bibitem{Ko13}
J.~Koll\'{a}r, \emph{Singularities of the minimal model program}, Cambridge Tracts in Mathematics, vol. 200, Cambridge University Press, Cambridge, 2013, With a collaboration of S\'{a}ndor Kov\'{a}cs.

\bibitem{Ko-modbook}
J.~Koll\'{a}r, \emph{Families of varieties of general type}, Cambridge Tracts in Mathematics, vol. 231, Cambridge University Press, Cambridge, 2023, With the collaboration of Klaus Altmann and S\'{a}ndor J. Kov\'{a}cs.

\bibitem{Kuronya_ACF}
A.~K\"uronya, \emph{Asymptotic cohomological functions on projective varieties}, {Amer. J. Math.} \textbf{128} (2006), no. 6., 1475--1519.

\bibitem{GeomAsp}
A.~K\"{u}ronya, V.~Lozovanu, \emph{Geometric aspects of Newton-Okounkov bodies}, Banach Center Publications \textbf{116} (2018), 137--212.

\bibitem{KLM_volume}
 A.~K\"uronya, V.~Lozovanu, C.~Maclean, \emph{Volume functions of linear series},  Mathematische Annalen, \textbf{356} (2013),  635--652.

\bibitem{KuznetsovProkhorovShramov}
A.~Kuznetsov, Yu.~Prokhorov, C.~Shramov, \emph{Hilbert schemes of lines and conics and automorphism groups of Fano threefolds}, Japanese Journal of Mathematics \textbf{13} (2018), 109--185.

\bibitem{PAGI}
R.~Lazarsfeld, \emph{Positivity in algebraic geometry. I}, Ergebnisse der Mathematik und ihrer Grenzgebiete. 3. Folge. \textbf{48}, Springer Verlag, Berlin, 2004.

\bibitem{LM}
R.~Lazarsfeld, M.~Musta{\c{t}}{\u{a}}, \emph{Convex bodies associated to linear series}, {Ann. Sci. \'Ec. Norm. Sup\'er.} \textbf{42} (2009),  783--835.

\bibitem{Li17}
C.~Li, \emph{K-semistability is equivariant volume minimization}, Duke Mathematical Journal \textbf{166} (2017), 3147--3218.

\bibitem{LWX21}
C.~Li, X.~Wang, C.~Xu, \emph{Algebraicity of the metric tangent cones and equivariant K-stability}, J. Amer. Math. Soc. \textbf{34} (2021), 1175--1214.

\bibitem{LX14}
C.~Li, Ch.~Xu, \emph{Special test configuration and {K}-stability of {F}ano varieties}, Ann. of Math. (2) \textbf{180} (2014), no.~1, 197--232.

\bibitem{Liu18}
Y.~Liu, \emph{The volume of singular {K}\"ahler-{E}instein {F}ano varieties}, Compos. Math. \textbf{154} (2018), no.~6, 1131--1158.

\bibitem{Liu22}
Y.~Liu, \emph{K-stability of cubic fourfolds}, J. Reine Angew. Math. \textbf{786} (2022), 55-77.

\bibitem{Liu}
Y.~Liu, \emph{K-stability of Fano threefolds of rank $2$ and degree $14$ as double covers}, Mathematische Zeitschrift \textbf{303} (2023), no. 2, Paper No. 38.

\bibitem{LX19}
Y.~Liu, C.~Xu, \emph{K-stability of cubic threefolds},
Duke Math. J. \textbf{168} (2019), no.~11, 2029--2073.

\bibitem{LXZ22}
Y.~Liu, C.~Xu, Z.~Zhuang, \emph{Finite generation for valuations computing stability thresholds and applications to K-stability}, Annals of Mathematics \textbf{196} (2022), 507--566.

\bibitem{LZ25}
Y.~Liu, J.~Zhao, \emph{K-moduli of Fano threefolds and genus four curves}, J. Reine Angew. Math. 2025. \url{https://doi.org/10.1515/crelle-2025-0016}

\bibitem{Ma57}
Y.~Matsushima, \emph{Sur la structure du groupe d'hom\'eomorphismes analytiques d'une certaine vari\'et\'e{} k\"ahl\'erienne}, Nagoya Math. J. \textbf{11} (1957), 145--150.

\bibitem{Mukai1989}
S.~Mukai, \emph{Biregular classification of Fano 3-folds and Fano manifolds of coindex $3$}, PProceedings of the~National Academy of Sciences of the~USA \textbf{86} (1989), 3000--3002.

\bibitem{Nakayama_ZD}
N.~Nakayama, \emph{Zariski-decomposition and abundance}, MSJ Memoirs vol. 14, Mathematical Society of Japan, Tokyo, 2004.

\bibitem{Odaka2013}
Y.~Odaka, \emph{On the~moduli of K\"ahler--Einstein Fano manifolds}, Proceeding of Kinosaki Symposium (2013), 112--126.



\bibitem{Od13klt}
Y.~Odaka, \emph{The {GIT} stability of polarized varieties via discrepancy}, Ann. of Math. \textbf{177} (2013), no.~2, 645--661.

\bibitem{Od13open}
Y.~Odaka, \emph{On the moduli of {K}{\"a}hler-{E}instein {F}ano manifolds}, Proc. Kinosaki symposium, 2013, pp.~112--126.

\bibitem{OdakaSano}
Y.~Odaka, Y.~Sano, \emph{Alpha invariant and K-stability of $\mathbb{Q}$-Fano varieties}, Advances in Mathematics \textbf{229} (2012), 2818--2834.

\bibitem{OSS16}
Y.~Odaka, C.~Spotti and S.~Sun, \emph{Compact moduli spaces of del {P}ezzo surfaces and {K}\"ahler-{E}instein metrics}, J. Differential Geom. \textbf{102} (2016), no.~1, 127--172.

\bibitem{Prokhorov}
Yu.~G.~Prokhorov, \emph{On the {Zariski} decomposition problem}, in: {Birational geometry: Linear systems and finitely generated algebras. Collected papers. Transl. from the Russian}, Moskva: Maik Nauka/Interperiodika, 2003, 37--65.

\bibitem{SS17}
C.~Spotti, S.~Sun, \emph{Explicit {G}romov-{H}ausdorff compactifications of moduli spaces of {K}\"ahler-{E}instein {F}ano manifolds}, Pure Appl. Math. Q. \textbf{13} (2017), no.~3, 477--515.

\bibitem{Gabor18}
G.~Szekelyhidi, \emph{K\"ahler--Einstein metrics}, Modern geometry: a celebration of the work of {S}imon {D}onaldson, Proc. Sympos. Pure Math. \textbf{99} (2018), 331--361.

\bibitem{Ti87}
G.~Tian, \emph{On K\"ahler--Einstein metrics on certain K\"ahler manifolds with $c_1(M)>0$}, Inventiones Mathematicae \textbf{89} (1987), 225--246.

\bibitem{Ti97}
G.~Tian, \emph{K\"ahler--Einstein metrics with positive scalar curvature}, Inventiones Mathematicae \textbf{130} (1997), 1--37.

\bibitem{Ti99}
G.~Tian, \emph{K\"ahler-{E}instein manifolds of positive scalar curvature}, Surveys in differential geometry: essays on {E}instein manifolds, Surv. Differ. Geom., vol.~6, Int. Press, Boston, MA, 1999, pp.~67--82.

\bibitem{Ti15}
G.~Tian, \emph{K-stability and {K}\"ahler-{E}instein metrics}, Comm. Pure Appl. Math. \textbf{68} (2015), 1085--1156.

\bibitem{Ti15corr}
G.~Tian, \emph{Corrigendum: {K}-stability and {K}\"ahler-{E}instein metrics}, Comm. Pure Appl. Math. \textbf{68} (2015), 2082--2083.

\bibitem{Wisniewski1990}
J.~Wisniewski, \emph{Fano 4-folds of index 2 with $b_2\geqslant 2$. A contribution to Mukai classification}, Bulletin of the~Polish Academy of Sciences, Mathematics \textbf{38}, (1990), 173--184.

\bibitem{Xu21}
C.~Xu, \emph{K-stability of Fano varieties: an algebro-geometric approach}, The EMS Surveys in Mathematical Sciences \textbf{8} (2021), 265--354.

\bibitem{Xubook}
C.~Xu, \emph{K-stability of Fano varieties}, Princeton University, New Jersey, 2025.

\bibitem{XuLiu}
C.~Xu, Y. Liu, \emph{K-stability of cubic threefolds}, Duke Mathematical Journal \textbf{168} (2019), 2029--2073.

\bibitem{XZ20}
C.~Xu, Z.~Zhuang, \emph{On positivity of the CM line bundle on K-moduli spaces}, Ann. of Math. (2) \textbf{192} (2020), 1005--1068.

\bibitem{XZ21}
C.~Xu, Z.~Zhuang, \emph{Uniqueness of the minimizer of the normalized volume function}, Camb. J. Math. \textbf{9} (2021), 149--176.

\bibitem{Ya78}
Sh.~T.~Yau, \emph{On the {R}icci curvature of a compact {K}\"ahler manifold and the complex {M}onge-{A}mp\`ere equation. {I}}, Comm. Pure Appl. Math. \textbf{31} (1978), 339--411.

\bibitem{Zar}
O.~Zariski, \emph{The theorem of Riemann--Roch for high multiples of an effective divisor on an algebraic surface}, Ann. of Math. \textbf{76} (1962), 560--615.

\bibitem{Zha24}
J.~Zhao, \emph{K-stability of Thaddeus' moduli of stable bundle pairs on genus two curves}, Forum Math. Sigma, to appear.


\bibitem{Zh20}
Z.~Zhuang, \emph{Product theorem for {K}-stability}, Adv. Math. \textbf{371} (2020), 107250, 18.

\bibitem{Zhuang21}
Z.~Zhuang, \emph{Optimal destabilizing centers and equivariant K-stability}, Inventiones Mathematicae \textbf{226} (2021), 195--223.




\end{thebibliography}
\end{document}